\documentclass[11pt,a4paper,english]{smfart}

\usepackage[utf8]{inputenc}
\usepackage[T1]{fontenc}
\usepackage[english,french]{babel}

\usepackage{amscd}
\usepackage{amssymb,smfthm} 
\usepackage{mathtools}
\usepackage{bm}

\usepackage{newtxtext}
\usepackage{newtxmath}
\usepackage{textcomp}
\usepackage{euscript}
\usepackage{manfnt}
\usepackage{fix-cm}

\usepackage[dvipsnames]{xcolor} 
\usepackage{graphicx} 

\usepackage[
textwidth=150mm,
textheight=245mm,
centering
]{geometry}
\usepackage{enumitem}
\usepackage{comment}
\usepackage{marginnote}
\usepackage{interval}
\usepackage{faktor}
\usepackage{thmtools}  

\usepackage{tikz-cd}   
\usepackage{tikz}
\usetikzlibrary{arrows.meta}

\usepackage[pagebackref]{hyperref}
\hypersetup{
	colorlinks=true,
	linkcolor=NavyBlue,
	citecolor=OliveGreen,
	urlcolor=BrickRed
}
\usepackage{cleveref}

\theoremstyle{plain}
\newtheorem{thmx}{Theorem}
\renewcommand{\thethmx}{\Alph{thmx}} 
\newtheorem{thm}{Theorem}[section]  
\newtheorem{lem}[thm]{Lemma}

\newtheorem{claim}[thm]{Claim} 
\newtheorem{proposition}[thm]{Proposition}
\newtheorem{cor}[thm]{Corollary} 
\newtheorem{conjecture}[thm]{Conjecture}
\newtheorem{question}[thm]{Question}

\newtheorem{corx}[thmx]{Corollary} 
\theoremstyle{definition}
\newtheorem{dfn}[thm]{Definition}
\newtheorem{deflem}[thm]{Definition-Lemma}

\theoremstyle{remark}
\newtheorem{rem}[thm]{Remark}

\theoremstyle{plain}
\newlist{thmlist}{enumerate}{1}
\setlist[thmlist]{wide = 0pt, labelwidth = 2em, labelsep*=0em, itemindent = 0pt, leftmargin = \dimexpr\labelwidth + \labelsep\relax, noitemsep,topsep = 1ex, font=\normalfont, label=(\roman*), ref=\thethm.(\roman{thmlisti})}

\addtotheorempostheadhook[thm]{\crefalias{thmlisti}{thm}}

\addtotheorempostheadhook[assumpsion]{\crefalias{thmlisti}{assumption}}

\addtotheorempostheadhook[question]{\crefalias{thmlisti}{question}}

\addtotheorempostheadhook[deflem]{\crefalias{thmlisti}{deflem}}
\addtotheorempostheadhook[cor]{\crefalias{thmlisti}{cor}}

\addtotheorempostheadhook[proposition]{\crefalias{thmlisti}{proposition}}

\addtotheorempostheadhook[dfn]{\crefalias{thmlisti}{dfn}}

\addtotheorempostheadhook[lem]{\crefalias{thmlisti}{lem}}
\addtotheorempostheadhook[main]{\crefalias{thmlisti}{main}}

\addtotheorempostheadhook[rem]{\crefalias{thmlisti}{rem}}

\newlist{thmenum}{enumerate}{1} 
\setlist[thmenum]{wide = 0pt, labelwidth = 2em, labelsep*=0em, itemindent = 0pt, leftmargin = \dimexpr\labelwidth + \labelsep\relax, noitemsep,topsep = 1ex, font=\normalfont, label=(\roman*), ref=\thethmx.(\roman{thmenumi})}
\crefalias{thmenumi}{thmx} 

\newlist{corlist}{enumerate}{1} 
\setlist[corlist]{wide = 0pt, labelwidth = 2em, labelsep*=0em, itemindent = 0pt, leftmargin = \dimexpr\labelwidth + \labelsep\relax, noitemsep,topsep = 1ex, font=\normalfont, label=(\roman*), ref=\thecorx.(\roman{corlisti})}
\crefalias{corlisti}{corx} 

\crefname{lem}{Lemma}{Lemmas} 
\crefname{conjecture}{Conjecture}{Conjectures}
\crefname{thm}{Theorem}{Theorems}
\crefname{proposition}{Proposition}{Propositions}
\crefname{dfn}{Definition}{Definitions}
\crefname{rem}{Remark}{Remarks}
\crefname{cor}{Corollary}{Corollaries}
\crefname{corx}{Corollary}{Corollaries}
\crefname{problem}{Problem}{Problems}
\crefname{thmx}{Theorem}{Theorems}
\crefname{claim}{Claim}{Claims}
\crefname{assumption}{Assumption}{Assumptions}
\crefname{main}{Main Theorem}{Main Theorems}
\crefname{setting}{Setting}{Settings}

\numberwithin{equation}{thm} 



\newcommand{\ep}{\varepsilon}

\renewcommand{\Im}{\operatorname{Im}}

\DeclareMathOperator{\Res}{Res}

\DeclareMathOperator{\rank}{rank}
\DeclareMathOperator{\obj}{Obj}    

\newcommand{\cT}{\mathcal{T}}

\makeatletter
\newcommand*{\rom}[1]{\expandafter\@slowromancap\romannumeral #1@}
\makeatother

\makeatletter     
\newcommand{\crefnames}[3]{%
	\@for\next:=#1\do{%
		\expandafter\crefname\expandafter{\next}{#2}{#3}%
	}%
}
\makeatother

\crefnames{part,chapter,section}{\S}{\S\S}

\newcommand{\sC}{\mathscr{C}}
\newcommand{\sD}{\mathscr{D}}

\newcommand{\sF}{\mathscr{F}}

\newcommand{\cH}{\mathcal{H}}

\newcommand{\sQ}{\mathscr{Q}}
\newcommand{\sR}{\mathscr{R}}

\newcommand{\cA}{\mathcal A}

\newcommand{\cE}{\mathcal E}
\newcommand{\cI}{\mathcal I}

\newcommand{\cO}{\mathcal O}
\newcommand{\hcO}{\widehat{\mathcal O}}

\newcommand{\cV}{\mathcal V}

\newcommand{\fV}{\mathbf V} 

\newcommand{\kn}{\mathfrak n} 
 
\newcommand{\km}{\mathfrak m} 

\newcommand{\bB}{\mathbb{B}}
\newcommand{\bC}{\mathbb{C}}
\newcommand{\bD}{\mathbb{D}}
\newcommand{\bE}{\mathbb{E}}

\newcommand{\bN}{\mathbb{N}}

\newcommand{\bQ}{\mathbb{Q}}

\newcommand{\bV}{\mathbb{V}}

\newcommand{\bZ}{\mathbb{Z}}

\newcommand{\kg}{\mathfrak{g}}

\newcommand{\kl}{\mathfrak{l}}

\def\db{\bar{\partial}}

  \def\spec{\textrm{Spec}\,}
 \def\d{\partial} 
\def\hess{\nabla''\nabla'}

\def\sn{\sqrt{-1}}

\DeclareMathOperator{\End}{End}
\DeclareMathOperator{\Def}{Def}
\DeclareMathOperator{\Art}{Art}
\DeclareMathOperator{\gr}{Gr}
\DeclareMathOperator{\Grp}{Grp}

\def\bnabla{\boldsymbol{\nabla}}
\def\bdb{\bm{\bar{\partial}}}

\def\bd{\bm{\partial}}
\def\bhess{\boldsymbol{\nabla}''\boldsymbol{\nabla}'}

\DeclareMathOperator{\Hom}{Hom}

\DeclareMathOperator{\GL}{GL}

\begin{document} 
	\title[Two-step nilpotent monodromy]{Two-step nilpotent monodromy of local systems on special varieties}
	\author[J. Cao]{Junyan Cao}
	\address{Laboratoire de Mathématiques J.A. Dieudonné UMR 7351 CNRS, Université Côte d'Azur Parc Valrose 06108, Nice, France} 
	\email{junyan.cao@unice.fr}
	\urladdr{https://sites.google.com/site/junyancao} 
	
	\author[Y. Deng]{Ya Deng}

	\address{CNRS,  
		Institut de Math\'ematiques de Jussieu-Paris Rive Gauche,
		Sorbonne Universit\'e, Campus Pierre et Marie Curie,
		4 place Jussieu, 75252 Paris Cedex 05, France}
	
	\email{ya.deng@math.cnrs.fr, deng@imj-prg.fr}
	\urladdr{https://ydeng.perso.math.cnrs.fr}

	\author[C. Hacon]{Christopher D. Hacon}
	\address{Department of Mathematics, University of Utah, Salt Lake City, UT 84112,
		USA}
	\email{hacon@math.utah.edu}
	\urladdr{https://www.math.utah.edu/~hacon/}

	\author[M. Paun]{Mihai Paun}
	\address{Universit\"at Bayreuth, Mathematisches Institut, Lehrstuhl Mathematik VIII, Universit\"atsstrasse 30,
		D-95447, Bayreuth, Germany}
	\email{mihai.paun@uni-bayreuth.de}
	\urladdr{https://www.mathe8.uni-bayreuth.de/de/team/prof-paun/index.php}
	\begin{abstract} 
		We  address a question of Campana by giving a criterion for when nilpotent local systems on quasi-projective varieties have 2-step nilpotent monodromy.  As an application, we prove that the monodromy group of any complex local system on   smooth complex quasi-projective variety that is special in the sense of Campana or $h$-special (e.g. containg a Zariski dense entire curve), is virtually 2-step nilpotent. This result sharply refines a theorem of Cadorel, Yamanoi, and the second author.
		
		To establish this result, we develop a deformation theory for certain local systems on quasi-compact K\"ahler manifolds by constructing universal deformations for such local systems.    
	\end{abstract}
	
	\maketitle
	\tableofcontents
	\section{Introduction}
	\subsection{Main theorems}

	The  motivation for this paper is   the following fundamental question by Campana
	concerning fundamental groups of complex algebraic varieties.
	\begin{question}\label{Q1}
		Let $X$ be a smooth complex quasi-projective variety.
		Assume that its fundamental group $\pi_1(X)$ is nilpotent.
		Is $\pi_1(X)$ necessarily virtually $2$-step nilpotent?
	\end{question} 
	For the definition of ($2$-step) nilpotent group, we refer the readers to \cref{def:nilpotent}.  

	Our first main result gives a positive answer to \Cref{Q1} when the quasi-Albanese map of $X$ is $\pi_1$-exact in the sense of \cref{def:pi1exact}.  
	\begin{thmx}[=\cref{thm:main}]\label{main:2step}
		Let $X$ be a smooth quasi-projective variety and  let $G$ be a connected nilpotent linear algebraic group over $\bC$. Assume that  the quasi-Albanese map $a:X\to A$ of $X$, is a dominant morphism with connected general fibers $F$ such that one has the following short exact sequence
		$$
		\pi_1(F)\to \pi_1(X)\to \pi_1(A)\to 1. 
		$$ 
		Assume that there exists a Zariski dense representation $\varrho:\pi_1(X)\to G(\bC)$. Then   the following short exact sequence
		$$
		1\to G'\to G\to G/G'\to 1
		$$
		is a \emph{central extension} of two commutative algebraic group, where $G':=[G,G]$.  In particular, the image $\varrho(\pi_1(X))$ is    2-step nilpotent. 
	\end{thmx} 
	We recall some background on \Cref{Q1} and summarize the progress made so far.   \Cref{Q1} was originally asked by Campana \cite[Remarque 1.5]{Cam95}
	for compact K\"ahler manifolds, and only recently raised in its present form for quasi-projective varieties in his joint work with Aguilar \cite[Question~25]{AC25}.  Very recently, there have been some interesting works regarding \Cref{Q1} by Shimoji \cite{Taito25}, who provided a positive answer when the rank of $\pi_1(X)$ is at most seven. In \cite{Rogov25}, Rogov proposed a strategy to study \Cref{Q1} using o-minimality, combined with Hain's higher Albanese maps \cite{Hai87} and the unipotent variations of mixed Hodge structures of Hain–Zucker \cite{HZ87}.  
	In the present paper, our approach to proving \cref{main:2step} is essentially self-contained and does not rely on any of these preceding strategies or tools. A key novelty of our work is the establishment of a \emph{$2$-step phenomenon in non-abelian Hodge theory} (see \cref{thm:strong}). 
	
	As an application of \cref{main}, we prove the following theorem on the fundamental groups of special or $h$-special quasi-projective varieties, addressing another question of Campana \cite{Cam04,Cam11}. This is our second main result.
	\begin{thmx}\label{main}
		Let $X$ be a smooth complex quasi-projective variety. Suppose that $X$ is special in the sense of Campana (e.g. if its logarithmic Kodaira dimension $\bar{\kappa}(X)=0$) or $h$-special in the sense of \cite{CDY25b} (e.g. $X$ contains a Zariski dense entire curve). Then for any linear representation $\varrho \colon \pi_1(X)\to \GL_N(\mathbb{C})$, the image $\varrho(\pi_1(X))$ is virtually nilpotent of class at most $2$.
	\end{thmx}
	
	Previously, in \cite{CDY25b}, Cadorel, Yamanoi, and the second author proved that for any $\varrho$ as in \cref{main}, the image $\varrho(\pi_1(X))$ is virtually nilpotent. Moreover, they constructed examples of special and $h$-special varieties whose fundamental groups are $2$-step nilpotent but not virtually abelian.
	Later, in \cite{DY23b}, Yamanoi and the second author proved that for any representation $\tau:\pi_1(X)\to \GL_{N}(K)$ where $K$ is a field of positive characteristic, the image $\tau(\pi_1(X))$ is virtually abelian. With these results in place, \cref{main} completes the picture for the monodromies of local systems of arbitrary characteristic over special or $h$-special quasi-projective varieties.

	\subsection{Results in Hodge theories}
	
	Let us stress here that,  to prove \cref{main:2step}, we develop a deformation theory for certain local systems on quasi-compact K\"ahler manifolds $X$. Inspired by the celebrated work of Eyssidieux--Simpson \cite{ES11}, we  construct a \emph{canonical and  universal} connection  for extendable unitary local systems on   $X$.   Since the precise statements are somewhat technical, we only state a special case here and refer the reader to \cref{main1,thm:universal2} for the general formulations. 
	
	\begin{thmx}[$\subset$\cref{main1,thm:universal2,thm:final}]\label{main3}
		Let $\overline{X}$ be a compact K\"ahler manifold and let $D=\sum_{i=1}^{m}D_i$ be a simple normal crossing divisor on $\overline{X}$. Set $X:=\overline{X}\setminus D$ and fix a base point $x\in X$. Let $(V,\nabla)$ be a trivial flat bundle of rank $N$ on $\overline{X}$ and denote the associated endomorphism bundle by $(E,\nabla):=(\End(V),\nabla)$.
		
		Let $Z$ be the affine subvariety of the affine space underlying the complex vector space  $H^1(\overline{X},E)\times H^0(D^{(1)},E)$ defined by the following quasi-homogeneous equations for   $(\alpha,\beta)\in H^1(\overline{X},E)\times H^0(D^{(1)},E)$ of type $(1,2)$:
		\begin{align*}
			g(\beta) + \frac{1}{2}[\alpha,\alpha] &= 0 \in H^2(\overline{X},E), \\
			[\iota^* \alpha, \beta] &= 0 \in H^1(D^{(1)},E), \\
			\bigl[ \iota_{ij}^* (\beta|_{D_i}), \iota_{ji} ^* (\beta|_{D_j}) \bigr] &= 0 \in H^0(D_i \cap D_j, E) \quad \text{for any } 1 \leq i < j \leq m,
		\end{align*}
		where $D^{(1)}:=\bigsqcup_{i=1}^{m}D_i$, and $\iota \colon D^{(1)}\to \overline{X}$ is the natural morphism. Here, $g \colon H^0(D^{(1)},E)\to H^2(X,E)$ denotes the Gysin morphism (cf. \cref{def:Gysin}), and $\iota_{ij} \colon D_i\cap D_j\to D_i$ is the inclusion map.
		
		Then there exists an open ball $B \subset H^1(\overline{X},E)\times H^0(D^{(1)},E)$ centered at the origin $O$ such that, over the analytic open subset $\mathcal{T} := Z \cap B$, there exists a \emph{canonical} and \emph{analytic} family of logarithmic connections $\nabla_{(\alpha,\beta)}$ on $V$, parametrized by $(\alpha,\beta) \in \mathcal{T}$. These connections are flat over $X_0$ and satisfy the following property: 
		The natural map defined by the monodromy representations of the flat connections $(V|_X,\nabla_{(\alpha,\beta)})$
		\begin{align*}
			\mathcal{T} &\to \Hom(\pi_1(X,x), \GL_{N}(\mathbb{C})) \\
			(\alpha,\beta) &\mapsto \operatorname{Mon}(\nabla_{(\alpha,\beta)})
		\end{align*}
		is a biholomorphism onto an open neighborhood of the trivial representation in the representation variety
		$\Hom(\pi_1(X,x), \GL_{N}(\mathbb{C}))$. 
		
		Furthermore, we can regard $\nabla_{(\alpha,\beta)}$ as defining a representation
		$$
		\varrho_{\cT} \colon \pi_1(X) \to \GL_N(\cO_{\cT,O}),
		$$
		where $\cO_{\cT,O}$ is the local ring of $\cT$ at the origin $O$. Consider the following representation
		$$
		\varrho_{\cT,3}:	\pi_1(X) \xrightarrow{\varrho_{\cT}} \GL_N(\cO_{\cT,O}) \to \GL_N(\cO_{\cT,O}/\km^3),
		$$
		where $\km$ denotes the maximal ideal of $\cO_{\cT,O}$. Then $\varrho_{\cT,3}(\pi_1(X))$ is 2-step nilpotent, and satisfies the following property:    let   $f \colon Y \to X$ be a compactifiable holomorphic map from a quasi-compact K\"ahler manifold $Y$, such that $f^*\varrho_{\cT,3}$ is a trivial representation. Then 
		for any linear representation $\tau:\pi_1(X)\to\GL_{N}(\bC)$ that lies in the same geometric irreducible component of the trivial representation,  $ f^*\tau$ is a trivial representation as well. 
	\end{thmx}
	We refer the reader to \cref{main1} for the precise formula for $\nabla_{(\alpha,\beta)}$. Notably, when $X$ is compact, the above quasi-homogeneous equations reduce to the homogeneous quadratic equation $[\alpha,\alpha]=0$. \cref{main3} provides an analytic viewpoint and a concrete construction of the equivalence of analytic germs between the representation variety $\Hom(\pi_1(X,x), \GL_{N})$ and the quadratic cone in $H^1(X,E)$ defined by $[\alpha,\alpha]=0$, a result originally established by Goldman and Millson \cite{GM88} and Eyssidieux and Simpson \cite{ES11}. The compact case of \cref{main3} is covered by the more general framework in \cref{thm:ES}. We refer the reader to \cref{main1} for the precise formula for $\nabla_{(\alpha,\beta)}$.  
	Subsequently, we establish the \emph{1-step phenomenon} for local systems on \emph{compact K\"ahler} manifolds in \cref{thm:1-step}. Together, these phenomena provide   a conceptually unifying framework that allows us to reprove several known results in a uniform way:
	\begin{enumerate}[label=(\arabic*)]
		\item the theorems of Katzarkov \cite{Kat95} and Leroy \cite{Ler} establishing the Shafarevich conjecture for compact K\"ahler manifolds with nilpotent fundamental groups (see \cref{thm:Kat});
		\item more generally, the theorems of Green--Griffiths--Katzarkov \cite{GGK24} and Aguilar--Campana \cite{AC25} on the holomorphic convexity of the universal covers of quasi-compact K\"ahler manifolds with nilpotent fundamental groups and \emph{proper} quasi-Albanese maps (see \cref{thm:GGK});
		\item Campana's theorem \cite{Cam04} on the virtual abelianity of the monodromy of local systems over special smooth projective varieties (see \cref{thm:Cam}); and
		\item a structure theorem for the linear Shafarevich morphism (see \cref{thm:structure}), established implicitly in the work of Eyssidieux, Katzarkov, Pantev, and Ramachandran on the linear Shafarevich conjecture \cite{EKPR12}, and later applied by Wang and the second author \cite{DW24b} in their proof of the linear Koll\'ar conjecture on the positivity of the holomorphic Euler characteristic for varieties with big fundamental groups.
	\end{enumerate}

	In the course of proving \cref{main3}, we derive several ancillary results, including the construction of explicit zigzags of \emph{$1$-quasi-isomorphisms} between the differential graded Lie algebras  (dglas) associated with $E$-valued logarithmic complexes and the cohomology complexes. We present this construction here due to its independent interest.
	\begin{thmx}[=\cref{thm:zigzag}]\label{main:zigzag}
		Let $\overline{X}$ be a compact K\"ahler $n$-fold and let $D=\sum_{i=1}^{m}D_i$ be a simple normal crossing divisor on $\overline{X}$. Let $(V,\nabla)$ be a unitary flat vector bundle   on $\overline{X}$ and denote the associated endomorphism bundle by $(E,\nabla):=(\End(V),\nabla)$. Then for the following maps
		\begin{equation*}
			\begin{tikzcd}
				A^0(\overline{X},D,E)\arrow[r,"\nabla"] & \cdots \arrow[r,"\nabla"]&  A^{2n}(\overline{X},D,E)  \\
				A^0(\overline{X},D,E)\cap \ker\nabla' \arrow[r,"\nabla''"]\arrow[d,"\psi"]\arrow[u,"\phi"]  & \cdots  \arrow[u,"\phi"]\arrow[d,"\psi"]\arrow[r,"\nabla''"] & A^{2n}(\overline{X},D,E)\cap\ker\nabla'\arrow[d,"\psi"]\arrow[u,"\phi"]    \\
				H^0(D^{(0)},E)\arrow[r,"d"] &\cdots \arrow[r,"d"] & \oplus_{p+q=2n}   H^{q}(D^{(p)},E)  
			\end{tikzcd}
		\end{equation*} 
		the vertical morphisms $\phi$ and $\psi$ are both $1$-equivalences (see \cref{def:1equiv}) connecting the three horizontal dgla complexes.
		Here:
		\begin{itemize}
			\item $\phi$ denotes the inclusion map;
			\item $D^{(k)}$ denotes the disjoint union of all $k$-fold intersections
			$$
			D_I := D_{i_1} \cap \cdots \cap D_{i_k},
			$$
			where $I=(i_1,\ldots,i_k)$ is a multi-index satisfying $1 \le i_1 < \cdots < i_k \le m$, with the convention $D^{(0)}:=\overline{X}$;
			\item the differential $d$ in the bottom chain complex is given by the Gysin morphism  with suitable signs  (see \eqref{eq:defdiff}), and the Lie brackets are defined in \eqref{eq:Lie};
			\item for any $\eta\in A^k(\overline{X},D,E)$, the map $\psi(\eta)$ is defined by
			$$
			\psi(\eta):=\left( (-1)^{|I|}\{\cH(\Res_{D_I}\eta)\}\right)_{I}\in H^{k-|I|}(D_I,E),
			$$
			where $I$ ranges over all ordered multi-indices contained in $\{1,\ldots,m\}$, and $\cH$ denotes the \emph{harmonic projection} of the residue  $\Res_{D_I}\eta\in A^{k-|I|}(D_I, \sum_{j\notin I}D_j\cap D_I, E)$.
		\end{itemize}
	\end{thmx}
	In particular, this refines previous results by Morgan \cite{Mor78}, Pridham \cite{Pri16}, and Lef\`evre \cite{Lef19},    while avoiding the use of either Sullivan’s minimal models or $L_\infty$-algebras. To the best of our knowledge, this provides the first explicit construction of zigzags connecting the dglas of $E$-valued logarithmic complexes to their  cohomology complexes, establishing a 1-quasi-equivalence.

	%

	The logical dependencies among the main theorems are summarized in \Cref{fig:dependencies}.
	
	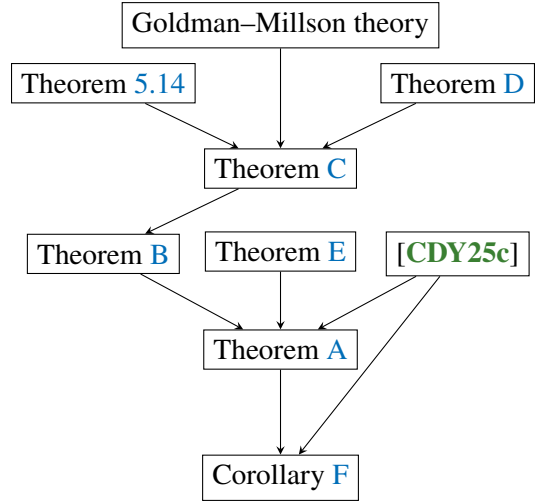
\begin{figure}[htbp]
		\centering
		\begin{tikzpicture}[
			box/.style={
				rectangle,
				rounded corners=2pt,
				draw=black!65,
				fill=black!2,
				line width=0.6pt,
				text width=3.1cm,
				minimum height=1cm,
				inner sep=4pt,
				align=center,
				font=\small
			},
			arrow/.style={
				-{Latex[length=2mm, width=1.4mm]},
				draw=black!75,
				line width=0.7pt,
				shorten >=2pt,
				shorten <=2pt
			}
			] 
			\node[box] (H) at (-4, 1.8)
			{\cref{thm:construction}};
			
			\node[box] (G) at (0, 1.8)
			{Goldman--Millson\\theory};
			
			\node[box] (F) at (4, 1.8)
			{\cref{main:zigzag}};
			
			\node[box] (C) at (0, 0)
			{\cref{main3}};
			
			\node[box] (D) at (-4, -0.9)
			{\cref{thm:strong}};
			
			\node[box] (B) at (0, -1.8)
			{\cref{main:2step}};
			\node[box] (E) at ( 4, -1.8)
			{\cite{CDY25b}};
			
			\node[box] (A) at (0, -3.6)
			{\cref{main}};
			
			\path[arrow]
			(G) edge (C)
			(F) edge (C)
			(H) edge (C)
			(H) edge (D)
			(C) edge (D)
			(C) edge (B)
			(D) edge (B)
			(B) edge (A)
			(E) edge (A);
		\end{tikzpicture}
		\caption{Logical dependencies between the main theorems}
		\label{fig:dependencies}
	\end{figure}

	\subsection{Proof of \cref{main}} 
	In this subsection, we give the proof of \cref{main} based on \cref{main:2step} and \cite{CDY25b}.
	
	By \cite[Proposition 10.11]{Cam11} (resp. \cite[Lemma 3.2]{CDY25b}), finite \'etale cover of special (resp. $h$-special) varieties are still special (resp. $h$-special), and thus we are free to replace $X$ by a finite \'etale cover.  
	By \cite{CDY25b} (see \cref{thm:nilpotent} below), the image $\varrho(\pi_1(X))$    is virtually nilpotent. Hence, after we replace $X$ by a finite \'etale cover, we may assume that  the Zariski closure  $G$ of $\varrho(\pi_1(X))$, is a connected nilpotent algebraic group $G$ over $\bC$.  
	
	By \cite{CDY25b} (see \cref{thm:pi1}), the quasi-Albanese map $a:X\to A$ is dominant with connected general fibers such that we have a short exact sequence
	\begin{align}\label{eq:pi_1first}
		\pi_1(F)\to \pi_1(X)\to \pi_1(A)\to 1,
	\end{align}
	where $F$ is a general fiber of $a$.    Conditions in  \cref{main:2step} are fulfilled. We thus conclude that the image   $\varrho(\pi_1(X))$ is  2-step nilpotent. The proof of \cref{main} is completed.

	\section*{Notations and conventions}
	Throughout this paper, unless otherwise stated, we assume that $(X,\omega)$ is a compact K\"ahler manifold of dimension $n$ and let $D = \sum_{i=1}^m D_i$ be a simple normal crossing divisor on $X$. We set $X_0:= X \setminus D$.
	
	Throughout this paper, we adopt the convention that any multi-index $I = (i_1, \ldots, i_k)$ consists of \emph{distinct entries} in $\{1,\ldots,m\}$. 
	We say that $I$ is \emph{ordered} if $i_1 < \cdots < i_k $.  For any $z=(z_1,\ldots,z_n)$, denote by $z_I:=z_{i_1}\cdots z_{i_{k}}$.  For such an index, we define the intersection
	\[
	D_I := D_{i_1} \cap \cdots \cap D_{i_k}.
	\]
	For any integer $k \geq 1$, let $D^{(k)}$ denote the disjoint union of all intersections indexed by ordered multi-indices of length $k$:
	\[
	D^{(k)} := \bigsqcup_{\substack{|I|=k \\ I \text{ ordered}}} D_I.
	\]
	In particular, $D^{(1)} = \bigsqcup_{i=1}^m D_i$.
	\begin{itemize}
		\item A complex manifold $X$ is called \emph{quasi-compact K\"ahler} if it is a dense Zariski open subset of a compact K\"ahler manifold $\overline{X}$, which we refer to as a \emph{compactification} of $X$. A holomorphic map between two quasi-compact K\"ahler manifolds is called \emph{compactifiable} if it extends to a meromorphic map between their compactifications.
		\item  Let $(V,\nabla)$ be a flat bundle of rank $N$ on $X$. Unless otherwise mentioned, we denote the associated endomorphism bundle by $(E,\nabla):=(\End(V),\nabla)$, and we write $H^k(X,E)$ for the de Rham cohomology of $E$ with respect to $\nabla$.  Throughout this paper, except in \cref{sec:tour}, we assume that the flat bundle $(V,\nabla)$ is unitary.
		\item For any unitary flat bundle $(V,\nabla)$ on $X$, we write
		$
		\nabla = \nabla' + \nabla''
		$
		for the decomposition of $\nabla$ into its $(1,0)$- and $(0,1)$-parts.  We endow $E=\End(V)$ with the holomorphic vector bundle structure induced by the $(0,1)$-component $\nabla''$. 
		\item  A local system $\mathcal{L}$ on $X\backslash D$ is called \emph{extendable} if it extends to a local system on $ {X}$. 
		\item Let $\cA^\bullet(X,D,E)$ be the sheaf of $E$-valued smooth differential forms
		with at most logarithmic poles along $D$. We write $A^k(X,E)$ for the space of smooth
		$E$-valued $k$-forms, $A^k(X,D,E)$ for those with at most logarithmic poles along $D$,
		$\sD^k(X,E)$ for the space of $E$-valued currents of degree $k$,  
		$A^k_{\mathrm{cts}}(X,E)$ for the space of continuous $E$-valued differential forms of degree $k$, and 
		$\cH^k(X,E)$ for the space of $E$-valued harmonic $k$-forms.
		\item For any $\eta \in A^k(X,D,V)$, we denote by $\bnabla \eta$ (resp.\ $\bnabla' \eta$ and $\bnabla'' \eta$, noting the bold font) its covariant derivative in the sense of currents. In general, $\bnabla'' \eta$ may contain nontrivial residue terms supported on $D$, whereas   $\bnabla'\eta = \nabla'\eta$.
		\item Any element   $\eta\in A^k(X,D,E)$ is called $\bnabla'$-exact if $\eta=\bnabla'\sigma$ for some $\sigma\in \sD^{k-1}(X,E)$. 
		\item For any closed form $\eta \in \sD^k(X,E)$ with $\bnabla \eta = 0$, we let $\cH(\eta)$ denote its harmonic projection, which is the unique harmonic representative of the class $\{\eta\} \in H^k(X,E)$. By abuse of notation, if $\alpha \in H^k(X,E)$ is a cohomology class, we also write $\cH(\alpha)$ for its unique harmonic representative.
		\item Let $W_\bullet \cA^k(X,D,E)\subset \cA^k(X,D,E)$ be the increasing filtration induced by the number
		of pure logarithmic factors, and let $W_\bullet A^k(X,D,E)$ denote its space of global
		sections on $X$. 
		\item Let $\cA^\bullet_{X,D}(E)$ be the kernel of the residue morphism
		\[
		\Res_D \colon W_1 \cA^\bullet(X,D,E) \longrightarrow
		\cA^{\bullet-1}(D^{(1)}, E).
		\] 
		\item Fix a base point $x\in X_0$. We denote the topological fundamental group of $X_0$ by $\pi_1(X_0,x)$. By abuse of notation, we usually omit the base point when the context is clear. Furthermore, we let $\kg := E_x$ denote the fiber of $E$ at $x$, which carries a natural Lie algebra structure.  
		\item  Let $R(X_0,N) \coloneqq \operatorname{Hom}(\pi_1(X_0), \operatorname{GL}_N(\mathbb{C}))$ denote the variety of representations of the fundamental group of $X_0$ into $\operatorname{GL}_N(\mathbb{C})$.
		\item  A group $G$ is said to be virtually $P$ if it contains a subgroup of finite index
		that has property $P$.  We denote by $G':=[G,G]$. 
		\item For brevity, we write dgla (resp. dglas) for differential graded Lie algebra(s).
		\item We denote by $\Art$ the category of Artin local $\bC$-algebras, and $\Grp$  the category of groupoids. For any Artin local $\bC$-algebra $A$, denote by $\km$ its maximal ideal. 
		\item For any $\ep>0$, let  $\bB_\ep^{p+q}$  denote the ball of radius $\ep$ centered at the origin in  $\bC^{p+q}$.  
	\end{itemize}
	\section{Technical Preliminaries}
	\subsection{Special varieties and orbifolds}
	Special varieties are introduced by Campana \cite{Cam04,Cam11} in his  remarkable program of classification of geometric orbifolds.  We collect some facts about orbifolds from \cite{Cam04,Cam11}. 
	
	Recall that a (geometric) orbifold $(X, D)$ is a pair consisting of a variety $X$ and a $\mathbb{Q}$-divisor $D = \sum c_i P_i$, where the $P_i$ are prime divisors and the coefficients $0 \leq c_i \leq 1$ are of the form $c_i = 1 - \frac{1}{m_i}$ for some $m_i \in \mathbb{Z}_{\geq 1} \cup \{+\infty\}$. We will say that $(X,D)$ is {\it log smooth} (or just {\it smooth}) if $X$ is non-singular and the support of $D$ has simple normal crossings. The multiplicities of $D$ along $P_i$ are $m_i=m_{P_i}(D)=\frac 1{1-c_i}$ so that
	\[D=\sum c_iP_i=\sum (1-\frac 1{m_i})P_i.\]
	Note that $m_i\in [1,+\infty]$ and $m_i=+\infty$ coincides with the case $c_i=1$ whilst $m_i=1$ when $c_i=0$.
	\begin{dfn}[Multiplicity and Orbifold base]
		\label{def:multiplicitydivisor}
		Let $f: X \to Y$ be a surjective morphism of normal varieties with connected fibers. Let $P$ be a prime divisor on $X$. We say $P$ is \emph{$f$-exceptional} if $\operatorname{codim}_Y f(P) \ge 2$. If $P$ dominates $Y$, it is called \emph{horizontal}; otherwise, it is \emph{vertical}. If $P$ is an $f$-vertical divisor that is not $f$-exceptional, then its image $f(P) = Q$ is a prime divisor on $Y$. Restricting to the regular loci $X^0 \subset X$ and $Y^0 \subset Y$, we can write the pullback of $Q$ as$$(f^0)^*Q = m(f,P)P + F$$where $F$ is an effective divisor whose support does not contain $P$. We define the integer $m(f,P)$ to be the \emph{multiplicity of $f$ along $P$}.
		
		Suppose now that $Y$ is smooth, and $(X,D)$ is a smooth orbifold, where the boundary divisor is written as $D = \sum (1 - \frac{1}{m_P(D)}) P$ with multiplicities $m_P(D) \in \mathbb{Z}_{\ge 1} \cup \{\infty\}$.   We define the \emph{orbifold base} (or \emph{orbifold divisor of $f$}) as the $\mathbb{Q}$-divisor on $Y$ given by (cf.\ \cite[D\'efinition A.3]{CC16}):$$ \Delta(f,D) = \sum_Q \left( 1 - \frac{1}{m_f(Q)} \right) Q$$where the sum runs over all prime divisors $Q \subset Y$, and the \emph{multiplicity} of $ \Delta(f,D)$ along $Q$ is defined as:
		$$m_Q\left( \Delta(f,D)\right)=m_f(Q):= \inf_{f(P)=Q} \big\{ m(f, P) \cdot m_P(D) \big\}.$$
	\end{dfn}

	We therefore introduce the following definition (cf. \cite[D\'efinitions 5.10 \& 5.16]{Cam11}):
	\begin{dfn}[Canonical dimension]\label{def:canonical dimension}
		Let $f: (Y, D) \to X$ be a fibration where both the orbifold pair $(Y, D)$ and the base $X$ are smooth. We define the \emph{canonical dimension} of the orbifold fibration $f$ as:
		$$\kappa(f, D) := \inf_{f' \sim f} \{\kappa(X', \Delta(f', D'))\}$$
		where $f': (Y', D') \to X'$ ranges over all models that are birationally equivalent to $f$. By definition, $\kappa(f, D)$ is a birational invariant of the fibration. If $\kappa(f, D) = \dim(X)$, we say that $f$ is a \emph{fibration of general type}.
	\end{dfn}

	\begin{dfn}[Special variety]\label{def:special}
		Let $(X,D)$ be a smooth projective orbifold, then   $(X, D)$ is special if it does not admit any fibration of general type in the sense of \cref{def:canonical dimension}.
	\end{dfn}
	
	
	\subsection{$h$-special varieties}
	
	Campana defined $X$ to be \emph{$H$-special} if $X$ has vanishing Kobayashi pseudo-distance.
	Motivated by \cite[Definition 9.1]{Cam04},  in \cite{CDY25b}, Cadorel, Yamanoi and the second named author introduced the following definition.
	
	\begin{dfn}[$h$-special]\label{defn:20230407}
		Let $X$ be a  quasi-projective variety.
		We define the equivalence relation $x\sim y$ of two points $x,y\in X$ if and only if there exists a sequence of holomorphic maps $f_1,\ldots,f_l:\mathbb C\to X$ such that letting $Z_i\subset X$ to be the Zariski closure of $f_i(\mathbb C)$, we have 
		$$x\in Z_1, Z_1\cap Z_2\not=\emptyset, \ldots, Z_{l-1}\cap Z_l\not=\emptyset, y\in Z_l.$$
		We set $R=\{ (x,y)\in X\times X; x\sim y\}$.
		We define $X$ to be \emph{hyperbolically special} ($h$-special for short) if and only if  $R\subset X\times X$ is Zariski dense.   
	\end{dfn}
	By definition, varieties with Zariski dense entire curves, and rationally connected projective varieties are all $h$-special.

	Motivated by Campana's suggestion \cite[11.3 (5)]{Cam11}, it was asked in \cite{CDY25b}   whether a smooth quasi-projective variety  is $h$-special  if and only if it is special. 
	
	\subsection{Nilpotentcy conjecture} 
	We recall the following definition of nilpotent group. 
	\begin{dfn}[Nilpotent group]\label{def:nilpotent}
		A group $G$ is \emph{nilpotent} if it possesses a central series of finite length. This is defined as a series of normal subgroups$$\{1\} = G_0 \trianglelefteq G_1 \trianglelefteq \cdots \trianglelefteq G_n = G$$such that $G_{i+1}/G_i \leq Z(G/G_i)$ for all $0 \leq i < n$. For a nilpotent group $G$, the minimal integer $n$ for which a central series of length $n$ exists is called the \emph{nilpotency class} of $G$. Such a group is said to be \emph{nilpotent of class $n$}.   We say that $G$ is   \emph{$n$-step nilpotent}  if it is  nilpotent of class  less than or equal to $n$. In particular, 1-step nilpotent groups are abelian. 
	\end{dfn}

	In \cite{CDY25b}, the authors conjectured that smooth special quasi-projective varieties have virtually nilpotent fundamental groups, and prove the following:	
	\begin{thm}[{\cite[Theorem A]{CDY25b}}]\label{thm:nilpotent}
		Let $X$ be a smooth quasi-projective variety that is special or $h$-special, and let $\varrho:\pi_1(X)\to \GL_N(\bC)$ be a representation. Then after replacing $X$ by a finite \'etale cover,  $\varrho(\pi_1(X))$ is a nilpotent group.   \qed
	\end{thm} 
	In this paper, we propose a refined version of the above nilpotency conjecture.
	\begin{conjecture}[2-step nilpotency conjecture]\label{conj:nilpotent}
		Let $X$ be a smooth special quasi-projective variety. Then $\pi_1(X)$ is virtually nilpotent  of class at most $2$. 
	\end{conjecture}
	We will need the following result on the \emph{$\pi_1$-exactness} of quasi-Albanese maps of special varieties in the proof of \cref{main}.
	\begin{proposition}[{\cite[Proposition 4.13]{CDY25b}}]\label{thm:pi1}
		Let $X$ be a smooth quasi-projective variety that is special or $h$-special. Then its quasi-Albanese map $a:X\to A$ is $\pi_1$-exact, i.e., $a$ is a dominant morphism with connected general fibers, such that we have the following short exact sequence
		$$\pi_1(F,x)\to \pi_1(X,x)\to \pi_1(A,a(x))\to 0, $$
		where  $F$ is a general smooth fiber, and $x\in F$ is a base point.   \qed
	\end{proposition} 
	
	\subsection{Hodge decomposition for currents and $\d\db$-lemma}\label{sec:Hodge}
	Let $X$ be a compact Kähler $n$-fold, and let $(E,\nabla)$ be a semisimple flat bundle on $X$. By Simpson \cite{Sim92}, there exists a harmonic metric $h$ on $E$. This metric induces a decomposition $\nabla=\nabla'+\nabla''$ such that the following identities hold:$$\Delta=2\Delta'=2\Delta''.$$ Here, the total Laplacian is defined as $\Delta:=\nabla^* \nabla+\nabla\nabla^* $, and the partial Laplacians $\Delta'$ and $\Delta''$ are defined analogously using $\nabla'$ and $\nabla''$ respectively. Note that the splitting $\nabla=\nabla'+\nabla''$ generally differs from the decomposition of $\nabla$ into its standard $(1,0)$- and $(0,1)$-parts, unless $(E,\nabla)$ is unitary.

	\begin{dfn}
		We denote by $\sD^k(X,E)$ the space of $E$-valued currents of degree $k$ on $X$.  It is  the topological dual of the space of smooth global forms of degree $2n-k$ with values in the dual bundle $E^*$:
		$$
		\sD^k(X,E) := \big(  {A}^{2n-k}(X,E^*) \big)'.
		$$
		The space of test forms $ {A}^{2n-k}(X,E^*)$ is equipped with the natural Fr\'echet topology, i.e. the topology of uniform convergence of all derivatives.
	\end{dfn}
	Note that if $D$ is a simple normal crossing divisor on $X$, then we have the inclusion
	$$
	A^k(X,D,E)\subset \sD^k(X,E).
	$$
	Let $G$ denote the Green operator associated with $\Delta'$. We can extend the operators
	$\bnabla', \bnabla'', \Delta'$, and $G$ to act on the space of $E$-valued currents of degree $k$. Since $X$ is compact K\"ahler, we have the following identities:
	\begin{align}\label{eq:basicgreen}
		[\bnabla',G] = 0,\quad  & [\bnabla'',G] = 0, \\
		[\bnabla'',(\bnabla')^*] = 0,\quad  & [\bnabla',(\bnabla'')^*] = 0, \nonumber \\
		\alpha = \cH \alpha + \Delta' G\alpha, & \nonumber
	\end{align}
	where $\cH$ is the harmonic projection.
	\begin{thm}[De Rham-Kodaira decomposition for currents]\label{currentdecomp}
		Given a current $T\in \sD^k(X, E)$, we have the following Hodge decomposition
		\begin{equation}\label{end17} 
			T= \cH(T)+ \Delta'   G(T),
		\end{equation}
		where $G$ is the Green operator with respect to the Laplacian $\Delta' :=[\bnabla', \bnabla'^*]$.
	\end{thm}
	
	\begin{proof}
		For the proof, we refer to \cite{Kod-deRham} or \cite[Section 3]{CP}. 
	\end{proof}
	
	\begin{lem}[$\d\db$-lemma]\label{lem:ddbar} Let $u\in \sD^k(X, E)$ such that $\bnabla u= 0$ as currents on $X$.
		If $u= \bnabla' v$, or $u=\bnabla''v$ on $X$ for some current $v\in \sD^{k-1}(X,E)$, then  we have
		$$
		u=\bnabla'\bnabla''\sigma
		$$
		for some $\sigma\in \sD^{k-2}(X,E)$. In particular, if $k=1$, then $u=0$.
	\end{lem}\begin{proof}
		Since $u$ is $\bnabla'$-exact or $\bnabla''$-exact and
		$\bnabla u=0$, we have
		$$
		\bnabla'u=\bnabla''u=0,
		\qquad \cH(u)=0.
		$$
		By the Hodge decomposition and the commutation relations,
		we obtain
		$$
		u=\bnabla'\bnabla'^*Gu
		=\bnabla''\bnabla''^*Gu.
		$$
		Set
		$$
		\sigma:=-\bnabla'^*G\bnabla''^*Gu.
		$$
		Then $\sigma$ is a current of degree $k-2$. Since
		$\bnabla''\bnabla'^*=-\bnabla'^*\bnabla''$ and
		$\bnabla''G=G\bnabla''$, we have
		$$
		\begin{aligned}
			\bnabla'\bnabla''\sigma
			&=\bnabla'\bnabla'^*G
			\bnabla''\bnabla''^*Gu\\
			&=\bnabla'\bnabla'^*Gu=u.
		\end{aligned}
		$$
		If $k=1$, then $\sigma=0$ for degree reasons, and hence
		$u=0$. The lemma is thus proved.
	\end{proof}
	
	\section{A quick tour of Goldman-Millson theory}\label{sec:GM}
	In this section, we briefly recall some essential ingredients of (Deligne–)Goldman–Millson theory used in the proof of \cref{main:2step,main3}.
	\subsection{Transformation groupoids}
	Recall that a (transformation) \emph{groupoid} is a small category in which all morphisms are isomorphisms. Let $({X}, {G})$ be a transformation groupoid and let $Y$ be a set upon which $G$ acts. Then the transformation groupoid $({X} \times {Y}, {G})$, where $G$ acts on ${X} \times {Y}$ by the diagonal action is a new groupoid which we denote by $({X}, {G}) \bowtie {Y}$. If $\varphi: ({X}', {G}' ) \to({X}, {G})$ is a morphism of transformation groupoids and Y is a $G$-set, then there is a corresponding morphism of transformation groupoids
	$$
	\varphi \bowtie {Y}: ({X}', {G}') \bowtie {Y} \to({X}, {G}) \bowtie {Y}
	$$
	where the ${G}'$-action on $Y$  is induced from the $G$-action on $Y$ by the homomorphism $\varphi: {G}' \to {G}$.
	\begin{lem}[{\cite[Lemma 3.8]{GM88}}]\label{lem:GMequiv}
		If $\varphi: ({X}', {G}' ) \to({X}, {G})$ is an equivalence of groupoids, then $\varphi \bowtie {Y}: ({X}', {G}' ) \bowtie {Y} \to({X}, {G}) \bowtie {Y}$ is also an equivalence of groupoids. \qed
	\end{lem}  
	
	Let us recall the definition of equivalence of categories \cite[p. 52]{GM88}. If $\mathscr{A}$ and $\mathscr{B}$ are categories, then a functor $F: \mathscr{A} \rightarrow \mathscr{B}$ is an \emph{equivalence} if it satisfies three basic properties:
	\begin{enumerate}[label=(\alph*)]
		\item\label{a1}  $F$ is surjective on isomorphism classes, i.e. the induced map $F_*: \mathrm{Iso}\, \mathscr{A} \rightarrow \mathrm{Iso}\, \mathscr{B}$ is surjective;
		\item \label{a2} $F$ is full, i.e. for any two objects $x, y \in \mathrm{Obj}\, \mathscr{A}$, the map
		$$
		F(x, y): \operatorname{Hom}(x, y) \rightarrow \operatorname{Hom}(F(x),  {F}(y))
		$$
		is surjective;
		\item \label{a3} $F$ is faithful, i.e. for any two objects $x, y \in \mathrm{Obj}\, \mathscr{A}$, the map
		$$
		F(x, y): \operatorname{Hom}(x, y) \rightarrow \operatorname{Hom}(F(x), F(y))
		$$
		is injective.
	\end{enumerate}
	In particular, an equivalence of categories $F: \mathscr{A} \rightarrow \mathscr{B}$ induces an bijection of sets
	$$
	F_*: \text { Iso } \mathscr{A} \rightarrow \text { Iso } \mathscr{B}. 
	$$
	\subsection{DGLA}
	Let $\kg$ be a Lie algebra. A $\kg$-augmented differential graded Lie algebra is a triple $(L^\bullet, d, \varepsilon)$ where $(L^\bullet, d)$ is a differential graded Lie algebra (dgla for short) and $\ep: L^0 \rightarrow \kg$ is a homomorphism of Lie algebras. A homomorphism of $\kg$-augmented differential graded Lie algebras $(L^\bullet, d, \varepsilon) \rightarrow(\overline{L^\bullet}, \bar{d}, \bar{\varepsilon})$ is a differential graded Lie algebra homomorphism $\varphi:(L^\bullet, d) \rightarrow(\overline{L^\bullet}, \bar{d})$ such that $\varepsilon \circ \varphi=\bar{\varepsilon}$. If $(L^\bullet, d, \varepsilon)$  is a $\kg$-augmented differential graded Lie algebra, then the augmentation extends trivially to a differential graded Lie algebra homomorphism $\varepsilon: L^\bullet \rightarrow \mathrm{g}$ where $\kg$ is given the dgla structure with no nonzero elements of positive degree. 

	\begin{dfn}[1-equivalence]\label{def:1equiv}
		A morphism of dglas
		$f: (L^\bullet_0, d_0) \to (L^\bullet_1, d_1)$ is called \emph{1-equivalent} 
		if the induced map on cohomology  is an isomorphism in 
		degrees 0 and 1, and a monomorphism in degree 2.
	\end{dfn}
	
	\begin{dfn}[(1-)quasi-isomorphism]\label{def:1qis}
		Two $\kg$-augmented differential graded Lie algebras 
		$(L^\bullet, d, \varepsilon)$ and $(\overline{L^\bullet}, \bar{d}, \bar{\varepsilon})$ 
		are called \emph{quasi-isomorphic} (resp. \emph{1-quasi-isomorphic}) 
		if there exists a zigzag of $\kg$-augmented dgla homomorphisms
		\[
		L^\bullet = L_0 \longrightarrow L_1 \longleftarrow L_2 \longrightarrow \dots 
		\longleftarrow L_{m-1} \longrightarrow L_m = \overline{L^\bullet}
		\]
		such that each homomorphism induces an isomorphism in cohomology 
		(resp. is 1-equivalent).
	\end{dfn}
	
	A $\kg$-augmented differential graded Lie algebra is formal if it is quasi-isomorphic to its cohomology. Note that in this paper, the dgla constructed to study the deformation of local systems on quasi-compact Kähler manifolds is, in general, not formal. This makes the present work significantly more involved than  \cite{GM88,ES11,EKPR12}, particularly regarding the analytic aspects of Hodge theory; consequently, we must establish new techniques in Hodge theory to study these dglas.
	\subsection{DGLA and deformation functor}
	Given a  dgla $C^\bullet$, there is a functor  from  Artin local $\bC$-algebras to groupoids 
	\begin{align*}
		{\rm Def} (C^\bullet): \Art  \rightarrow  {\rm Grp} 
	\end{align*}
	as follows. Let $A$ be an Artin local $\bC$-algebra and $\km$ be its maximal ideal. 
	Then the objects of ${\rm Def} (C^\bullet,A)$ are defined by
	$$
	{\rm Obj }\ {\rm Def} (C^\bullet,A)=\sC(C^\bullet,A):=\{\alpha\in C^1\otimes\km\mid d\alpha+\frac{1}{2}[\alpha,\alpha]=0\},
	$$
	and  the morphisms are 
	$$
	\operatorname{Hom}(\alpha,\beta) =\exp(C^0\otimes \km):= \{\lambda\in C^0\otimes\km\mid \exp(\lambda)\alpha=\beta\}
	$$
	where $\exp(\lambda)\alpha$ denotes the gauge transformation $e^\lambda\circ \alpha\circ e^{-\lambda}-d(e^\lambda)\circ e^{-\lambda}$.
	We often write $[\sC(C^\bullet,A)/\exp(C^0\otimes \km)]$ for this transformation groupoids. 

	We shall encounter augmented ones.  Let $\mathfrak{g}$ be a Lie algebra over $\mathbb{C}$, viewed as a dgla in degree 0, and let $\ep: L^{\bullet} \rightarrow \mathfrak{g}$ be an augmentation of the dgla $L^{\bullet}$. Let  $(A, \km)$ be an Artin local $\mathbb{C}$-algebra. Then one defines the small groupoid $\Def (L^{\bullet}, \ep, A )$ by: 
	\begin{align}\label{eq:defaug}
		\obj \Def\left(L^{\bullet}, \ep, A\right) & =\sC(L^{\bullet}, \ep, A):=\left\{\left(\alpha, e^r\right) \in L^1 \otimes \km \times \exp (\mathfrak{g} \otimes \km) \left\lvert\, d \alpha+\frac{1}{2}[\alpha, \alpha]=0\right.\right\}, \\\nonumber
		\operatorname{Hom}\left(\left(\alpha, e^r\right),\left(\beta, e^s\right)\right) & =\left\{\lambda \in L^0 \otimes \km \mid \exp (\lambda) \alpha=\beta, \exp (\ep(\lambda)) \cdot e^r=e^s\right\} .
	\end{align} 
	In the preceding notation, one has
	\begin{align}\label{eq:dglaep}
		\Def\left(L^{\bullet}, \ep, A\right)=\Def\left(L^{\bullet}, A\right) \bowtie \exp (\kg \otimes \km)_\ep,
	\end{align} 
	where the $\ep$ subscript meaning that the gauge group acts via $\ep$.
	This gives a covariant functor $\Def\left(L^{\bullet}, \ep,-\right)$ on the category Art of $\mathbb{C}$-Artin local rings with values in the category of small groupoids.
	

\medspace 

A main result by Goldman-Millson in \cite{GM88} is the following:
\begin{thm}[ {\cite[Theorems 2.4 \& 3.8]{GM88}}]\label{thm:GM}
	If $(C_1^\bullet,\ep_1) \rightarrow (C_2^\bullet,\ep_2)$ is a 1-quasi-isomorphism of $\kg$-augmented dgla, then   it induces an equivalence of groupoids $\Def(C_1^\bullet, \ep_1, A)\to \Def(C_2^\bullet, \ep_2, A)$ for any local Artin $\bC$-algebra $A$.    \qed
\end{thm} 
An analytic germ  $(S, s)$ defines a (covariant) functor of Artin rings
\begin{align}\label{eq:functor}
	F_{S, s}:  {\rm Art} & \longrightarrow \text { Set } \\\nonumber
	A & \longmapsto \operatorname{Hom} (\widehat{O}_{S, s}, A ), 
\end{align}
where the $\Hom$ denotes the set of $\bC$-algebra homomorphisms. 
Such a functor is called \emph{pro-representable}.
\begin{thm}[ {\cite[Theorem 3.1]{GM88}}]\label{thm:GM2} Two analytic germs  $(S_1, s_1)$ and $(S_2, s_2)$ are analytically isomorphic if and only if the associated pro-representable functors $F_{S_1, s_1}, F_{S_2, s_2}$ are isomorphic. \qed
\end{thm}

\subsection{Application to the variety of representations}
Let $X$ be a differential manifold.  
For any $\varrho:\pi_1(X)\to\GL_N(\bC)$, we  define a  functor by
$$
\begin{aligned}
	\sR_\varrho:{\rm Art}& \longrightarrow  {\rm Grp},
\end{aligned}
$$ 
whose objects are 
\begin{align}\label{eq:Rrho}
	\obj \sR_\varrho(A) =R_\varrho(A):=\{\tilde{\rho} \in \Hom(\pi_1(X), \GL_N(A)) \mid \tilde{\rho}=\rho  \mbox{ mod } \km\}
\end{align}
and the morphisms 
$G_A^0:=\exp(\kg\otimes\km)$, where $\kg:={\rm Lie}\, \GL_N(\bC)$.  Here, $\kg\otimes\km$ is a nilpotent Lie algebra, the  Baker-Campbell-Hausdorff series define  a nilpotent Lie group structure on the space $G_A^0$. If  we  consider the morphism $q:\GL_N(A)\to \GL_N(\bC)$ induced by $A\to  A/\km\simeq \bC$, then its kernel is $G_A^0$. Note that there exists a natural map
$$
q_*: \Hom(\pi_1(X), \GL_N(A)) \to  \Hom(\pi_1(X), \GL_N(\bC))
$$
whose fibers are objects $\sR_\varrho(A)$. Since $G_A^0$ preserves $R_\varrho(A)$, they defines a groupoid $$[R_\varrho(A)/ G_A^0]=:\sR_\varrho(A).$$

We now consider a flat bundle $(V,\nabla)$ and  let  $(E,\nabla):=(\End(V),\nabla)$ be the associated endomorphism bundle.  We consider a dgla
$$
(L^\bullet,d):=(A^\bullet(X,E),\nabla). 
$$ 
Fix a base point $x \in X$ and an isomorphism $V_x\simeq \bC^N$.     Let $$\varrho:\pi_1(X,x)\to \GL(V_x)\simeq \GL_N(\bC)$$ be the holonomy representation of $\nabla$.   In \cite[Theorems 6.3 \& 6.6]{GM88}, the following result is proved. 
\begin{lem}\label{lem:GMhol}
	There is a natural equivalence of groupoids 
	\begin{align*}
		h:  [\sC(L^\bullet,A)/\exp(L^0\otimes\km)]&\to  [R_\varrho(A)/G_A^0]\\
		\alpha&\mapsto   {\rm hol}_x(\nabla+\alpha).
	\end{align*} \qed
\end{lem}
By eliminating the superfluous isotropy, we have the following isomorphism between groupoids
\begin{align*}
	[R_\varrho(A)/\{{\rm Id} \}]&\to  [R_\varrho(A)/G_A^0]\bowtie G_A^0\\
	\tilde{\rho}&\mapsto (\tilde{\rho},1)
\end{align*}
Consider the functor
\begin{align*}
	\sR_\varrho':	{\rm Art}& \longrightarrow  {\rm Grp},\\
	A & \mapsto 	 [R_\varrho(A)/\{{\rm Id} \}]
\end{align*}
where 
$$\obj \sR'_\varrho(A) =R_\varrho(A)$$ and the only  morphisms in $ \sR'_\varrho(A)$ are   
the identity morphisms. 
Then we have
\begin{lem}\label{thm:GM3}
	The functor $A\mapsto  \sR'_\varrho(A)$ it is prorepresented by the analytic germ $(R(X,N),\varrho)$. \qed
\end{lem}  Therefore, by \cref{lem:GMequiv,lem:GMhol},  the morphism
\begin{align}\label{eq:equiv}
	h\bowtie G_A^0:  [\sC(L^\bullet,A)/\exp(L^0\otimes\km)]\bowtie G_A^0\to  [R_\varrho(A)/\{{\rm Id} \}]
\end{align} 
is an equivalence of groupoids. 

Let $\ep_x:L^\bullet\to \End(V_x)=E_x$ defined by  the evaluation 
$ A^0(X,E)\to E_x $. It follows from \eqref{eq:dglaep} that  there exists a natural equivalence
$$
[\sC(L^\bullet,A)/\exp(L^0\otimes\km)]\bowtie G_A^0\simeq \Def(L^\bullet,\ep_x,A). 
$$
\begin{lem}\label{lem:holonomy}
	The map 
	\begin{align*}
		h:  [\sC(L^\bullet,\ep_x,A)/\exp(L^0\otimes\km)]&\to  [R_\varrho(A)/{\rm Id}]\\
		(\alpha,e^r) &\mapsto  \exp(-r) {\rm hol}_x(\nabla+\alpha) \exp(r)
	\end{align*} 
	is well-defined, and is an equivalence of groupoids. 
\end{lem}
\begin{proof}
	Recall that 
	$$	\operatorname{Hom}\left(\left(\alpha, e^r\right),\left(\beta, e^s\right)\right)   =\left\{\lambda \in L^0 \otimes \km \mid \exp (\lambda) \alpha=\beta, \exp (\ep_x(\lambda)) \cdot e^r=e^s\right\}.$$
	Note that we have
	$$
	\nabla+\exp(\lambda)\alpha  =\exp(\lambda)\circ(\nabla+\alpha)\circ \exp(-\lambda). 
	$$
	This implies that 
	$$
	{\rm hol}_x(\nabla+\exp(\lambda)\alpha ) =\exp(\ep_x(\lambda)){\rm hol}_x(\nabla+\alpha)\exp(\ep_x(-\lambda)). 
	$$
	Therefore, we have
	$$
	h(\beta,e^s)=\exp(-s)\exp(\ep_x(\lambda)){\rm hol}_x(\nabla+\alpha)\exp(\ep_x(-\lambda))\exp(s)=h(\alpha,e^r).
	$$
	Hence $h$ is well-defined. By the above arguments, it is an equivalence of groupoids. 
\end{proof}

\subsection{An extension to quasi-compact K\"ahler manifolds}
Let $X$ be a compact Kähler $n$-fold and let $D$ be a simple normal crossing divisor on $X$. Set $X_0 := X \setminus D$. Let $(V, \nabla)$ be a semisimple flat bundle on $X$, and we denote the associated endomorphism bundle by $(E,\nabla):=(\End(V),\nabla)$.

Consider a twisted logarithmic De Rham complex
$$
(T^\bullet,d):=( A^\bullet(X,D,E),\nabla) = \left( A^0(X,D,E)\stackrel{\nabla}{\to} \cdots\stackrel{\nabla}{\to}  A^{2n}(X,D,E)\right).
$$   
which constitutes a dgla on $X$. Let
$$
(L^\bullet,d):=(A^0(X_0,E)\stackrel{\nabla}{\to}\cdots\stackrel{\nabla}{\to}A^{2\dim X}(X_0,E)) 
$$
be the standard de Rham complex on $X_0$. Fix a base point $x \in X_0$ and let $\varepsilon_x: A^0(X_0, E) \to E_x$ be the augmentation map. The restriction of $(A^\bullet(X, D, E), \nabla)$ to $X_0$ induces a natural morphism of augmented dglas
$$
a:(T^\bullet,\ep_x, d)\to (L^\bullet,\ep_x,d).
$$
Let $\mathcal{V}$ be the local system associated to $(V, \nabla)|_{X_0}$. By \cite[II, 6.10]{Del70}, there is a natural isomorphism between the logarithmic de Rham cohomology and the singular cohomology with coefficients in $\mathcal{V}$:
$$
H^i(T^\bullet,d)\simeq H_{\rm sing}^i(X_0,\cV).
$$
This isomorphism factors through the map
$$
H^i(T^\bullet,d)\to  H^i(L^\bullet,d) 
$$
induced by the restriction morphism $a:T^\bullet|_{X_0}\to L^\bullet$. 
This implies that, the dgla morphism of $a$ is a quasi-isomorphism. Hence, 
by \cref{thm:GM,lem:holonomy}, we have
\begin{lem}\label{lem:holonomy2}
	Let $\varrho$ be the monodromy representation of $(V, \nabla)|_{X_0}$. 
	We define	$$
	R_\varrho(A) := \{ \tilde{\rho} \in \Hom(\pi_1(X_0, x), \GL_N(A)) \mid \tilde{\rho} \equiv \varrho \pmod{\mathfrak{m}} \}
	$$
	The holonomy map 
	h:  \begin{align*}
		h:  [\sC(T^\bullet,\ep_x,A)/\exp(T^0\otimes\km)]&\to  [R_\varrho(A)/{\rm Id}]\\
		(\alpha,e^r) &\mapsto  \exp(-r) {\rm hol}_x(\nabla+\alpha) \exp(r)
	\end{align*}    is well-defined and induces an equivalence of groupoids.\qed
\end{lem}

\section{New perspectives on deformations of local systems on compact K\"ahler manifolds}\label{sec:tour}
In this section, we offer new perspectives on several foundational results established by Goldman--Millson \cite{GM88}, Eyssidieux--Simpson \cite{ES11}, and Eyssidieux et al. \cite{EKPR12}.  While this section is logically independent of the rest of the paper,  it serves as a   warm-up for understanding the quasi-compact case. Readers familiar with the standard compact theory may skip this section without loss of continuity. 
\subsection{Universal connection: analytic construction}\label{subsec:compactuni}
Let $X$ be a compact K\"ahler manifold and let $(V,\nabla)$ be a semisimple flat bundle on $X$, and  we denote the associated endomorphism bundle by $(E,\nabla):=(\End(V),\nabla)$. Fix a basis $\eta_1,\ldots,\eta_p\in \cH^1(X,E)$.  
For any multi-index $I=(i_1,\ldots,i_k)$ with $|I|\geq 2$, we define inductively
\begin{align}  \label{eq:gammaI}
	\gamma_I:&=\sum_{\substack{I=(I_1,I_2); \\ 1\leq |I_1| < |I_2| }}[\eta_{I_1 },\eta_{I_2 }]+\frac{1}{2}\sum_{\substack{I=(I_1,I_2); \\   |I_1| = |I_2|  }}[\eta_{I_1 },\eta_{I_2 }]\\
	\label{eq:induction}\eta_{I} :&=-\nabla''^*G\gamma_I\in A^1(X,E),
\end{align}
where $G$ is Green operator with respect to the Laplacian $\Delta'$.  Note that the decomposition $\nabla = \nabla' + \nabla''$ does not correspond to the standard splitting into $(1,0)$ and $(0,1)$ parts, unless $(V,\nabla)$ is unitary. Instead, this decomposition arises from the harmonic bundle (see \cite{Sim92}).


\begin{lem}\label{lem:gammaeta}
	For any $|I|\geq 2$, we have
	$\eta_I\in \Im \nabla'$, and any $|I|\geq 3$,  $\gamma_I\in  \Im \nabla'$. 
\end{lem}
\begin{proof}
	We prove by induction on $|I|$, and assume that it holds for any $|I|\in \{2,\ldots,k-1\}$. Fix some $I$ with $|I|=k$, we have 
	\begin{align*}
		\nabla'\nabla'^*G\eta_I&=  (1-\cH)\eta_I-\nabla'^*G\nabla'\eta_I\\
		&=  \eta_I+\nabla''^*G\nabla'\gamma_I  =\eta_I.
	\end{align*} 
	Together with \eqref{eq:gammaI}, it follows that $\gamma_I$ is also $\nabla'$-exact when $|I|\geq 3$. 
	The lemma is proved. 
\end{proof}

\begin{lem}\label{lem:quadratic}
	Set 
	\begin{align}\label{eq:sigma}  
		\sigma_k:=  \sum_{|I|=k} z_I\eta_I.
	\end{align}  
	Let $\sQ'$ be the quadratic cone in $\bC^p$ defined by  $$\{[\sum_{i=1}^{p}z_i\eta_i,\sum_{i=1}^{p}z_i\eta_i]\}=0\in H^2(X,E), 
	$$Then for any $z\in \sQ'$,  we have
	\begin{align}\label{eq:MC}
		\nabla\sigma_k=\nabla''\sigma_k=-\sum_{i=1}^{k-1}\sigma_i\wedge\sigma_{k-i}. 
	\end{align}
\end{lem}
\begin{proof}
	Set $\sigma_1(z):=\sum_{i=1}^{m}z_i\eta_i$. 
	Note that 
	$$
	\nabla\sigma_2(z):=-\nabla\nabla''^*G\frac{1}{2}[\sigma_1,\sigma_1]=-(1-\cH)\frac{1}{2}[\sum_{i=1}^{p}z_i\eta_i,\sum_{i=1}^{p}z_i\eta_i]. 
	$$ 
	By the assumption,  for any $z\in \sQ'$,   $[\sigma_1,\sigma_1]$ is cohomologous  to zero.  Hence  we have
	$$
	\nabla \sigma_2(z)+\frac{1}{2}[\sigma_1(z),\sigma_1(z)]=0  
	$$   
	for any $z\in \sQ'$. 
	
	We now make the induction for $\ell\geq 3$. Assume that \eqref{eq:MC} holds for $k=2,\ldots,\ell-1$. 
	Then  
	we have 
	\begin{align*}
		\nabla''(\sum_{|I|=\ell}z_I\gamma_I)=  \nabla''(\sum_{i=1}^{\ell-1}\sigma_i\wedge\sigma_{\ell-i})=0
	\end{align*}
	by the induction for \eqref{eq:MC} and the Jacobian identity. Thus, 
	\begin{align*}
		\nabla\sigma_\ell(z)&=-\sum_{|I|=\ell}z_I\nabla\nabla''^*G\gamma_I\\
		&=-(1-\cH)(\sum_{|I|=\ell}z_I\gamma_I)+\sum_{|I|=\ell}z_I\nabla''^*G\nabla''(\gamma_I)\\      &=-\sum_{|I|=\ell}z_I\gamma_I
	\end{align*}
	where the last equality follows from \cref{lem:gammaeta}.   
	Hence the induction is proved for $\ell$. The lemma  is proved.  
\end{proof}
For any $z\in \bC^p$,  
we   define  a connection  
\begin{align} \label{eq:universalconnection}
	\nabla_z:=\nabla+\sum_{|I|\geq 1}z_I \eta_I.
\end{align}    
for $V$. In general, this connection is neither flat nor smooth. However, by applying elliptic estimates for the Laplacian, there exists an $\ep > 0$ such that on the intersection $\sQ'_\ep := \sQ' \cap \mathbb{B}_\ep^{p}$, the operator $\nabla_z$ forms an analytic family of flat connections for $V$. We omit the proof here, as we will treat the more general quasi-projective setting later.
\subsection{Universal connection: universal property} Consider a dgla defined by
$$
(L^\bullet,d):=(A^\bullet(X,E),\nabla).
$$  Set
$$
(C^\bullet,d):=(\ker \nabla',\nabla''), \quad
(D^\bullet,d):=(H^i(L^\bullet,d),0), 
$$
both of which are naturally endowed with dgla structures.
Then by the formality in \cite{Sim92}, we have the quasi-isomorphisms 
\begin{align}\label{eq:formality}
	L^\bullet \stackrel{\phi}{\leftarrow}  C^\bullet\stackrel{\psi}{\rightarrow}  D^\bullet. 
\end{align} 
wehre $\phi$ is the inclusion,  and $\psi(\eta):=\{\cH(\eta)\}$.  
Fix a base point $x \in X$. Let $\varepsilon: L^0 \to E_x$, $\varepsilon: C^0 \to E_x$, and $\varepsilon: D^0 \to E_x$ be the evaluation maps which define the augmentations.
By \cref{thm:GM}, for any Artin local $\bC$-algebra $A$,  the natural maps 
\begin{align}
	[\sC(L^\bullet,\ep,A)/\exp(L^0\otimes\km)]\leftarrow	  [\sC(C^\bullet,\ep,A)/\exp(C^0\otimes\km)]\to   [\sC(D^\bullet,\ep,A)/\exp(D^0\otimes\km)]
\end{align}
are equivalences of groupoids.  

Let $\sQ_0$ be the quadratic cone in $D^1$ defined by 
\begin{align}\label{eq:defquadratic}
	\sQ_0:=	 \left\{u \in D^1 \otimes \km\mid   [u, u]=0\right\}.
\end{align} 
Denote by 
\begin{align}\label{eq:sQ}
	\sQ:=\mathscr{Q}_0\times \kg / \varepsilon(D^0). 
\end{align}
The following lemma was proven in {\cite[Lemma 3.10]{GM88}}. Here, we refine the result by providing an explicit natural transformation between the functors.
\begin{lem} \label{lem:GM}
	The analytic germ $(\sQ,0)$ (non-canonically) prorepresents the functor  
	\begin{align*}
		{\rm Art}&\to {\rm Grp}\\ 
		A&\mapsto \Def (D^{\bullet}, \ep, A )
	\end{align*}
	where $\Def (D^{\bullet}, \ep, A )$ is the small groupoids  defined in \eqref{eq:defaug}. 
\end{lem}\begin{proof}
	Let $\ell := \dim_{\mathbb{C}}(\mathfrak{g}/\ep(C^0))$. We choose a subset  $\{v_{1},\ldots,v_{\ell}\}$ of $\mathfrak{g}$ such that 
	the images of $\{v_1,\ldots,v_{\ell}\}$ form a basis for the quotient space $\mathfrak{g}/\ep(C^0)$.  
	Then the natural linear map $W\to  \kg/\ep(D^0)$ is an isomorphism.  Since $
	\ep:D^0\to E_x
	$ 
	defined by the evaluation is injective, it follows that  $\exp \left(\ep(D^0) \otimes \km\right)$ acts freely on $G_A^0:=\exp(\kg\otimes\km)$  by left-multiplication.  Hence the following (non-canonical) maps 
	\[
	\begin{tikzcd}
		\kg/\ep(D_0)\otimes \km \arrow[d]  &  W\otimes\km \arrow[r]\arrow[l]\arrow[d]&	\frac{\exp(\kg\otimes\km)}{\exp(\ep(D^0)\otimes \km)}\arrow[d] \\
		p(r)& r   \arrow[l,maps to] \arrow[r,maps to]&  \exp(\ep(D^0)\otimes \km)\cdot \exp(r)
	\end{tikzcd}
	\]
	are both bijective.  Here   $\exp(\ep(D^0)\otimes \km)\cdot \exp(r)$ is the right coset of $\exp(r)$.  
	Hence  we can define an equivalence
	\begin{align}\label{eq:first}
		\sC(D^\bullet,A)  \times W\otimes\km &\to	 \frac{\sC(D^\bullet,\ep,A)}{\exp(D^0\otimes \km)}  \\ \nonumber
		(\alpha,r) &\mapsto  [(\alpha,\exp(r))]. 
	\end{align}  
	Note that $C^0=D^0$, the map
	\begin{align}\label{eq:first12}
		\sC(C^\bullet,A)  \times W\otimes\km &\to	 \frac{\sC(C^\bullet,\ep,A)}{\exp(C^0\otimes \km)}  \\ \nonumber
		(\alpha,r) &\mapsto  [(\alpha,\exp(r))]. 
	\end{align}   
	
	We consider $D^1\times W$ as an affine $\bC$-space.   Let $\cI$ be the ideal of $\bC[D^1\times W]$ generated by the  quadratic equations 
	$$
	[q(u),q(u)]=0, 
	$$   
	where $q:D^1\times W\to D^1$ denotes the projection map. It follows that 
	$$
	\bC[D^1\times W]/\cI=\spec(\sQ). 
	$$
	Then for any element $\alpha\in D^1\otimes\km$  and $r\in W\otimes \km$,  it can be written as 
	\begin{align}\label{eq:alpha}
		\alpha=\sum_{i=1}^{m} \{\eta_i\}\otimes f_i, \quad r=\sum_{i=1}^{\ell} v_i\otimes g_i, 
	\end{align}
	where $f_i,g_j\in \km$. This gives rise to a morphism
	\begin{align*}
		\Hom(\bC[D^1\times W], A)\\
		(\{\eta_i\}^*,(v_j)^*)\mapsto  f_i+g_j
	\end{align*} 
	Here we consider $\{\eta_i\}^*\in (D^1)^*$ and $(v_j)^*\in W^*$. 
	Note  that the above morphism factors through
	\begin{align*}
		\Hom(\bC[D^1\times W]/\cI, A) 
	\end{align*} 
	if and only if $\alpha$ satisfies  $$\frac{1}{2}[\alpha,\alpha]=d(\alpha)+\frac{1}{2}[\alpha,\alpha]=0,$$
	where $d: D^1 \to D^2$ denotes the differential of the dgla  $(D^\bullet, d)$, that is zero.
	
	Let \(\cI_1\) be the ideal of \(\bC[\sQ]\) generated by the image of 
	\((D^1\times W)^* \subset \bC[D^1\times W]\) under the natural projection 
	\(\bC[D^1\times W] \to \bC[D^1\times W]/\cI\).
	Let \(\widehat{\bC[\sQ]}_{\cI_1}\) denote the \(\cI_1\)-adic completion of \(\bC[\sQ]\).
	Since $A$ is an Artin local $\bC$-algebra, it follows that 
	\begin{align*}
		\Hom(\bC[D^1\times W]/\cI, A) \simeq 	\Hom(\widehat{\bC[\sQ]}_{\cI_1}, A). 
	\end{align*} 
	By standard properties of completion of quotient rings, there is a natural isomorphism
	\[
	\widehat{\bC[\sQ]}_{\cI_1} \;\simeq\; 
	\frac{\bC[[D^1\times W]]}{\cI},
	\]
	where, by abuse of notation, we continue to denote by \(\cI\) the ideal  
	\(\cI\cdot  \bC[[D^1\times W]]\).  
	
	On the other hand, \(\cI_1\)-adic completion commutes with localization at the
	maximal ideal \(\cI_1\) of \(\bC[\sQ]\).
	Therefore  we have a natural isomorphism
	$$
	\widehat{\bC[\sQ]}_{\cI_1} \stackrel{\simeq}{\to}
	\widehat{\cO_{\sQ,0}}.
	$$  
	Therefore, the following map 
	\begin{align}\label{eq:precise}
		\sC(D^\bullet,A)  \times W\otimes\km\to 	\Hom( \widehat{\cO_{\sQ,0}}, A) \\ \nonumber
		(\alpha,r) \mapsto \left(   	(\{\eta_i\}^*,v_j^*)\mapsto  f_i+g_j \right)
	\end{align}
	is an equivalence, where $(\alpha,r)$ is written as in~\eqref{eq:alpha}, and  $\{\eta_i\}^*$ and $v_j^*$ are understood as the images of  $\{\eta_i\}^*$,$v_j^*$  under the composite morphisms 
	$$\bC[D^1\times W]\to \bC[\sQ]\to \cO_{\sQ,0}\to\hcO_{\sQ,0}.$$  
	We remark that $\widehat{\cO_{\sQ,0}}$ is generated by 
	$\{\eta_i\}^*$ and $v_j^*$.  
	Therefore, for any morphism in 
	$\Hom(\widehat{\cO_{\sQ,0}},A)$, the images of these generators
	completely determine the morphism.  
	Combining \eqref{eq:precise} with \eqref{eq:first} yields the lemma. 
\end{proof}

Fix a positive integer $k\geq 2$.   We set
\begin{align*} 
	\sQ'':=\sQ'\times \bC^\ell. 
\end{align*} 
Let $\km$ be the maximal ideal of $\cO_{\mathscr{Q}'',0}$ and we consider $A_k:= \cO_{\mathscr{Q}'',0}/\km^{k+1}$, that is an  Artin local $\bC$-algebra.  Write $\km_k:=\km A_k$, that is the maximal ideal for $A_k$.   We now regard $z_I$ and $x_j$  as the image of   $z_I$ and  $x_j$ under the morphism 
$$\cO_{\bC^{p+\ell},0}\to \cO_{\mathscr{Q}'',0}\to \cO_{\mathscr{Q}'',0}/\km^{k+1}=A_k.$$
Then $z_I\in \km_k$   for each $I$ with $|I|\in \{1,\ldots,k\}$. 
For any $i\in \{1,\ldots,k\}$, let 
$$\sigma_i=\sum_{|I|=i}z_I \eta_I $$
be defined in \eqref{eq:sigma}. 
By \cref{lem:gammaeta},   we have
$\eta_I\in \Im \nabla'$ for $|I|\geq 2$, and for any $z\in \sQ'$, we have
$$\nabla'\sigma_i(z)=0$$ for any $i\geq 1$.  This implies that 
\begin{align}\label{eq:psivani}
	\psi(\sigma_i)=0\in D^1\otimes\km_k,\quad \mbox{ if }i\geq 2,
\end{align} 
and $$
\delta_k:=\sum_{1\leq |I|\leq k}z_I\eta_I\in C^1\otimes\km_k. 
$$
On the other hand, by \cref{lem:quadratic},   we have 
$$
\delta_k:=\sum_{1\leq |I|\leq k}z_I\eta_I\in \sC(C^\bullet,A_k). 
$$
Together with \eqref{eq:psivani}, this yields
$$
\psi(\delta_k)=\sum_{i=1}^{p}z_i\{\eta_i\}\in \sC(D^\bullet,A_k). 
$$  
By \eqref{eq:formality} and \cref{thm:GM},    we have  the equivalences
\begin{align}\label{eq:Phi}
	\Phi:  	\sC(C^\bullet,A)  \times W\otimes\km_k &\stackrel{\psi\times \mathrm{Id}}{\to } 	\sC(D^\bullet,A)  \times W\otimes\km_k  {\to} \Hom(\hcO_{\sQ,0},A_k),
\end{align}  
with 
$$
\Phi(\delta_k, \sum_{i=1}^{\ell}w_iv_i)	=  \left(	(\{\eta_i\}^*,v_j^*)\mapsto  z_i+w_j  \right).
$$
As remarked above,  since $\widehat{\cO_{\sQ,0}}$ is generated by 
$\{\eta_i\}^*$ and $v_j^*$,  any morphism in 
$\Hom(\widehat{\cO_{\sQ,0}},A_k)$ is determined by the images of these generators.

\begin{thm}\label{thm:ES}Let \(\nabla_{z}\) denote the analytic family of flat connections on \(V\),
	parametrized by \(z \in \sQ'_{\ep}\) for some \(\ep>0\), as defined in \eqref{eq:universalconnection}. Let $\mathrm{ev}_{x}:H^0(X,E)\to E_x$  be the evaluation map, and let  $\{v_{1},\ldots,v_{\ell}\}$ be a subset of $E_x$ such that 
	the images of $\{v_1,\ldots,v_{\ell}\}$ form a basis for the quotient space $\mathfrak{g}/\mathrm{ev}_{x}H^0(X,E)$.  
	Let  ${\rm Mon}(\nabla_z) $ be the monodromy representation for any $z\in \mathscr{Q}'_\ep$.  Then  the natural map between analytic germs
	\begin{align}\label{eq:monof}
		f:	(\sQ'_\ep\times \bC^\ell,0) &\to \left( R(X,N),\varrho\right)\\\nonumber
		(z,w)&\mapsto \exp(-\sum_{i=1}^{\ell}w_iv_i){\rm Mon}(\nabla_z)\exp(\sum_{i=1}^{\ell}w_iv_i) 
	\end{align}
	is an  analytic isomorphism. 
\end{thm} 
\begin{proof}
	We recall that 	 we have morphisms of groupoids 
	\begin{equation}\label{eq:great}\hspace{-0.45cm}
		\begin{tikzcd}[column sep=1em]
			& \Hom\!\left(\widehat{\cO}_{\sQ,0},A_k\right)		& \\
			{ [R_\varrho(A_k)/{\rm Id}] }	 &	\sC(C^\bullet,A_k) \times W\otimes\km_k \arrow[r,"\psi\times \mathrm{Id}"] \arrow[d]\arrow[u,"\Phi"]
			& 	\sC(D^\bullet,A_k)  \times W\otimes\km_k  \arrow[ul] \arrow[d]  \\
			{[\sC(L^\bullet,\ep,A_k)/\exp(L^0\otimes\km_k)]} \arrow[u,"h"']
			& {[\sC(C^\bullet,\ep,A_k)/\exp(C^0\otimes\km_k)]} \arrow[r,"\psi"]\arrow[l,"\Phi"'] 
			& {[\sC(D^\bullet,\ep,A_k)/\exp(D^0\otimes\km_k)]}    
		\end{tikzcd}
	\end{equation}
	in which each of them is an isomorphism by \eqref{eq:first}, \eqref{eq:first12}, \eqref{eq:formality} and \cref{thm:GM}.     Here \begin{align*}
		h:  [\sC(L^\bullet,\ep_x,A_k)/\exp(L^0\otimes\km)]&\to  [R_\varrho(A_k)/{\rm Id}]\\
		(\alpha,e^r) &\mapsto  \exp(-r) {\rm hol}_x(\nabla+\alpha) \exp(r)
	\end{align*}  
	is defined in   \Cref{lem:holonomy}. 
	We denote by
	$$
	\Upsilon: \sC(C^\bullet,A_k) \times W\otimes\km \to  [R_\varrho(A_k)/{\rm Id}]
	$$
	the induced equivalence of groupoids in \eqref{eq:great}. It follows that 
	$$
	\Upsilon(\delta_k,\sum_{i=1}^{\ell}w_iv_i)=\exp(-\sum_{i=1}^{\ell}w_iv_i){\rm Mon}(\nabla_z)\exp(\sum_{i=1}^{\ell}w_iv_i)\in R_\varrho(A_k)
	$$
	corresponds to the morphism 
	$$
	f_k: \hcO_{R(X,N),\varrho}\to  A_k=\cO_{\mathscr{Q}'',0}/\km^{k+1}
	$$
	induced by $f$ defined in \eqref{eq:monof}.  
	
	Let us define a morphism  
	$
	g_k\in \Hom(\hcO_{\sQ,0},A_k) 
	$
	by setting
	\begin{align}\label{eq:g_k}
		g_k(\{\eta_i\}^*)=  z_i, \quad g_k(v_j^*)=w_j. 
	\end{align}
	for each $i\in \{1,\ldots,m\}$ and $j\in \{1,\ldots,\ell\}$.  By \eqref{eq:Phi}, we have
	$$
	\Phi (\delta_k, \sum_{i=1}^{\ell}w_iv_i)=g_k. 
	$$

	One can see the compatibility of $ {g}_k$ by
	\[
	\begin{tikzcd}
		& \vdots \arrow[d] \\
		& A_{k+1} \arrow[d] \\
		\hcO_{\mathscr{Q},0}  \arrow[ur," {g}_{k+1}"] \arrow[r," {g}_k"']  \arrow[ddr,"g_1"']
		& A_k \arrow[d] \\ 
		& \vdots\arrow[d]\\
		&	A_1
	\end{tikzcd}
	\]
	In conclusion, there exists   natural maps \begin{equation*}
		\begin{tikzcd} 
			\Hom\!\left(\widehat{\cO}_{\sQ,0},A_k\right)
			\arrow[r] 
			&
			\Hom\!\left(\widehat{\cO}_{\sQ,0},A_{k-1}\right) 
			\\
			\sC(C^\bullet,A_k) \times W\otimes\km_k 
			\arrow[r]
			\arrow[u,"\Phi"]
			\arrow[d,"\Upsilon"']
			&
			\sC(C^\bullet,A_{k-1}) \times W\otimes\km_{k-1}
			\arrow[u,"\Phi"']
			\arrow[d,"\Upsilon"]
			\\
			\Hom(\widehat{\cO}_{{R(X,N)},\varrho}, A_k)
			\arrow[r]
			&
			\Hom(\widehat{\cO}_{{R(X,N)},\varrho}, A_{k-1})
		\end{tikzcd}
	\end{equation*}
	such that all vertical maps are isomorphisms.   We have the inverse limit
	$$
	\varprojlim_k g_k\in  \Hom\!\left(\widehat{\cO}_{\sQ,0},\varprojlim_k A_k\right)=\Hom\!\left(\widehat{\cO}_{\sQ,0},\widehat{\cO}_{\sQ'',0}\right)
	$$
	and by \eqref{eq:g_k}, it is an isomorphism.
	By \eqref{eq:great}, it follows that
	$$
	f^*=\varprojlim_k f_k: \widehat{\cO}_{R(X,N),\varrho}\to   \hcO_{\mathscr{Q}'',0},
	$$
	is  also an isomorphism. 
	By \cref{thm:GM2}, this implies that 
	$f:(\sQ'_\ep\times W,0)\to (R(X,N),\varrho)$ is a \emph{biholomorphism} between analytic germs. 
	The theorem is therefore proved. 
\end{proof}

\subsection{1-step phenomenon}
We are now in a position to prove the result on the so-called \emph{1-step phenomenon} for semisimple systems on compact Kähler manifolds.
\begin{thm}[1-step phenomenon]\label{thm:1-step}
	Let $h:Y \to X$ be a holomorphic map between compact K\"ahler manifolds.
	Let $(V,\nabla)$ be a semisimple flat bundle on $X$ with monodromy 
	representation $\varrho:\pi_1(X,x)\to \GL_N(\bC)$, and set $y=h(x)$.
	Denote by 
	\[
	h^*: R(X,N) \longrightarrow R(Y,N)
	\]
	the induced morphism of representation varieties, given by
	$\tau \mapsto h^*\tau$ for any $\tau \in R(X,N)$.
	
	Assume that:
	\begin{enumerate}[label=(\alph*)]
		\item\label{item:center} 
		the pullback representation $h^*\varrho$ has image in the center 
		$Z(\GL_N(\bC))$, that is, $h^*\varrho$ factors through a representation
		$\tau_0:\pi_1(Y,y)\to Z(\GL_N(\bC))$, where
		\[
		Z(\GL_N(\bC))=\{\lambda\,\mathrm{Id}_N\mid \lambda\in \bC^\times\};
		\]
		
		\item 
		for every $\eta \in H^1(X,E)$, one has $h^*\eta = 0$ in 
		$H^1(Y,h^*E)$.
	\end{enumerate}
	
	Let $R$ be any (geometrically) irreducible component of $R(X,N)$ 
	containing $\varrho$.  
	Then the image of $R$ under $h^*$ consists of the single point
	\[
	h^*(R) = \{\tau_0\}.
	\]
\end{thm}

\begin{proof}
	We use the notations introduced in this subsection. Fix any $z\in \sQ'_\ep$.  We shall prove that $h^*\sigma_k(z)=0$ for each $k\geq 1$.  We prove  by induction on $k$. For $k=1$,  we note that $h^*\sigma_1(z)\in \cH^1(Y,h^*E)$.  By Item (b), we have $h^*\sigma_1(z)=0$. 
	
	Assume that $h^*\sigma_k(z)=0$ for any $k=1,\ldots,m-1$. By \eqref{eq:MC}, we have
	$$
	\nabla h^*\sigma_m(z)=0. 
	$$
	Since $\sigma_i(z)\in {\rm Im}\nabla'$ for each $i\geq 2$, it follows that
	$$
	h^*\sigma_m\in \ker\nabla\cap {\rm Im}\nabla'. 
	$$
	By \cref{lem:ddbar}, it yields that $h^*\sigma_m(z)=0$. This completes the induction and thus $h^*\sigma_k(z)=0$ for each $z\in \sQ'_\ep$ and each $k\geq 1$.   
	This implies that
	\[
	h^*\nabla_z = h^*\nabla
	\qquad \text{for every } z\in \sQ'_\ep.
	\]
	
	On the other hand, by \Cref{item:center}, for any $g\in \GL_N(\bC)$ we have 
	$g^{-1}\tau_0 g = \tau_0$; hence the representation $\tau_0$ is fixed under
	conjugation. 
	Therefore, for the map $f: \sQ'_\ep\times W \to  R(X,N) $ constructed in \cref{thm:ES}, the composition
	\[
	h^* \circ f : \sQ'_\ep\times W  \longrightarrow R(Y,N).
	\]
	is constant. 
	By \cref{thm:ES}, this implies that $h^*$ is constant on some neighborhood of $\varrho$ in  $R$. 
	Since $h^*$ is an algebraic morphism, it follows that 
	\[
	h^*(R)=\{h^*\varrho\}=\{\tau_0\}. 
	\]
	The theorem follows.
\end{proof}

\subsection{Geometric applications I: Katzarkov's and Campana's theorems}\label{subsec:KC}
\begin{dfn} [Shafarevich morphism]\label{def:Sha}
	Let $X$ be a quasi-compact K\"ahler manifold, and let 
	$\varrho:\pi_1(X)\to \GL_N(\mathbb{C})$ be a linear representation. 
	Assume that there exists a \emph{compactifiable} dominant holomorphic map
	\[
	{\rm sh}_\varrho : X \to {\rm Sh}_\varrho(X)
	\]
	onto a complex normal variety ${\rm Sh}_\varrho(X)$ with connected general fibers, 
	such that for every compactifiable holomorphic map $f:Y\to X$, the following properties are equivalent:
	\begin{enumerate}[label=\textup{(\alph*)}]
		\item ${\rm sh}_\varrho\circ f(Y)$ is a point;
		\item $ f^*\varrho(\pi_1(Y)) $ is finite; 
	\end{enumerate} 
	Then such ${\rm sh}_\varrho$ is called the \emph{Shafarevich morphism} of $\varrho$. 
\end{dfn}

\begin{dfn}[Big representation]\label{def:bigrep}
	Let $X$ be a smooth quasi-projective variety.
	A linear representation
	$
	\varrho:\pi_1(X)\to \GL_N(\bC)
	$ 
	is called \emph{big} if, for any positive-dimensional closed subvariety
	$Z \subset X$ passing through a very general point,
	the group
	\[
	\varrho\!\left(\mathrm{Im}\bigl[\pi_1(Z^{\mathrm{norm}})\to \pi_1(X)\bigr]\right)
	\]
	is infinite, where $Z^{\mathrm{norm}}$ denotes the normalization of $Z$.
\end{dfn}
In \cite{Kat95}, Katzarkov proved that for a smooth projective variety with   nilpotent fundamental groups, its Shafarevich morphism is the Stein factorization of its Albanese map. This was extended to compact K\"ahler manifolds by Leroy \cite{Ler} in his PhD thesis (see also the survey paper by Claudon \cite{Cla}).  Based on \cref{thm:1-step}, we can give another proof of  Katzarkov's theorem following same ideas in \cite{EKPR12}. 

We first prove the following lemma. 
\begin{lem}\label{lem:common}
	Let $X$ be a compact K\"ahler manifold. Assume that $\pi_1(X)$ is a torsion free nilpotent group. 
	Let $a:X\to A$ denote the Albanese map, and let $f:Y\to X$ be a holomorphic map
	from another compact K\"ahler manifold such that $(a\circ f)(Y)$ is a point.
	Then $f_*(\pi_1(Y))=\{1\}$. 
\end{lem}
\begin{proof}
	Since we assume that $\pi_1(X)$ is nilpotent and torsion-free, by a theorem of Malcev,  there exists a  faithful representation
	$$
	\varrho:\pi_1(X)  \hookrightarrow \operatorname{UT}_N(\bC),
	$$
	where $  \operatorname{UT}_N$  is the algebraic group of upper triangular matrices with all diagonal entries equal to \(1\).
	Hence $\varrho$ belongs to the same irreducible component of   the trivial representation
	$\varrho_0:\pi_1(X)\to \GL_N(\bC)$ (see e.g  \cref{claim:samecon} below for a proof).

	Let $(V,\nabla)$ denote the trivial flat bundle of rank $N$.    
	By the universal property of the Albanese map, one has $f^*\omega=0$ for 
	every $\omega\in H^1(X,\bC)$.  
	Since $(V,\nabla)$ is the trivial flat bundle, it follows that 
	$f^*\eta=0$ for all $\eta\in H^1(X,\End(V))$.  
	Thus the hypotheses of \cref{thm:1-step} are satisfied.   
	Hence $f^*\varrho$ is a trivial representation by \cref{thm:1-step}. Since $\varrho$ is faithful, il yields the lemma. 
\end{proof}
\begin{thm}[Katzarkov, Leroy]\label{thm:Kat}
	Let $X$ be a compact K\"ahler manifold. Assume that $ \pi_1(X)$ is   nilpotent and torsion free. Then the Shafarevich morphism of $X$ is the Stein factorization of its Albanese map $a:X\to A$. Moreover, the universal cover of $X$ is holomorphically convex. 
\end{thm}
\begin{proof}
	Let  $ X\stackrel{g}{\to} S\to A$ be the Stein 
	factorization of $a$.  
	Let $f:Y\to X$ be a holomorphic map from another compact K\"ahler manifold 
	such that $(g\circ f)(Y)$ is a point.  
	By  \cref{lem:common},  $f_*(\pi_1(Y))=\{1\}$.   
	
	Finally, by the characterization of the Albanese map, for any holomorphic 
	map $h:Z\to X$ from a compact K\"ahler manifold, the group $h^*H^1(X,\bZ)$ 
	is finite if and only if $(a\circ h)(Z)$ is a point.  
	Combining this with the property established above, we conclude that $g$ 
	is the Shafarevich morphism of $X$. 
	
	We omit the proof of holomorphic convexity here, as it is identical to the proof of the more general \cref{thm:GGK} established later.
\end{proof}
We also provide a new proof of Campana's abelianity conjecture in the linear
and compact setting.

\begin{thm}[\cite{Cam04}]\label{thm:Cam}
	Let $X$ be a smooth projective variety.  
	If $X$ is special, then every linear representation 
	$\varrho:\pi_1(X)\to \GL_{N}(\bC)$ has virtually abelian image.
\end{thm}

\begin{proof}
	By \cite{CDY25b}, after replacing $X$ by a finite \'etale cover, we may assume
	that $\Gamma:=\varrho(\pi_1(X))$ is nilpotent and torsion free.  We choose  faithful representation
	$$
	j:\Gamma  \hookrightarrow \operatorname{UT}_N(\bC),
	$$
	and replace $\varrho$ by $j\circ \varrho$. Then 
	$\varrho$ belongs to the same irreducible component of   the trivial representation
	$\varrho_0:\pi_1(X)\to \GL_N(\bC)$. 
	
	By \cite{Cam04}, the Albanese map 
	$a:X\to A$ is a surjective morphism with connected fibers $F$. 
	Let
	By  \cite{CDY25b}, we have an exact sequence
	\begin{equation}\label{eq:pi}
		\pi_1(F)\longrightarrow \pi_1(X)\longrightarrow \pi_1(A)\longrightarrow 1.
	\end{equation} 
	By the same proof of \cref{lem:common}, for   a general smooth fiber $F$ of $a$, the image $\varrho({\rm Im}[\pi_1(F)\to \pi_1(X)])$   is trivial.  
	Hence  the representation $\varrho$   factors through 
	$\pi_1(A)$, and thus has abelian image.  
	This proves the theorem.
\end{proof}
\subsection{Geometric applications II: a structure theorem for the linear Shafarevich morphism}

The following result is implicit in \cite{EKPR12}. It played a crucial role in the work of Wang and the second author on the proof of the linear Koll\'ar conjecture (cf. \cite[Proof of Theorem 4.6]{DW24b}). Here, we provide an alternative proof using \cref{thm:1-step}.

\begin{thm}\label{thm:structure}
	Let $X$ be a smooth projective variety, and let $\varrho : \pi_1(X) \to \GL_N(\bC)$ be a big linear representation. Let $\sigma:\pi_1(X)\to \GL_N(\bC)$ denote its semisimplification, which has torsion free image. Let ${\rm sh}_{\sigma}:X\to {\rm Sh}_{\sigma}(X)$ be the Shafarevich morphism associated with $\sigma$, whose existence is guaranteed by \cite{Eys04} (see also \cite{DY23}). If $Y$ is a very general fiber of ${\rm sh}_{\sigma}$, then the Albanese map of $Y$ is generically finite onto its image.
\end{thm}

\begin{proof}
	Let $(V,\nabla)$ be the semisimple flat bundle on $X$ associated with the monodromy representation of $\sigma$. 
	Let $f:Y\to X$ be the natural morphism. Since $\varrho$ is big and $Y$ is a very general fiber of ${\rm sh}_{\sigma}$, $Y$ is smooth and  the pullback $f^*\varrho:\pi_1(Y)\to \GL_N(\bC)$ remains big. On the other hand, by \cref{def:Sha}, the image of $\pi_1(Y)$ under $\sigma$ is finite; given our torsion-free assumption, $f^*\sigma$ is trivial. Consequently, the pullback bundle $f^*(V,\nabla)$ is a trivial flat bundle.
	
	Let $a:Y\to A$ be the Albanese map. Let $Z$ be a connected component of a very general smooth fiber of $a$, and let $g:Z\to Y$ denote the natural inclusion.  The pullback $g^*f^*\varrho:\pi_1(Z)\to \GL_N(\bC)$ is also big. Since $Z$ maps to a point in $A$, the induced map on cohomology vanishes, i.e., $g^*H^1(Y,\bC)=0$. Since $f^*(V,\nabla)$ is a trivial flat bundle,  we have
	\[
	g^*H^1(Y,f^*\End(V)) \simeq g^*H^1(Y,\bC)\otimes \bC^{N^2} = 0.
	\]
	Define $h:=f\circ g$. The  conditions in \cref{thm:1-step} are thus satisfied for $h:Z\to X$.
	\begin{claim}\label{claim:samecon}
		The representations     $\varrho$ and $\sigma$ are in the same irreducible component of $R(X,N)$. 
	\end{claim}
	\begin{proof}
		Let $V = \mathbb{C}^N$. Since $V$ is a finite-dimensional vector space, there exists a Jordan-Hölder filtration  of $\pi_1(X)$-submodules:
		\[
		0 = V_0 \subset V_1 \subset V_2 \subset \dots \subset V_k = V,
		\]
		such that the successive quotients $W_i := V_i / V_{i-1}$ are irreducible representations. The semisimplification $\sigma$ is defined as the direct sum representation acting on $\bigoplus_{i=1}^k W_i$.
		
		We choose a basis for $V$ compatible with this filtration. With respect to this basis, for any $\gamma \in \pi_1(X)$, the matrix $\rho(\gamma)$ is block upper-triangular:
		\[
		\rho(\gamma) = \begin{pmatrix}
			\rho_1(\gamma) & * & \dots & * \\
			0 & \rho_2(\gamma) & \dots & * \\
			\vdots & \vdots & \ddots & \vdots \\
			0 & 0 & \dots & \rho_k(\gamma)
		\end{pmatrix},
		\]
		where $\rho_i(\gamma)$ denotes the action on the quotient $W_i$. The semisimplification corresponds to the block-diagonal part:
		\[
		\sigma(\gamma) = \text{diag}(\rho_1(\gamma), \rho_2(\gamma), \dots, \rho_k(\gamma)).
		\]
		
		To construct the deformation, we define a one-parameter family of diagonal matrices. Let $n_i = \dim(W_i)$. For $t \in \mathbb{C}^*$, define $D(t) \in \mathrm{GL}_N(\mathbb{C})$ as:
		\[
		D(t) = \text{diag}(t^{w_1} I_{n_1}, t^{w_2} I_{n_2}, \dots, t^{w_k} I_{n_k}),
		\]
		where   $w_i := k-i$.
		
		We define an algebraic morphism $\Phi: \bC^* \to R(X,N)$ by conjugation:
		\[
		\Phi(t)(\gamma) := D(t) \varrho(\gamma) D(t)^{-1}.
		\]
		A direct calculation shows that the limit $\lim_{t \to 0} \Phi(t)(\gamma)$ exists and equals $\sigma(\gamma)$. Thus, $\Phi$ extends to a regular morphism on $\bC$ such that $\Phi(0)=\sigma$. Since $\Phi(1)=\varrho$, the representations $\sigma$ and $\varrho$ are connected by an affine line, and therefore lie in the same irreducible component of $R(X,N)(\bC)$.
	\end{proof}
	
	We now apply \cref{thm:1-step} to conclude that $h^*\varrho$ is trivial. However, recall that $h^*\varrho$ is big when restricted to the very general fiber $Z$. The only way a representation can be both trivial and big is if the domain is a point. Therefore, $\dim Z = 0$. Since $Z$ is the generic fiber of the Albanese map, this implies $\dim a(Y)=\dim Y$. The theorem is thus proved.
\end{proof}

\medskip

In the remainder of the paper, we pursue a similar strategy to study \cref{conj:nilpotent}. However, in general there is no formality as in \eqref{eq:formality}, and, as we shall see later, the \emph{1-step phenomenon} of \cref{thm:1-step} does not hold in this setting. We therefore have to develop Hodge-theoretic methods for the deformation of local systems over quasi-compact K\"ahler manifolds, and subsequently establish a \emph{2-step phenomenon} in this more general context.

\section{Deformation theory of local systems over quasi-compact K\"ahler manifolds}
The results in this section form the technical core of the paper. Here, we prove \cref{main3,main:zigzag}, which provide the foundation for the deformation theory of local systems over quasi-compact K\"ahler manifolds. 

Throughout this section, let $(X,\omega)$ be a compact K\"ahler manifold and let $D=\sum_{i=1}^{m}D_i$ be a simple normal crossing divisor on $X$. Set $X_0 = X \setminus D$. Let $(V,\nabla)$ be a  \emph{unitary} flat bundle on $X$, and let $(E,\nabla) = (\End(V), \nabla)$ be the induced bundle of endomorphisms. Note that the unitary structure on $V$  induces a  hermitian flat   metric $h$ for $(E,\nabla)$.  We write    $\nabla=\nabla'+\nabla''$ where $\nabla'$ is the $(1,0)$-part and $\nabla''$ is the $(0,1)$-part.   Let  $\bnabla'^*$ and $\bnabla''^*$ denote the formal adjoints of $\bnabla'$ and $\bnabla''$, defined with respect to the metrics $h$ and $\omega$. Finally, let $G$ be the Green operator associated with the Laplacian $\Delta'=\Delta''$. These operators satisfy the relations given in \eqref{eq:basicgreen}; we will use these identities freely throughout  this section. 

\subsection{Residue maps}
Let us first recall the definition of residues.

For a multi-index $I = (i_1, \ldots, i_k)$ with entries in $\{1, \ldots, m\}$, we denote its length by $|I|=k$. We say that $I$ is \emph{ordered} if $i_1 < \cdots < i_k$. 

For any ordered $I$, we set $D_I:=D_{i_1}\cap\cdots\cap D_{i_j}$.  
Set $D^{(j)}:=\sqcup_{|I|=j}D_I$ and let 
$$D_{(j)}:=\sqcup_{|I|=j}\sum_{p\notin I}D_I\cap D_p$$
that is a simple normal crossing divisor on $D^{(j)}$.   

Fix any $\eta \in A^k(X,D,E)$ and any integer $j \leq k$.  For any point $x \in D_I$ with $|I|=j$, choose a coordinate neighborhood $(U; z_1, \ldots, z_n)$ centered at $x$ such that there exists $j'\geq j$ with 
\[
D\cap U=(\cup_{\ell=1}^{j'}D_{i_\ell})\cap U; \quad D_{i_\ell}\cap U  = \{ z_\ell = 0 \}
\]
for each $\ell \in \{1, \ldots, j'\}$. On this neighborhood $U$, we can decompose $\eta$ as
\begin{align*} 
	\eta|_U =  \frac{dz_1}{z_1} \wedge \cdots \wedge \frac{dz_j}{z_j}\wedge a(z) + b(z),
\end{align*} 
where $a(z)$ is a section of $\cA^{k-j}(X, \sum_{p \notin I}D_p, E)(U)$, and $b(z)$ is a term that does not contain the full wedge product $\frac{dz_1}{z_1} \wedge \cdots \wedge \frac{dz_j}{z_j}$. We then define the residue locally by
\begin{align}\label{eq:localresidue}
	\Res_{D_I}(\eta|_U) \coloneqq (2\pi\sn)^ja(z)|_{D_I \cap U} \in \cA^{k-j}\Big(D_I  , \sum_{p \notin I}D_p \cap D_I, E\Big)(U\cap D_I).
\end{align}
One can verify that this definition is well-defined and independent of the choice of local coordinates. Consequently, these local constructions glue together to get
$$
\Res_{D_I}\eta\in  A^{k-j}(D_I,\sum_{p\notin I}D_I\cap D_p, E).
$$
for each strata $D_I$ with $|I|=j$.

We write 
$$
\Res_D \eta = \left(\Res_{D_i}\eta\right)_{i=1,\ldots,m}. 
$$
We rely on the following fact throughout this section.
\begin{lem}\label{lem:equiv}
	Let $\eta \in A^k(X, D, E)$. Then we have\begin{equation}\label{eq:diffresi}
		\bnabla''\eta =    \iota_*(\Res_D\eta) + \nabla''\eta.
	\end{equation} 
\end{lem}

\begin{proof}
	It suffices to verify the statement locally. By \eqref{eq:localresidue}, we have  
	\begin{equation*} 
		\bnabla''\eta = \sum_{i=1}^{m}(\iota_i)_*\Res_{D_i}\eta + \nabla''\eta = \iota_*(\Res_D\eta) + \nabla''\eta.
	\end{equation*}
	The lemma follows immediately.
\end{proof}

\subsection{Gysin morphism} 
Recall the complex of sheaves of twisted logarithmic forms
\begin{equation}\label{eq:logcomplex}
	(\mathcal{A}^\bullet(X,D,E),\nabla) := \mathcal{A}^0(X,D,E)\xrightarrow{\nabla} \cdots\xrightarrow{\nabla} \mathcal{A}^{2n}(X,D,E),
\end{equation}
where $\mathcal{A}^k(X,D,E)$ denotes the sheaf of smooth differential $k$-forms on $X$ with logarithmic singularities along $D$ taking values in the flat vector bundle $E$, and the differential $\nabla$ is the exterior covariant derivative induced by the flat connection on $E$.

Consider the weight filtration $W_\bullet \mathcal{A}^{\bullet}(X,D,E)$ defined by the number of logarithmic poles. Locally, $W_k \mathcal{A}^m(X, D, E)$ is spanned by forms containing at most $k$ logarithmic differentials $dz_{i}/z_{i}$. Explicitly, these are linear combinations of terms of the form
$$
\alpha \wedge \frac{dz_{i_1}}{z_{i_1}} \wedge \cdots \wedge \frac{dz_{i_j}}{z_{i_j}},
$$
where $\alpha$ is a smooth form and $j \le k$. In this paper, we are primarily interested in the case $W_1 \mathcal{A}^\bullet(X, D, E)$, which contains forms with at most one logarithmic factor.

For the general filtration, we have the $k$-th residue map
\begin{equation}
	\operatorname{Res}^{(k)}: W_k \mathcal{A}^{\bullet}(X,D,E) \to (\iota_k)_*\mathcal{A}^{\bullet-k}_{D^{(k)}}(E),
\end{equation}
where $\iota_k: D^{(k)} \hookrightarrow X$ denotes the inclusion. We adopt the convention that $\mathcal{A}^{j}_{D^{(k)}}(E)=0$ if $j<0$. Let $W'_{k-1}\mathcal{A}^{\bullet}(X,D,E)$ denote the kernel of this residue map, and define its complex of global sections by
\begin{equation}\label{eq:global}
	W'_{k-1} A^{\bullet}(X,D,E) := \Gamma\left(X, W'_{k-1}\mathcal{A}^{\bullet}(X,D,E)\right).
\end{equation}
Since these are fine sheaves, the functor of taking global sections is exact. Consequently, we obtain the following short exact sequence of complexes:
\begin{equation}\label{eq:extension}
	0\to W'_{k-1} A^{\bullet}(X,D,E) \to W_k A^{\bullet}(X,D,E)\xrightarrow{\operatorname{Res}^{(k)}} A^{\bullet-k}(D^{(k)},E) \to 0.
\end{equation}
Note that we have the quasi-isomorphism
\begin{align}\label{eq:gr}
	{\rm Gr}^W_k \cA^{\bullet}(X,D,E)\to (\iota_k)_*\cA^{\bullet-k}_{D^{(k)}}(E) 
\end{align} 
Since we have 
$$0\to W_{k-1} \cA^{\bullet}(X,D,E)\to W_{k} \cA^{\bullet}(X,D,E)\to \gr_k^W\cA^{\bullet}(X,D,E)\to 0,$$
if we take global sections, we have 
$$0\to W_{k-1}  A^{\bullet}(X,D,E)\to W_{k}  A^{\bullet}(X,D,E)\to \Gamma(X,\gr_k^W\cA^{\bullet}(X,D,E))\to 0. $$
This implies that
$$
\Gamma(X,\gr_k^W\cA^{\bullet}(X,D,E))\simeq \gr_k^WA^{\bullet}(X,D,E).
$$
Together with \eqref{eq:gr}, the natural morphism 
\begin{align*} 
	{\rm Gr}^W_k A^{\bullet}(X,D,E)\to  A^{\bullet-k}_{D^{(k)}}(E) 
\end{align*} 
induced by the residue  gives to  an isomorphism of their cohomologies 
\begin{align}\label{eq:qis}
	H^i({\rm Gr}^W_k A^{\bullet}(X,D,E))\to H^i(A^{\bullet-k}_{D^{(k)}}(E)).
\end{align} 
It follows that  we have the long exact sequence
\begin{equation*}
	\fontsize{9.3pt}{10pt}\selectfont
	\hspace*{-1.5cm} 
	\begin{tikzcd}[column sep=1.2em, row sep=1.7em]
		H^i(W_{k} A^{\bullet}(X,D,E))\arrow[r] \arrow[d,"\simeq"]&H^i({\rm Gr}^W_k  A^{\bullet}(X,D,E))\arrow[r]\arrow[d,"\simeq"]&H^{i+1}(W_{k-1}  A^{\bullet}(X,D,E))\arrow[d]\arrow[r]&	H^{i+1}(W_{k}	 A^{\bullet}(X,D,E))\arrow[d,"\simeq"]  \arrow[r] &H^{i+1}({\rm Gr}^W_k  A^{\bullet}(X,D,E))\arrow[d,"\simeq"] \\
		H^i(W_{k}  A^{\bullet}(X,D,E))\arrow[r] &H^i(A^{\bullet-k}_{D^{(k)}}(E) )\arrow[r]&H^{i+1}(W'_{k-1}  A^{\bullet}(X,D,E))\arrow[r]&	H^{i+1}(W_{k}A^{\bullet}(X,D,E)) \arrow[r]&H^{i+1}(A^{\bullet-k}_{D^{(k)}}(E) ) 
	\end{tikzcd}
\end{equation*}
By the Five Lemma, we obtain:
\begin{lem}\label{lem:fivelemma}
	The inclusion
	\[
	i_k : W_k A^{\bullet}(X,D,E) \hookrightarrow W_k' A^{\bullet}(X,D,E)
	\]
	between the two complexes is a quasi-isomorphism.
\end{lem}

We write $A^\bullet_{X,D}(E):=W_0'A^\bullet(X,D,E)$ for simplicity.  Then the natural morphism $A^\bullet(X,E)\to A^\bullet_{X,D}(E)$  induces an  isomorphism
\begin{align}\label{eq:iso}
	H^i(X,E)\to H^i(A^\bullet_{X,D}(E)).
\end{align}

For any $\beta\in H^q(D^{(1)},E)$,  we can find an element $\eta\in W_1A^{q+1}(X,D,E)$ such that $\Res_D\eta$ is $\nabla$-closed and  represents $\beta$. Then by the snake lemma, $\nabla\eta\in A^{q+2}_{X,D}(E)$ that is $\nabla$-closed. Hence it represents a cohomology class in $H^{q+2}(A^\bullet_{X,D}(E))\simeq H^{q+2}(X,E) $ 
\begin{dfn}[Gysin morphism]\label{def:Gysin}
	The map
	\begin{align*}
		g: H^q(D^{(1)},E)&\to H^{q+2}(X,E) \\
		\beta&\mapsto -\{\nabla\eta\}
	\end{align*} is called the \emph{Gysin morphism}.
\end{dfn} 
We will the following result.  
\begin{lem}\label{lem:Gysin}
	For any $\beta \in A^q(D^{(1)}, E)$ such that $\nabla\beta=0$,  we have
	\[
	g(\{\beta\}) = \{  \iota_*\beta \},
	\]
	where $\iota_*\beta \in \sD^{q+2}(X, E)$ denotes the $E$-valued current defined by the pushforward under the natural inclusion $\iota \colon D^{(1)} \to X$.
\end{lem}
\begin{proof}
	Choose  an element $\eta\in W_1A^{q+1}(X,D,E)$ such that $\Res_D\eta$ is $\nabla$-closed and  represents $\beta$. We then have  the following formula: 
	\[
	\bnabla \eta =\bnabla' \eta+\bnabla'' \eta  \stackrel{\eqref{eq:diffresi}}{=} \nabla \eta+ \iota_*(\beta).
	\]
	Since $\{\bnabla\eta\}=0$, it follows that 
	$$
	g(\{\beta\})= -\{\nabla \eta\}= \{ \iota_*(\beta)\}.
	$$The lemma is proved. 
\end{proof}
%

\subsection{A canonical lift of residue map}\label{setting}

Let $\pi_i:U_i\to D_i$ be a deformation retraction of some open neighborhood $U_i$ of $D_i$. Let $s_i$ be a canonical section of $\cO_X (D_i)$ and let $h_i$ be a smooth metric for $\cO_X (D_i)$. Let $\varrho_i$ be a smooth function on $X$ supported in $U_i$, that is equal to 1 in a neighborhood of $D_i$. 
Consider a flat section $\beta\in H^0(D^{(1)},E)$. Then we have
\begin{align}
	\nabla'\beta=0, \quad \nabla''\beta=0. 
\end{align} 
Since $\pi_i$ is a deformation retraction, it follows that $\pi_1(D_i)\to \pi_1(U_i)$ is an isomorphism. Therefore, $\beta$ extends to a flat section of $H^0(U_i,E)$, denoted by $\beta_i$.   
Since $\nabla\beta=0$, it follows that $\nabla\beta_i=0$. Therefore, one has 
\begin{align} \label{eq:lift}
	\frac{\sn}{2\pi }\nabla'(\sum_{i}\varrho_i \log |s_i|_{h_i}^2 \cdot \beta_i)	\in A^{1}(X,D,E)
\end{align}  
and  the residue  
$$
\Res_D\frac{\sn}{2\pi }\nabla'(\sum_{i}\varrho_i \log |s_i|_{h_i}^2 \cdot \beta_i)=-\beta. 
$$ 
Hence, 
$$
h(\beta):=\frac{\sn}{2\pi }\nabla'' \nabla'(\sum_{i}\varrho_i \log |s_i|_{h_i}^2 \cdot \beta_i)\in A^{2}(X,E)
$$ 
is $\nabla$-closed, and lies  in the cohomology class of $g(\beta)$.   

If we consider the covariant derivative in the sense of currents, then we have
$$
\frac{\sn}{2\pi }\bnabla'' \bnabla'(\sum_{i}\varrho_i \log |s_i|_{h_i}^2 \cdot \beta_i)=-\iota_*\beta+ h(\beta).
$$   
Set 
$$\gamma:=2\pi \sn(\nabla')^*G(\nabla'')^*G(h(\beta)).$$
Then $$\frac{\sn}{2\pi }\nabla''\nabla'\gamma=-(1-\cH)(h(\beta)),$$
where $\cH$ is the harmonic projection. 
Therefore, we have
$$
\frac{\sn}{2\pi }\bnabla'' \bnabla'(\sum_{i}\varrho_i \log |s_i|_{h_i}^2 \cdot \beta_i+\gamma)=-\iota_*(\beta)+\cH(h(\beta)). 
$$
Hence, we define  a map 
$$
f:H^0(D^{(1)},E)\to A^1(X,D,E)
$$
by setting 
\begin{align} \label{eq:f}
	f(\beta):=\frac{\sn}{2\pi  } \bnabla'(\sum_{i}\varrho_i \log |s_i|_{h_i}^2 \cdot \beta_i+ \gamma).   
\end{align} 
Then we have  
\begin{align} \label{eq:ff}
	\bnabla f(\beta)=\bnabla''f(\beta)=-\iota_*\beta+\cH(h(\beta))= -\iota_*\beta+\cH(g(\beta))
\end{align}
where $g$ is the Gysin morphism defined in \cref{def:Gysin},
and 
\begin{align} \label{eq:ffff}
	{\rm Res}_{D}(f(\beta))=-\beta. 
\end{align}
One can check that $f$ is a linear map.

\subsection{Regularity}
Note that although the complex $A^\bullet(X,D,E)$ is preserved under the Lie bracket, it is not preserved by Green operator $G$ or $\bnabla''^*$. Moreover, since the product of currents is generally ill-defined, extending the analytic constructions in \cref{subsec:compactuni} requires establishing additional regularity properties when applying these operators. In this subsection, we establish various regularity theorems used in the construction of \cref{thm:construction}.
\begin{lem}
	Let $\eta\in W_1A^2(X,D,E)$ such that  $\eta=\bnabla' \sigma$, where $\sigma\in \sD^1(X,E)$. Assume that 
	\begin{itemize}
		\item  	$$
		\nabla\eta=0
		$$ 
		\item  $\Res_D\eta\in A^1(D^{(1)},E)$ is $\nabla$-closed, such that its cohomology class in $H^1(D^{(1)},E)$ is trivial. 
	\end{itemize} 
	Then $\Res_D \eta=0$. In particular,  $\bnabla''\eta=\bnabla\eta=0$. 
\end{lem}
\begin{proof}
	Let $\delta:=\Res_D \eta\in A^1(D^{(1)},E)$. Then we have
	$$
	\bnabla\eta=\iota_*(\delta). 
	$$   
	Since $\bnabla'\bnabla \eta=0$, it follows that
	$$
	\bnabla'\bnabla\eta=\iota_*\nabla'\delta=0.
	$$
	Since $\iota_*:A^k(D^{(1)},E)\to \sD^{k+2}(X,E)$ is injective, it follows that
	$$
	\nabla'\delta=0.
	$$
	By the assumption, $\delta=\nabla h$ for some $h\in A^0(D^{(1)},E)$.  Therefore, we have
	$$
	\hess h=0. 
	$$
	This implies that
	$$
	\Lambda \hess h=\Delta' h=0.
	$$
	Hence, $\delta=\nabla h=0$. 
	Therefore, we have 
	$$
	\bnabla \eta=0, \quad \bnabla''\eta=0.
	$$ 
\end{proof}

\begin{lem}\label{lem:restriction}
	Let $\bB^n$ be the unit ball in $\bC^n$. Let $\sigma$ be a distribution on $\bB^n$ such that $$\partial\sigma=\sum_{i=1}^{m}a_i(z)\frac{dz_i}{z_i}+\sum_{j=m+1}^{n}a_j(z)dz_j.$$
	Assume that $a_i(z)$  is smooth for each $i\in \{1,\ldots,n\}$, and $a_i(z)|_{D_i}=0$ when restricted on the divisor $D_i:=(z_i=0)$ for each $i\in \{1,\ldots,m\}$. 
	Then $\sigma$ is a continuous function on $\bB^n$.  Moreover,  for each $i\in \{1,\ldots,m\}$, we have
	\begin{align}\label{eq:remove}
		\d (\sigma|_{D_i})=\big(\sum_{j=1,\ldots,m; j\neq i}a_j(z)\frac{dz_j}{z_j}+\sum_{j=m+1}^{n}a_j(z)dz_j\big)  |_{D_i}. 
	\end{align} 
\end{lem}
\begin{proof} 
	Write $\hat{z}_i:=(z_1,\ldots,z_{i-1},z_{i+1},\ldots,z_n)$.  Since $a_i$ is smooth and vanishes on $(z_i=0)$, the functions  defined by 
	$$f_i(\hat{z}_i) :=\frac{\partial a_i (z_1,\ldots,z_{i-1},0,z_{i+1},\ldots,z_n)}{\partial z_i},\quad  g_i(\hat{z}_i) := \frac{\overline{\partial} a_i (z_1,\ldots,z_{i-1},0,z_{i+1},\ldots,z_n)}{\partial \overline{z}_i},$$
	are smooth functions. 
	Then  we have 
	\begin{align}\label{eq:bound2}
		b_i(z):=a_i (z)-z_i f_i(\hat{z}_i)  - \bar{z}_ig_i(\hat{z}_i)  \in O(|z_i|^2). 
	\end{align} 
	Therefore, 
	\begin{align}\label{eq:bound}
		\frac{1}{z_i}\frac{\d b_i(z)}{\d \bar{z}_i}\in L^\infty(\bB^n_{\ep})
	\end{align} 
	for any $\ep\in (0,1)$. 
	Consider
	$$
	f:=\sigma-\sum_{i=1}^{m}\bar{z}_i \log |z_i|^2 g_i(\hat{z}_i)- z_if_i(\hat{z}_i), 
	$$
	then we have 
	\begin{align*}
		\d f&=\d\sigma-\sum_{i=1}^{m}(\bar{z}_i   g_i(\hat{z}_i)\frac{dz_i}{z_i}+\bar{z}_i \log |z_i|^2\d g_i+ f_i(\hat{z}_i)dz_i+z_i\d f_i)\\
		&=\sum_{i=1}^{m}( a_i(z)-g_i(\hat{z}_i)\bar{z}_i- z_if_i(\hat{z}_i))\frac{dz_i}{z_i}+\sum_{i=1}^{m}(\bar{z}_i \log |z_i|^2\d g_i+ z_i\d f_i)+ \sum_{j=m+1}^{n}a_j(z)dz_j\\
		&	 =\sum_{i=1}^{m}b_i(z)\frac{dz_i}{z_i}+\sum_{i=1}^{m}(\bar{z}_i \log |z_i|^2\d g_i+ z_i\d f_i)+ \sum_{j=m+1}^{n}a_j(z)dz_j.
	\end{align*} 
	Hence 
	\begin{align*}
		\Delta'f&=\Lambda\db\d f=\sum_{i=1}^{m}\frac{\d b_i(z)}{\d \bar{z}_i}\frac{dz_i}{z_i}+\Lambda\sum_{i=1}^{m}\db(\bar{z}_i \log |z_i|^2\d g_i+ z_i\d f_i)+ \Lambda\db(\sum_{j=m+1}^{n}a_j(z)dz_j)\\
		&=\sum_{i=1}^{m}\frac{\d b_i(z)}{\d \bar{z}_i}\frac{dz_i}{z_i} + {\mbox{ smooth terms}}.
	\end{align*} 
	It follows from \eqref{eq:bound} that $\Delta'f\in L^\infty(\bB^n_{\ep})$. Therefore, 
	$ 
	f\in C^{1,\alpha}(\bB^n)
	$ for any $\alpha\in (0,1)$. This implies that
	$$
	\sigma-\sum_{i=1}^{m}\bar{z}_i \log |z_i|^2 g_i(\hat{z}_i)\in C^{1,\alpha}(\bB^n).
	$$
	In particular, $\sigma$ is a continuous function. 
	
	Write 
	$$
	h_i(z):=\sigma- \bar{z}_i \log |z_i|^2 g_i(\hat{z}_i) - z_if_i(\hat{z}_i),
	$$
	which is a continuous function. Note that 
	$$
	\sigma|_{D_i}=h_i|_{D_i}. 
	$$ 
	Since we have 
	\begin{align*}
		\d h_i&  =\sum_{j=1,\ldots,m; j\neq i}a_j(z)\frac{dz_j}{z_j}+\sum_{j=m+1}^{n}a_j(z)dz_j 	\\
		&+ ( a_i-g_i(\hat{z}_i)\overline{z_i}- z_if_i(\hat{z}_i))\frac{dz_i}{z_i} 
		+ (\bar{z}_i \log |z_i|^2\d g_i+ z_i\d f_i)\\
		& =\sum_{j=1,\ldots,m; j\neq i}a_j(z)\frac{dz_j}{z_j}+\sum_{j=m+1}^{n}a_j(z)dz_j 	\\
		&+ \frac{b_i(z)}{z_i}dz_i + (\bar{z}_i \log |z_i|^2\d g_i+ z_i\d f_i)\\
	\end{align*} 
	By \eqref{eq:bound2},  
	$
	\frac{b_i(z)}{z_i} 
	$ is continuous and $ 
	\frac{b_i(z)}{z_i} 
	|_{D_i}=0$. It follows that 
	$$
	\d (\sigma|_{D_i})=\d ( h_i|_{D_i})= \big(\sum_{j=1,\ldots,m; j\neq i}a_j(z)\frac{dz_j}{z_j}+\sum_{j=m+1}^{n}a_j(z)dz_j\big)  |_{D_i}. 
	$$
	The lemma is proved. 
\end{proof}

\begin{lem}\label{lem:continuous}
	Let $\eta \in A^1(X,D,E)$ be such that $\Res_D\eta=0$. If there exists some $f\in \sD^0(X,E)$ such that $\bnabla' f=\eta$, then $f$ is smooth on $X_0$ and $f\in A^0_{\mathrm{cts}}(X, E)$.  Moreover, for any $D_i$, we have
	$$
	\bnabla' (f|_{D_i})\in A^1(D_i,\sum_{j\neq i}D_j\cap D_i,E). 
	$$
	The following property is satisfied: For any coordinate system $(U;z_1,\ldots,z_n)$ such that $D\cap U=(z_1\cdots z_k=0)$ with $D_i\cap U=(z_1=0)$, if we write 
	$$
	\eta|_{U}=  \sum_{j=1,\ldots,k}a_j(z)\frac{dz_j}{z_j}+\sum_{j=k+1}^{n}a_j(z)dz_j,
	$$
	then we have
	\begin{align}\label{eq:restrictionabu2}
		\bnabla' (f|_{D_i})|_{U\cap D_i}= \big(\sum_{j=2,\ldots,k}a_j(z)\frac{dz_j}{z_j}+\sum_{j=k+1}^{n}a_j(z)dz_j\big)  |_{(z_1=0)}.
	\end{align} 
\end{lem}
In the above setting, we shall abusively write 
\begin{align}\label{eq:restrictionabu}
	\eta|_{D_i}:=	\bnabla' (f|_{D_i})\in A^1(D_i,\sum_{j\neq i}D_i\cap D_j,E). 
\end{align}
Then for any $j\neq i$,  by \eqref{eq:restrictionabu2}, we have
\begin{align}\label{eq:2residue}
	\Res_{D_i\cap D_j} (\eta|_{D_i})=	(\Res_{D_j}\eta)|_{D_i\cap D_j}=0.
\end{align}
\begin{proof}[Proof of \cref{lem:continuous}]
	Fix any point $p\in D$. We take admissible coordinates $(U;z_1,\ldots,z_n)$ centered at $p$ such that $U\cap D=(z_1\cdots z_m=0)$.   Let $e_1,\ldots,e_r$ be a flat frame for $E|_U$.  Then $$\nabla''e_i=\nabla'e_i=0.$$ 
	Let us write $f=\sum_i f_i e_i $  and $\eta=\sum_i \eta_i e_i$
	where $f_i\in \sD^0(U)$ and $\eta_i\in \cA^1(X,D)(U)$.  Then we have 
	$$
	\bd f_i=\eta_i.  
	$$
	We write 
	$$\eta_i=\sum_{j=1}^{m}a_j(z)\frac{dz_j}{z_j}+\sum_{j=m+1}^{n}a_j(z)dz_j$$
	such that   $a_j(z)$  is a smooth function for each $j$.    Since $\Res_D\eta=0$, it follows that $a_j(z)|_{(z_j=0)}\equiv 0$  for each $j\in \{1,\ldots,m\}$.   By \cref{lem:restriction}, $f_i$ is continuous, and hence  $f$ is continuous.   The second assertion follows from \eqref{eq:remove}.  The lemma is proved. 
\end{proof}

The following criterion determines when the Lie bracket of two log forms is $\bnabla'$-exact.
\begin{lem}\label{lem:exact0}
	Let $\eta_1,\eta_2\in A^1(X,D,E)$. Assume that  $\eta_1$ is $\bnabla'$-closed and
	$\eta_2$ is $\bnabla'$-exact. Then  $[\eta_1, \eta_2]$ is $\bnabla'$-exact. 
\end{lem}
\begin{proof}
	
	As $\eta_2$ is $\bnabla'$-exact, we have $\eta_2 =\bnabla'\sigma $ for some current $\sigma$. As $\sigma$ is of degree $0$, we have 
	$$(\bnabla' + (\bnabla')^* )\sigma = \eta_2 .$$
	As $\eta_2 \in L^{2-\ep}$ for any $\ep >0$, by the elliptic regularity of $\bnabla' + (\bnabla')^*$, we know that $\sigma \in L^{2+c_0}$ for some $c_0 >0$ and $\sigma$ is smooth on $X\backslash D$.
	Since $\eta_1$ has logarithmic poles, the bracket product
	\[
	[\eta_1, \sigma] \in L^1 ,
	\]
	is a well-defined current in $\sD^1(X, E)$.
	
	Now we show that 
	\[
	\bnabla' [\eta_1, \sigma] =  - [\eta_1, \eta_2] \qquad\text{ on } X
	\]
	in the sense of currents.
	Let $\mu_\ep$ be the standard cut-off functions with respect to $X\setminus D$, namely 
	$$\mu_\ep = \rho_\ep (\log \log \frac{1}{|s_D|^2_h}) ,$$
	where $s_D$ is the canonical section of $\mathcal{O}_X (D)$, $h$ is a smooth metric, $\rho_\ep$ is a function which is equal to $1$ on $[1, \ep^{-1}]$ and is equal to $0$ on $[1+ \ep^{-1}, +\infty[$ and $|\rho_\ep'| \leq 2$.
	
	Let $\beta \in A^{2n-2} (X, E^*)$. Then
	$$\int_X [\eta_1, \sigma] \wedge (\bnabla' \beta ) = \lim_{\ep \to 0} \int_X \mu_\ep \cdot [\eta_1, \sigma] \wedge (\bnabla' \beta ) $$
	$$= -\lim_{\ep \to 0} \int_X \partial \mu_\ep \wedge [\eta_1, \sigma] \wedge \beta - \lim_{\ep \to 0} \int_X\mu_\ep \cdot \bnabla' [\eta_1, \sigma] \wedge \beta .$$
	We now compute the two terms on the right-hand side. Note that on the support of $\mu_\ep$, $\eta_1$ and $\sigma$ are smooth. We have thus $\mu_\ep \cdot \bnabla' [\eta_1, \sigma]= -\mu_\ep [\eta_1, \eta_2]  $. 
	Therefore we have 
	$$-\lim_{\ep \to 0} \int_X\mu_\ep \cdot \bnabla' [\eta_1, \sigma] \wedge \beta  =\lim_{\ep \to 0} \int_X\mu_\ep [\eta_1, \eta_2] \wedge \beta  = \int_X [\eta_1, \eta_2] \wedge \beta .$$
	For the other term $\lim_{\ep \to 0} \int_X \partial \rho_\ep \wedge [\eta_1, \sigma] \wedge \beta$, note that 
	$$\partial \mu_\ep \wedge \eta_1 = \rho'_\ep \cdot \frac{1}{\log \frac{1}{|s_D|^2_h}} ( \partial (\log |s_D|^2_h )\wedge \eta_1 ).$$
	As $ \partial (\log |s_D|^2_h )\wedge \eta_1 \in A^1 (X, D, E)$ and  $\sigma \in L^{2+c_0}$ for some $c_0 >0$, 
	we know that 
	$$\partial (\log |s_D|^2_h )\wedge \eta_1 \cdot \sigma \in L^1 .$$ 
	Then by Lebesgue's dominated convergence theorem, we have 
	$$\lim_{\ep \to 0} \int_X \partial \mu_\ep \wedge [\eta_1, \sigma] \wedge \beta =0.$$
	
	As a consequence, we obtain 
	$$\int_X [\eta_1, \sigma] \wedge (\bnabla' \beta ) =  - \int_X [\eta_1, \eta_2] \wedge \beta . $$
	Thus, $[\eta_1, \eta_2] = \bnabla' \left(-[\eta_1, \sigma]\right)$ in the sense of currents.  The lemma is thus proved.
\end{proof}

\begin{cor}\label{corexact}
	Let $\eta_1, \eta_2 \in A^1 (X, D,E)$. We assume that $\eta_1$ is $\bnabla'$-closed and $\eta_2$ is $\bnabla'$-exact, and $\Res_{D} \eta_2=0$. Then $\Res_{D} (\eta_1 \wedge \eta_2)$ is $\bnabla'$-exact. 
\end{cor} 
\begin{proof}
	By \cref{lem:continuous} and \eqref{eq:restrictionabu2},  we can check that,  for each $D_i$,   we have 
	\begin{align}\label{eq:resiwedge}
		\Res_{D_i} (\eta_1 \wedge \eta_2)=   \Res_{D_i} (\eta_1) \wedge (\eta_2 |_{D_i}),
	\end{align} 
	where $\eta_1 |_{D_i}$ and  $\eta_2 |_{D_i}$ are defined in \eqref{eq:restrictionabu}. 
	By \cref{lem:continuous} again, $\eta_2 |_{D_i}$   is $\bnabla'$-exact. It follows that  
	$$
	\Res_{D_i} (\eta_1) \wedge (\eta_2 |_{D_i})
	$$
	is also $\bnabla'$-exact. The corollary is thus proved. 
\end{proof}

\begin{proposition}\label{prop:regularity}
	Let $\eta\in  A^2(X,D,E)$ such that $\eta$ is $\bnabla$-closed and $\bnabla'$-exact in the sense of currents. If we define 
	$$
	\omega :=\bnabla'\bnabla'^*G\bnabla''^* G \eta=\bnabla''^* G \eta, 
	$$ 
	then
	$\omega \in  A^1 (X,D,E)$, and $\omega$ is the unique element in $\sD^1(X,E)$ that is $\bnabla'$-exact, and $$\bnabla\omega=\bnabla''\omega=\eta.$$ 
\end{proposition}

\begin{proof} 
	It is straightforward to see that 
	$$
	\omega= \bnabla'\bnabla'^*G\bnabla''^* G \eta=(1-\cH)\bnabla''^* G \eta-\bnabla'^*\bnabla'G\bnabla''^* G \eta=\bnabla''^* G \eta.
	$$

	\medskip

	If $\bnabla  \bnabla' f=0$ for some $f\in \sD^0(X,E)$, then by elliptic regularity, $f$ is smooth. So $f$ is pluri-harmonic. By the compactness of $X$, we know that $\bnabla' f=0$. This in particular shows the uniqueness of $\omega$. 
	
	\medskip
	
	Now we would like to show that $\omega \in  A^1 (X,D,E)$. Write $f:=\bnabla''^* G \eta.$
	Fix any point $p\in D$. We take an admissible  coordinate $(U;z_1,\ldots,z_n)$ centered at $p$ such that $U\cap D=(z_1\cdots z_m=0)$.   Let $e_1,\ldots,e_r$ be a flat frame for $E|_U$.  Then $$\nabla''e_i=\nabla'e_i=0$$ 
	Let us write $f=\sum_i f_i e_i $  and $\eta=\sum_i \eta_i e_i$
	where $f_i\in \sD^0(U)$ and $\eta_i\in \cA^2(X,D)(U)$.  Then we have 
	$$
	\bdb\bd f_i=\eta_i. 
	$$
	Therefore, $\eta_i\in \cA^{1,1}(X,D)(U)$. 
	
	Note that $(\cA^{p,\bullet}(X,D),\db)$ is a fine resolution for $\Omega_X^p(\log D)$. It follows that there exists some $h_i\in \cA^{1,0}(X,D)(U)$   such that
	$$
	\db h_i=\eta_i.
	$$ 
	Since $\Res_D\eta_i=0$, it follows that $\Res_{D_j} \db h_i=0$, which implies that
	$$
	\db\Res_{D_j} h_i=0. 
	$$ 
	Hence $ \Res_{D_j} h_i$  is a holomorphic function on $D_j\cap U$. Define a holomorphic function  $g_{i,j}\in \cO(U)$ such that $g_{i,j}|_{D_j}=\Res_{D_j} h_i$. We replace $h_i$ by 
	$$
	h_i-\sum_{j=1}^{m}g_{i,j}\frac{d z_j}{z_j}. 
	$$
	Then $\Res_D h_i=0$. It follows  from \cref{lem:equiv} that
	$$
	\bdb h_i=\db h_i=\eta_i. 
	$$
	Therefore, we have
	$$
	\bdb(h_i-\bd f_i)=0.
	$$
	
	By the regularity of $\db$-operator, it follows that $h_i-\bd f_i$ is a holomorphic 1-form. Hence $\bd f_i\in \cA^1(X,D)(U)$.  Therefore, we obtain 
	$  \omega = \bnabla' f\in A^1(X,D,E).$
\end{proof}

Following  \cite{CDHP25, LRW}, we prove an effective estimate of the solution in Proposition \ref{prop:regularity}. 
To begin with, we consider the vector bundle 
$$E_p := \Omega_X ^p (\log D) \otimes E$$
on $X$. We have an identification 
$$\phi:  A^i (X, D, E) \simeq \oplus_{p+q =i} C^\infty _{(0,q)} (X, E_p)$$
We equip the vector bundle $E_p$ with a fixed smooth hermitian metric $h_p$, and we fix a Kähler metric $\omega_X$ on $X$. For any constant $k\in \mathbb N$, and any $\alpha \in A^i (X, D, E)$, we define
\begin{equation}\label{normsob}
	\|\alpha\|_{W^k} := \|\phi (\alpha)\|_{W^k} ,
\end{equation}
where $\|\phi (\alpha)\|_{W^k}$ is the norm of the Sobolev space whose derivatives up to order $k$ are $L^2$ with respect to $h_p$ and $\omega_X$. We have the following effective version of Proposition \ref{prop:regularity}.
\begin{cor}\label{effectiveest}
	Let $\eta\in  A^2(X,D,E)$ such that $\eta$ is $\bnabla$-closed and $\bnabla'$-exact in the sense of currents.  Fix  $k\in \mathbb N^* $.  
	Let $\omega\in A^1(X,D,E)$ be the unique $\bnabla'$-exact solution of
	$$\bnabla \omega =\eta \qquad\text{ on } X $$
	as in Proposition \ref{prop:regularity}. Then there exists a constant $C$ (depending only on $k$) such that 
	$$\|\omega\|_{W^k} \leq C \|\eta\|_{W^k} .$$
\end{cor}

\begin{proof}
	We follow ideas from the proof of \cite[Thm 8.7]{CDHP25}. Suppose by contradiction that there is no such uniform constant $C$. Then there exists a sequence of $\eta_i \in A^2 (X, D, E)$ and $\omega_i \in A^1 (X, D, E)$ such that $\omega_i \in \Im \bnabla'$,
	$$\bnabla \omega_i =\eta_i \qquad\text{ on } X, $$
	$\|\omega_i\|_{W^k}=1$ for every $i$, and $\lim_{i\to +\infty}\|\eta_i\|_{W^k}\to 0 $.
	
	After passing to some subsequence, we may assume that $\omega_i$ converges weakly to some $\omega$ in $W^k$. Then $\bnabla\omega=0$, and by the closeness of $\Im \nabla'$ in the Hodge decomposition, it follows that $\omega \in \Im \nabla'$. Therefore $\omega=0$.
	
	For every open set $U \Subset X \setminus D$, by elliptic regularity, we know that $\omega_i |_{U}$ converges strongly to $\omega |_{U} =0$ in $W^k (U)$. Near the boundary $D$, we suppose that $D$ is defined locall by $\prod_{s\in I} z_s =0$. Then $\omega_i$ is locally of type 
	$$\omega_i =\sum_{s\in I} \frac{dz_s}{z_s} f_{i, s} + g_{i,s} ,$$ 
	where $g_{i,s}$ is $(0,1)$-type or $(1,0)$-type with $dz_s$ for $s\notin I$.
	By using the fact that $\bnabla' \omega_i =0$ and $\bnabla \omega_i \to 0$, we know that 
	$$\db f_{i, s} \to 0$$
	in the $W^k$ sense.
	Together with the fact that $f_{i, s} \to 0$ on $U \Subset X \setminus D$ in $W^k$, we know that $f_{i, s}\to 0$ near $D$ in the $W^k$ sense. For the term $g_{i,s}$, by the same argument, we know also that $g_{i,s}\to 0$ near $D$ in the $W^k$ sense. As a consequence, we know that $\omega_i \to 0$ in $W^k$. This contradicts the fact that $\|\omega_i\|_{W^k}=1$ for every $i$.
\end{proof}

\begin{lem}\label{lem:exact 2}
	Consider $\eta_1,\eta_2\in A^1(X,D,E)$. Assume that
	\begin{itemize}
		\item   $\bnabla \eta_i=\nabla\eta_i$  for $i=1,2$.
		\item $\eta_1,\eta_2$ are both $\bnabla'$-exact.
	\end{itemize}    Then 
	\begin{thmlist}
		\item \label{item:zero residue}$\bnabla[\eta_1,\eta_2]=\nabla[\eta_1,\eta_2],$ or equivalently, $\Res_D[\eta_1,\eta_2]=0$. 
		\item   $[\eta_1,\eta_2]$ is $\bnabla'$-exact. 
	\end{thmlist}
\end{lem}
\begin{proof}
	Fix any point $p\in D$. We take admissible  coordinates $(U;z_1,\ldots,z_n)$ centered at $p$ such that $U\cap D=(z_1\cdots z_m=0)$.   With these coordinates, we can write 
	$$\eta_1=\sum_{i=1}^{m}  \frac{dz_i}{z_i} \otimes a_i +\sum_{j=m+1}^{n}dz_j\otimes e_j$$
	$$\eta_2=\sum_{i=1}^{m}  \frac{dz_i}{z_i} \otimes b_i +\sum_{j=m+1}^{n}dz_j\otimes f_j$$
	where $a_i,b_i,e_j,f_j\in A^0(U,E)$.   Since $\bnabla \eta_i=\nabla\eta_i$, by \Cref{lem:equiv}, we have  $\iota_i^*a_i=0$ and  $\iota_i^*b_i=0$ for any $i\in \{1,\ldots,m\}$. Here $\iota_i:D_i\to X$ is the inclusion map.    Note that 
	\begin{align*} 
		[\eta_1,\eta_2]&=\sum_{1\leq i<j\leq m} ([a_i,b_j]-[a_j,b_i])\frac{dz_i}{z_i}\wedge \frac{dz_j}{z_j} \\  
		&+\sum_{1\leq i\leq m<j\leq n} ([a_i,f_j]-[e_j,b_i])  \frac{dz_i}{z_i}\wedge dz_j +\sum_{m\leq  i<j\leq n}   ([e_i,f_j]-[e_j,f_i]) dz_i\wedge dz_j 
	\end{align*}
	It follows that $\Res_D[\eta_1,\eta_2]=0$.  This implies that $\bnabla[\eta_1,\eta_2]=\nabla[\eta_1,\eta_2]$ by \cref{lem:equiv}.   The first item is proved.  
	
	By \cref{lem:continuous} and our assumption, there exists $\sigma_i\in A^0_{\mathrm{cts}}(X, E)$ such that $\bnabla'\sigma_i=\eta_i$. Therefore,  $[\sigma_1,\eta_2]\in \sD^1(X,E)$, and by the proof of \cref{lem:exact0}, we  have
	$$
	[\eta_1,\eta_2]=\bnabla'[\sigma_1,\eta_2]. 
	$$
	Hence $[\eta_1,\eta_2]$ is $\bnabla'$-exact.  The lemma is thus proved. 
\end{proof}

\begin{lem}\label{lem:residue}
	Let $\eta=f(\beta)\in A^1(X,D,E)$ for some $\beta\in H^0(D^{(1)},E)$, where $f$ is defined in \eqref{eq:f}.   Assume also that for any $i<j$, we have 
	\begin{align}\label{eq:commuteresidue}
		[\iota_{ij}^*(\beta|_{D_i}),\iota_{ji}^*(\beta|_{D_j})]=0,
	\end{align}  
	where we denote by $\iota_{ij}:D_{ij}\to D_i$  the natural map.  
	Then there exists some $h\in \sD^0(D^{(1)},E)$ such that
	\begin{align}\label{eq:1}
		\Res_{D}[\eta,\eta]&=\bnabla'h\\ \label{eq:2}
		\bnabla  \Res_{D}[\eta,\eta]&= 2[\beta,\iota^*\nabla \eta]. 
	\end{align} 
	In particular, if $g(\beta)=0$, where $g$ is the Gysin morphism, then we have
	\begin{align*}
		\Res_{D}[\eta,\eta]&=0\\
		\bnabla[\eta,\eta]&=0.
	\end{align*} 
\end{lem}
\begin{proof} 
	We keep the notations from \Cref{setting} and do not recall their meaning. By   \eqref{eq:ff}, there exists some $\gamma\in A^0(X,E)$ such that we have
	\begin{align*} 
		f(\beta):= \frac{\sn}{2\pi}\bnabla'(\sum_{i=1}^{m}\varrho_i \log |s_i|_{h_i}^2 \cdot \beta_i+\gamma)\\\nonumber
		=\frac{\sn}{2\pi}(\sum_{i=1}^{m}\d(\varrho_i \log |s_i|_{h_i}^2)   \beta_i+\nabla'\gamma)
	\end{align*}
	We then have
	$$
	-4\pi^2 [\eta,\eta]=\sum_{i\neq j}\d(\varrho_i \log |s_i|_{h_i}^2)\wedge  \d(\varrho_j \log |s_j|_{h_j}^2)   [\beta_i ,  \beta_j]+ 2 [\sum_{i=1}^{m}\d(\varrho_i \log |s_i|_{h_i}^2)   \beta_i,\nabla'\gamma]+[\nabla'\gamma,\nabla'\gamma].
	$$
	It follows that
	\begin{align}\nonumber
		\Res_{D_i} [\eta,\eta]&=\frac{-\sn}{\pi}\left(  [\beta|_{D_i} , \iota_i^*( \sum_{j\neq i}   \d(\varrho_j \log |s_j|_{h_j}^2)  \beta_j+\nabla'\gamma)]\right)\\\nonumber
		&=\frac{-\sn}{\pi}\left(  [\beta|_{D_i} , \iota_i^*( \sum_{j=1}^{m}   \d(\varrho_j \log |s_j|_{h_j}^2)  \beta_j+\nabla'\gamma)]\right)\\ \label{eq:preciseres2}
		&=\frac{-\sn}{\pi}\left(  [\beta|_{D_i} , \iota_i^*\eta]\right)\\ \label{eq:exactbrack}
		& =\frac{-\sn}{\pi}\bnabla'\left(  [\beta|_{D_i} , \iota_i^*( \sum_{j=1}^{m}    \varrho_j \log |s_j|_{h_j}^2   \beta_j+ \gamma)]\right)
	\end{align}  
	This proves \eqref{eq:1}. 
	
	Let us prove \eqref{eq:2}.   Recall that we have assumed that $\Res_{D_{ij}}[\eta,\eta]=0$ for any $1\leq i<j\leq m$.  By \eqref{eq:preciseres2} and \eqref{eq:commuteresidue}, this implies that for any fixed $j\neq i$, we have
	$$
	\Res_{D_j\cap D_i}\Res_{D_i}[\eta,\eta]=  2  [\iota_{ij}^*(\beta|_{D_i}),\iota_{ji}^*(\beta|_{D_j})]=0
	$$
	for each $j>i$.  \cref{lem:equiv} together with \eqref{eq:exactbrack} imply that
	$$
	\bnabla \Res_{D_i}[\eta,\eta]=\bnabla''\Res_{D_i}[\eta,\eta]=\nabla''\Res_{D_i}[\eta,\eta]= 2[\beta,\iota^*\nabla \eta]. 
	$$
	\Cref{eq:2} is thus proved. 
	
	If $g(\beta)=0$, then by the construction of $f$, we have $\nabla\eta=0$ and  $\bnabla\eta=\iota_*(\beta)$.  By  \eqref{eq:2}, $$
	\bnabla  \Res_{D}[\eta,\eta]=0. 
	$$ By \eqref{eq:1},  
	$$
	\Delta'h=\sn \Lambda\bhess h=\sn \Lambda\bnabla\bnabla'h =0. 
	$$
	Hence $$\Res_{D}[\eta,\eta]=\bnabla'h=0. $$ The last claim is proved. 
	The lemma  is proved. 
\end{proof}

	\subsection{Explicit construction of the deformation}
	\begin{thm}\label{thm:construction}
		Let $X$ be a compact K\"ahler manifold and let $D$ be a simple normal crossing divisor on $X$.  Let $(V,\nabla)$ be a unitary bundle on $X$. Denote by $(E,\nabla):=(\End(V),\nabla)$.      Assume that there exist  $\eta_1\in \cH^1(X,E)$ and $\beta\in H^0(D^{(1)},E)$ (it might happen that  $\eta_1=0$ or $\beta=0$) such that  
		\begin{enumerate}[label=(\alph*)]
			\item\label{item:condition1}  $g(\beta)+\frac{1}{2}\{[\eta_1,\eta_1]\}=0$, where $g:H^0(D^{(1)},E)\to H^2(X,E)$ is the Gysin morphism. 
			\item \label{item:condition2}$ \{[\iota^*\eta_1,\beta]\}=0\in H^1(D^{(1)},E)$, where $\iota:D^{(1)}\to X$ is the natural morphism.
			\item \label{item:condition3}For any $i\neq j$,  $ [\iota_{ij}^*(\beta|_{D_i}),\iota_{ji}^*(\beta|_{D_j})]=0\in H^0(D_i\cap D_j,E)$, where $\iota_{ij}:D_i\cap D_j\to D_i$ is the natural morphism.
		\end{enumerate}
		
		Then 
		\begin{thmlist}
			\item\label{item:nabla'exact} for every $i\geq 2$, there exists $\eta_i \in A^1(X,D,E) \cap \Im \bnabla'$ such that they solve the following equations: 
			\begin{itemize}
				\item $\Res_D\eta_2=-\beta$ and 
				\begin{align}\label{eq:diffeta2}
					\bnabla \eta_2 = -\iota_*(\beta) - \eta_1\wedge\eta_1 
				\end{align} 
				
				\item \label{item:etak}For every $k\geq 3$, we have 
				$$ 
				\eta_k:=- \bnabla''^*G(\sum_{i=1}^{k-1}\eta_i\wedge\eta_{k-i}),  
				$$
				such that 
				\begin{align}\label{eq:maurercartan}
					\bnabla \eta_k = -\sum_{i=1}^{k-1}\eta_i\wedge \eta_{k-i}. 
				\end{align}
			\end{itemize}
			In particular, we have \begin{align}\label{eq:residuenull}
				\Res_D\eta_k=0\quad \mbox{for } k\geq 3.
			\end{align}
			\medskip
			\item \label{item:2step}	There exists some $\ep>0$ such that for any  $t\in \bD_\ep$, we have 
			$$\sum_{i=1}^{\infty}t^i\eta_i \in A^1(X,D,E) .$$ 
			As a consequence, we have an \emph{analytic} familly of \emph{logarithmic connections}    $ \nabla_t$ for $V$ parametrized by   $t\in \bD_\ep$ in the following explicit  way:
			$$
			\nabla_t:=\nabla+\sum_{i=1}^{\infty}t^i\eta_i,
			$$ 
			with its $(0,1)$-part given by $$\nabla_t''=\nabla''+t\eta_1^{0,1}.$$
			\item \label{item:pullback2step}Let $f: Y \to X$ be a holomorphic map from a compact Kähler manifold $Y$ such that $D_Y := f^{-1}(D)$ is a simple normal crossing divisor. Let $f_0$ denote the restriction $f|_{Y_0}$. If $f^*\eta_1 = 0$ and $f^*\eta_2 = 0$, then $f_0^*\nabla_t = f_0^*\nabla$ for all $t \in \mathbb{D}_\ep$. 
		\end{thmlist}
	\end{thm}
	\begin{proof}
		\noindent {\bf Step 1.} \emph{Construction $\eta_2$}. 
		
		Consider the map \begin{align} \label{eq:f2}
			f:H^0(D^{(1)},E)\to A^1(X,D,E)
		\end{align} 
		defined in \eqref{eq:f}. Then by \eqref{eq:ff}, we have 
		$\Res_D f(\beta)=-\beta$, $f (\beta)$ is $\bnabla'$-exact and 
		$$
		\bnabla f(\beta)=-\cH(g(\beta))-\iota_*(\beta). 
		$$   We define
		$$
		\eta_2:=  f(\beta)-\frac{1}{2}\nabla''^*G [\eta_1,\eta_1]. 
		$$
		Then we have 
		$$\bnabla \eta_2 = \cH(g(\beta))-\iota_*(\beta)+(\cH-1)\frac{1}{2} [\eta_1,\eta_1].$$
		\Cref{item:condition1} implies that, 
		\begin{align}\label{eq:second}
			\bnabla \eta_2 =-\iota_*(\beta)-\frac{1}{2} [\eta_1,\eta_1].
		\end{align}
		It remains to check that $\eta_2 \in A^1(X,D,E)$. Indeed, 
		$$\bnabla (\eta_2 - f (\beta)) = -\frac{1}{2} [\eta_1,\eta_1] + \cH(g(\beta)) \in A^2 (X,E).$$
		Note that  $\eta_2 - f (\beta) \in \Im \nabla'\cap  A^1(X,E)$. Therefore $\eta_2 \in A^1(X,D, E)$. 
		Since $f(\beta)$ is $\bnabla'$-exact, and 
		$$
		\nabla''^*G [\eta_1,\eta_1]= \nabla'\nabla'^*G \nabla''^*G [\eta_1,\eta_1],
		$$
		it follows that $\eta_2$ is $\bnabla'$-exact.

		\noindent {\bf Step 2.} \emph{Constructing $\eta_3$}. 
		Since $\eta_2 \in \Im \bnabla'$, and $\eta_1$ is harmonic, we know that $$[\eta_1,\eta_2] \in \Im \bnabla' \cap W_1A^2(X,D,E) . $$
		
		Note that
		\begin{align}
			\nabla[\eta_1,\eta_2]=\frac{1}{2}[\eta_1,[\eta_1,\eta_1]]=0
		\end{align}
		by the Jacobian identity.   We also have
		$$
		\Res_D[\eta_1, \eta_2]=[\iota^*\eta_1,\beta] =0\in H^1(D^{(1)},E) ,
		$$
		where the last equality comes from \Cref{item:condition2}.
		Then by \cref{lem:equiv}, $[\eta_1,\eta_2] \in \Im \bnabla' \cap \ker \bnabla \cap A^2(X,D,E)$.
		
		Let us define
		$$
		\eta_3:=-\bnabla''^*G[\eta_1,\eta_2]. 
		$$	Applying \cref{lem:ddbar} and \Cref{prop:regularity} to $[\eta_1,\eta_2]$, it follows that $\eta_3 \in A^1 (X,D,E) \cap \Im \bnabla'$, and that
		\begin{align}\label{eq:second2}
			\bnabla \eta_3 =-[\eta_1,\eta_2].
		\end{align} 
		In particular, this implies that
		\begin{align}\label{eq:residue}
			\Res_D[\eta_1,\eta_2] =0.
		\end{align}
		
		\medspace
		
		\noindent {\bf Step 3.} \emph{Constructing $\eta_4$}. 
		
		We consider  $\sum_{i=1}^3 \eta_{i} \wedge \eta_{4-i}$. Thanks to Lemma \ref{lem:exact0}, $\sum_{i=1}^3 \eta_{i} \wedge \eta_{4-i}$ is $\bnabla'$-exact.
		
		Now we show that $\sum_{i=1}^3 \eta_{i} \wedge \eta_{4-i}$ is $\bnabla$-closed. 
		\begin{align}
			\nabla(\frac{1}{2}[\eta_2,\eta_2]+[\eta_1,\eta_3])=[\eta_2,\eta_2]+[\eta_1,\eta_3]= 
			-\frac{1}{2}[\eta_2,[\eta_1,\eta_1]]+[\eta_1,-[\eta_1,\eta_2]]=0
		\end{align} 
		By the Jacobian identity, we have
		\begin{align}
			[\eta_2,[\eta_1,\eta_1]] - [\eta_1,[\eta_2,\eta_1]]+ [\eta_1,[\eta_1,\eta_2]]=0
		\end{align}
		By \cref{lem:equiv}, $\Res_{D} \eta_3=0$. Hence
		$$
		\Res_D [\eta_1,\eta_3]=0.
		$$
		Hence we have
		\begin{align}
			\bnabla(\frac{1}{2}[\eta_2,\eta_2]+ [\eta_1,\eta_3])= \frac{1}{2}\iota_*(\Res_D[\eta_2,\eta_2]).
		\end{align}

		Using \Cref{item:condition3} and by \Cref{lem:residue}, we know that $\Res_{D_i}[\eta_2,\eta_2] \in \Im \bnabla' $ and
		\begin{align} \nonumber
			\bnabla  \Res_{D_i}[\eta_2,\eta_2]&= -2[\beta_i,\iota^*\nabla \eta_2]\\ \nonumber
			&	=  [\beta_i,\iota^*[\eta_1,\eta_1]]\\ \label{eq:forkernel}
			&	=2[\iota^*\eta_1,[\beta_i,\iota^*\eta_1]]\\\nonumber
			&	\stackrel{\eqref{eq:residue}}{=}0. \nonumber
		\end{align} 
		Hence we have $\Res_{D_i}[\eta_2,\eta_2] \in \Im \bnabla' \cap \ker\bnabla$. By \cref{lem:ddbar} and the fact that $\Res_{D_i}[\eta_2,\eta_2]$ is a $1$-form, we know that $\Res_{D_i}[\eta_2,\eta_2]=0$. 
		Therefore, we have
		\begin{align}
			\bnabla(\sum_{i=1}^3 \eta_{i} \wedge \eta_{4-i})= 0.
		\end{align}
		Since $\sum_{i=1}^3 \eta_{i} \wedge \eta_{4-i}$ is $\bnabla'$-exact, we have
		$$\sum_{i=1}^3 \eta_{i} \wedge \eta_{4-i} \in \Im \bnabla' \cap \ker \bnabla \cap A^2 (X, D, E) .$$
		Let us define
		$$
		\eta_4:=-\bnabla''^*G(\sum_{i=1}^3 \eta_{i} \wedge \eta_{4-i}).
		$$
		By   \Cref{prop:regularity}, $\eta_4\in A^1(X,D,E) \cap \Im \bnabla'$ and  
		\begin{align}\label{eq:second3} 
			\bnabla\eta_4=- (\sum_{i=1}^3 \eta_{i} \wedge \eta_{4-i}).
		\end{align} 
		In particular, $\Res_D\eta_4=0$.

		\medspace
		
		\noindent {\bf Step 4.} \emph{Constructing general $\eta_k$ for $k\geq 5$}. 
		The construction of higher order $\eta_k$ is similar and we continue this process by induction. Assume that for some $k\geq 5$, we obtain a sequence of $\eta_3,\ldots,\eta_{k-1}\in A^1(X,D,E)$ such that  for any $i\in \{3,\ldots,k-1\}$, we have 
		\begin{enumerate}[label=(\alph*)]
			\item\label{item:exact}  
			$\eta_i\in \Im \bnabla'$ 
			
			\item\label{item:inductive} 
			$\bnabla \eta_i =- \sum_{j=1}^{i-1}\eta_j\wedge \eta_{i-j}$
			
			\item\label{item:zeroresidue} $\Res_D\eta_i=0$.
		\end{enumerate} 
		Let us define 
		$$
		\gamma_k:=-\sum_{j=1}^{k-1}\eta_j\wedge \eta_{k-j}.
		$$
		It follows from \Cref{item:exact,item:inductive} together with \cref{lem:exact0} that   $\gamma_k$ is $\bnabla'$-exact. 
		Now we prove that $\gamma_k$ is $\bnabla$-closed. First, by indcution \eqref{item:inductive}, we have
		\begin{align*}
			-\nabla\gamma_k&= \sum_{j=1}^{k-1} \nabla \eta_j\wedge \eta_{k-j} - \sum_{j=1}^{k-1}  \eta_j\wedge \nabla \eta_{k-j}\\
			&=-\sum_{j=1}^{k-1} \sum_{i=1}^{j-1} \eta_i \wedge \eta_{j-i-1}\wedge \eta_{k-j} + \sum_{j=1}^{k-1}\sum_{i=1}^{k-j-1}  \eta_j\wedge \eta_i \wedge \eta_{k-j-i-1} \\
			&=- \sum_{ i+j+s =k-1} \eta_i \wedge \eta_j \wedge \eta_s + \sum_{ i+j+s =k-1} \eta_i \wedge \eta_j \wedge \eta_s =0,
		\end{align*}
		where the index $i, j, s \geqq 1$. 

		By \eqref{eq:resiwedge},  for each $D_i$, we have 
		\begin{align*}
			\Res_{D_i}\gamma_k &= -\Res_{D_i} \eta_2 \wedge \eta_{k-2} + \eta_{k-2} \wedge \Res_{D_i} \eta_2\\
			& =  \beta_i \wedge (\eta_{k-2} |_{D_i}) - (\eta_{k-2} |_{D_i}) \wedge  \beta_i.    \end{align*}
		Here  $(\eta_{k-2} |_{D_i})$ is defined in \eqref{eq:restrictionabu}.   
		By \Cref{item:zeroresidue} together with \eqref{eq:2residue}, we have 
		$$
		\Res_{D_i\cap D_j} (\eta_{k-2} |_{D_i})=0 
		$$
		for each $j\neq i$.
		This implies that
		\begin{align}\label{eq:ttresi}
			\Res_{D_j\cap D_i}\Res_{D_i}\gamma_k=0.
		\end{align} 
		for each $j\neq i$. 
		
		As $(\eta_{k-2} |_{D_i})$ is $\bnabla'$-exact by \eqref{eq:restrictionabu}, $\beta_i \wedge (\eta_{k-2}|_{D_i}) $ is $\bnabla'$-exact.  
		It follows that
		$
		\Res_{D_i}\gamma_k
		$ 
		is also $\bnabla'$-exact.  
		Note that
		$$
		\nabla\Res_{D_i}\gamma_k=-\Res_{D_i}(\nabla\gamma_k)=0. 
		$$
		By \cref{lem:equiv} together with \eqref{eq:ttresi},
		$$
		\bnabla \Res_{D_i}\gamma_k=0
		$$
		for each $i$.  
		Then 
		$$\Res_{D_i}\gamma_k \in \Im \bnabla' \cap \ker \bnabla.$$
		Since $\Res_{D_i}\gamma_k\in A^1(D_i,\sum_{j\neq i}D_j\cap D_i,E)$, by \cref{lem:ddbar}, $\Res_{D_i}\gamma_k =0$. It follows from \cref{lem:equiv} that $\gamma_k$ is   $\bnabla$-closed.  Hence we have  $$\gamma_k \in \Im \bnabla' \cap \ker \bnabla.$$   
		Let us define
		$$
		\eta_k:= \bnabla''^*G\gamma_k. 
		$$
		Then by \Cref{prop:regularity},  $\eta_k\in A^1(X,D,E) \cap \Im \bnabla'$ and we have 
		\begin{align*} 
			\bnabla\eta_k=- \sum_{j=1}^{k-1}\eta_j\wedge \eta_{k-j}.
		\end{align*} 
		In particular, $\Res_D\eta_k=0$. 
		Therefore, \Cref{item:exact,item:inductive,item:zeroresidue} all hold for $\eta_k$. So the induction is proved.  We conclude the existence of each $\eta_k$ that satisfies the above conditions.  This proves Item (i). 
		
		\medskip
		
		\noindent {\bf Step 5.} We would like to show the convergence of $\sum_i t^i \eta_i$ for $|t| \ll 1$. 
		
		By our construction, for $k\geq 3$, $\eta_k$ is the unique solution such that $\eta_k$ is $\bnabla'$-exact and satisfying
		$$\bnabla \eta_k = -\sum_{i=1}^{k-1} \eta_{i} \wedge \eta_{k-i} .$$
		Fix some $m>\dim X$ such that the  Sobolev multiplication theorem holds. 
		By using Corollary \ref{effectiveest}, there exists some \emph{uniform} constant $C>1$ such that
		$$\|\eta_k\|_{W^m} \leq C (\sum_{i=1}^{k-1} \|\eta_i\|_{W^m} \cdot \|\eta_{k-i}\|_{W^m} ) ,$$
		where the $W^m$-norm is defined in \eqref{normsob}.
		Therefore by induction and the recurrence formula for Catalan numbers, we have
		\begin{equation*}
			\|\eta_k\|_{W^m} \leq \sum_{n=\lceil k/2 \rceil}^{k} \frac{1}{n} \binom{2n-2}{n-1} C^{n-1} \binom{n}{k-n} \|\eta_1\|_{W^m}^{2n-k} \|\eta_2\|_{W^m}^{k-n}.
		\end{equation*}
		This yields the following estimate valid for all $k \geq 3$:
		\begin{equation}\label{eq:cata}
			\|\eta_k\|_{W^m} \leq \frac{1}{2C} \Lambda^k.
		\end{equation}
		where $$\Lambda = 2C \|\eta_1\|_{W^m} + \sqrt{4C^2 \|\eta_1\|_{W^m}^2 + 4C \|\eta_2\|_{W^m}}.$$ 
		
		Then $\sum_i t^i \eta_i$  converges in the $W^m$-normed space for $|t| \leq  \frac{1}{2\Lambda}$. 
		\medskip
		
		Set $\eta (t) := \sum_{i=1}^{\infty} t^i \eta_i$.  It remains to show that $\sum_{i=1}^{\infty}  t^i \eta_i \in A^1 (X, D,E)$ for $|t| \ll 1$. 
		By the above paragraph, we know that $\|\eta (t)\|_{W^m} < +\infty$ for $|t| \ll 1$.
		Moreover, by construction, we have
		$$\nabla  \eta (t)  +\eta (t) \wedge \eta (t)=0.$$
		Moreover, as $\bnabla' \eta_i =0$ for every $i$, we know that
		and $\nabla' \eta (t)=0$. 
		
		Then by the standard elliptic bootstrapping argument, we know that $\eta (t) \in A^1 (X, D,E)$. To see this, we suppose by induction that $\eta (t)$ is in $W^k$ for some $k\in\mathbb N$ in the sense of \eqref{normsob}. 
		We suppose that $D$ is locally defined by $\prod_{i\in I} z_i =0$. 
		Then $\eta (t)$ can be locally written as 
		$$\eta (t) = \sum_{i\in I} f_i \frac{dz_i}{z_i} +g$$ 
		where $f_i$ is $E$-valued function in $W^k$ and $g \in W^k$ is of $(0,1)$-type or $(1,0)$ not containing $dz_i$ for $i\in I$.  
		As $\nabla (\eta (t))$ and $\nabla' \eta (t)$ are in $W^k$, we know that $\nabla'' \eta (t)$ is also in $W^k$. This implies that $\db f_i \in W^k$.
		Then $f_i \in W^{k+1}$ by elliptic regularity. The term $g$ is also in $W^{k+1}$ by the classical elliptic estimates. Therefore $\eta (t)$ is in $W^{k+1}$. It follows that $\eta (t) \in W^k$ for every $k\in\mathbb N$. Then $\eta (t)\in A^1 (X, D,E)$. Furthermore, since  $\eta_i\in A^1(X,D,E)\cap \Im\bnabla'$ for $i\geq 2$, we have 
		$$\nabla_t''=\nabla''+t\eta_1^{0,1}.$$   This proves Item (ii).

		\noindent \textbf{Step 6.} We now prove Item (iii). It suffices to show that $f^*\eta_k = 0$ for all $k \ge 2$. We proceed by induction on $k$. The base cases $k=1, 2$ hold by assumption. Assume that $f^*\eta_i = 0$ for all $1 \le i \le k-1$. Then by \eqref{eq:maurercartan}, we have
		\begin{equation}
			\bnabla f^*\eta_k = -\sum_{j=1}^{k-1} f^*\eta_j \wedge f^*\eta_{k-j} = 0.
		\end{equation}
		Since $\eta_k$ is $\bnabla'$-exact, its pullback $f^*\eta_k$ lies in $\operatorname{Im}\bnabla' \cap A^1(Y,   D_Y, f^*E)$. By \cref{lem:ddbar}, we conclude that $f^*\eta_k = 0$. This completes the inductive step and the proof of Item (iii). 
		The theorem is proved.   
	\end{proof}

	\subsection{Precise Zigzag between dglas}\label{sec:qis}
	Consider the global dgla $(T^\bullet, d)$ defined by the logarithmic Dolbeault complex:
	$$
	(T^\bullet, d):=	A^0(X,D,E)\stackrel{\nabla}{\to}\cdots\stackrel{\nabla}{\to} A^{2\dim X}(X,D,E).
	$$
	Although this dgla is not formal in general, we can nevertheless construct a   dgla $(C^\bullet, d)$ where each $C^i$ is a finite-dimensional vector space, that connects to it by a zigzag of 1-equivalences. This construction enables us to compute the analytic germ that pro-represents the deformation functor associated with $(T^\bullet, d)$. Roughly speaking, the terms of $(C^\bullet, d)$ are formed by the direct sums of the cohomology groups of the divisor strata:$$C^k := \bigoplus_{p+q=k} H^q(D^{(p)}, E).$$The differential $d: C^k \to C^{k+1}$  is defined using Gysin morphisms  with suitable signs, and the graded Lie bracket is naturally induced by the intersection structure. Precise definitions will be given later in this subsection.

	\begin{lem}\label{lem:1-equi}
		Let $L^\bullet:=(\ker \nabla',\nabla'')$.  Then $\phi:L^\bullet\to (A^\bullet(X,D,E),\nabla)$ is  1-equivalent. 
	\end{lem}
	\begin{proof}
		It is obvious that $H^0(\phi)$ is an isomorphism. Note that by the degeneration of $E_1$-level of Deligne (see also  \cite{LRW}), we have  $$H^1(A^\bullet(X,D,E),\nabla)=H^0(X,\Omega_X(\log D)\otimes E)\oplus H^{0,1}(X, E). $$ 
		Here we endow $E$ with the holomorphic vector bundle structure induced by $\nabla''$. Assume that  $\alpha\in A^1(X,D,E)$ such that $\nabla'\alpha=0$ and $\nabla''\alpha=0$. We decompose $\alpha=\alpha'+\alpha''$, where $\alpha'\in A^{1,0}(X,D,E)$ and $\alpha''\in A^{0,1}(X,E)$. Then we have 
		\begin{align*}
			\nabla''\alpha'=0\quad 		\nabla''\alpha''=0\\
			\nabla'\alpha'=0\quad 		\nabla'\alpha''=0.
		\end{align*}
		This implies that $\alpha'\in H^0(X,\Omega_X(\log D)\otimes E)$ and $\alpha''\in \cH^{0,1}(X,E)$. Hence,    $\phi$ induces an isomorphism for  $H^1$.  
		
		Let us prove that   $H^2(\phi)$ is injective. Again by the degeneration at the $E_1$-level, we have  $$H^2(A^\bullet(X,D,E),\nabla)=H^0(X,\Omega^2_X(\log D)\otimes E)\oplus H^{0,2}(X, E)\oplus H^{1}(X,\Omega_X^1(\log D)\otimes E). $$
		Assume that  $\alpha\in A^2(X,D,E)$ such that $\nabla'\alpha=0$ and $\nabla''\alpha=0$. We decompose $\alpha=\alpha^{2,0}+\alpha^{1,1}+\alpha^{0,2}$, where $\alpha^{i,2-i}\in A^{i,2-i}(X,D,E)$. Then we have 
		\begin{align*}
			\nabla''\alpha^{i,2-i}=0,\quad 			\nabla'\alpha^{i,2-i}=0.
		\end{align*} 
		This implies that $\alpha^{2,0}\in H^0(X,\Omega^2_X(\log D)\otimes E)$ and $\alpha^{0,2}\in \cH^{0,2}(X,E)$.  
		
		Assume that $\phi(\alpha)=0$. Then  there exists  some $\beta\in A^1(X,D,E)$ such that $\alpha=\nabla \beta$.   It follows that
		$$
		\alpha^{2,0}=\nabla'\beta^{1,0},\quad  \alpha^{0,2}=\nabla''\beta^{0,1},\quad   \alpha^{1,1}=\nabla''\beta^{1,0}+\nabla'\beta^{0,1}, 
		$$
		where $\beta^{1,0}\in A^{1,0}(X,D,E)$ and $\beta^{0,1}\in A^{0,1}(X,E)$. 
		Since 
		$$\nabla'\nabla''\beta^{0,1}=\nabla'\alpha^{0,2}=0, $$ by \cref{lem:ddbar}, we have 
		$$
		\nabla''\beta^{0,1}=0,\quad 	\nabla'\beta^{0,1}=\nabla'\nabla'' h
		$$
		for some $h\in A^0(X,E)$. Therefore, if we replace $\beta^{0,1}$ by $\beta^{0,1}-\nabla'' h$, and $\beta^{1,0}=\beta^{1,0}+\nabla' h$, then we have
		$$
		\alpha^{0,2}=\nabla''\beta^{0,1},\quad \nabla'\beta^{0,1}=0. 
		$$
		Thus, by replacing $\alpha$ with $\alpha-\nabla''\beta^{0,1}$, we may assume without loss of generality that $\alpha^{0,2}=0$ and that $\beta$ is of type $(1,0)$. It follows that 
		$$
		\alpha^{2,0}+\alpha^{1,1}=\nabla\beta^{1,0}. 
		$$
		Note that we have
		$$
		\nabla' \nabla''\beta^{1,0}=\nabla'\alpha^{1,1}=0. 
		$$
		This implies that
		$$
		\nabla' \nabla''\Res_D \beta^{1,0}=		\Res_D(\nabla' \nabla'' \beta^{1,0})=0.
		$$
		Hence $\Res_D \beta^{1,0}\in H^0(D^{(1)},E)$. 
		By \cref{lem:equiv}, we have 
		$$
		\bnabla'' \beta^{1,0}= \nabla''\beta^{1,0}+\iota_*\Res_D \beta^{1,0}. 
		$$
		This implies that 
		$$
		\bnabla'	\bnabla'' \beta^{1,0}=\bnabla' \nabla''\beta^{1,0}+\iota_*\bnabla'\Res_D \beta^{1,0}=\bnabla' \nabla''\beta^{1,0} =\nabla' \nabla''\beta^{1,0} =0. 
		$$
		By   \cref{lem:ddbar}, we have 
		$$
		\bnabla' \beta^{1,0}=	 \bnabla'  \bnabla'' h
		$$
		for some $h\in \sD^0(X,E)$. Since the right side has type $(1,1)$, whereas the left side has type  $(2,0)$, it implies that $
		\bnabla' \beta^{1,0}=	0. 
		$
		Hence $\alpha^{2,0}=0,$ and we have
		$$
		\alpha^{1,1}=\nabla''\beta^{1,0}
		$$  
		for some $\beta^{1,0}\in \ker\nabla'$.  This proves that the class $\{\alpha\}$ vanishes in $H^2(L^\bullet,d)$, thereby establishing the injectivity of $H^2(\phi)$. This completes the proof of the lemma. 
	\end{proof}

	\medskip

	Let us now introduce  a new dgla.  Let $D=\sum_{i=1}^{m}D_i$. Dentoe by $D^{(k)}:=\sqcup_{|I|=k}D_I$ and let 
	$$D_{(k)}:=\sqcup_{|I|=k}\sum_{j\notin I}D_I\cap D_j$$
	that is a simple normal crossing divisor on $D^{(k)}$.

	We first introduce some terminology.  For an arbitrary multi-index $I$, let $\sigma \in S_{|I|}$ be the unique permutation such that $(i_{\sigma(1)}, \ldots, i_{\sigma(j)})$ is ordered. We denote by $\sigma(I):=(i_{\sigma(1)}, \ldots, i_{\sigma(j)})$ this ordered multi-index. Namely, we have $(-1)^{\ep(I)} := \mathrm{sgn}(\sigma)$.  
	We set $C^{p,q}:=H^q(D^{(p)},E)$ and the \emph{differential} $d:C^{p,q}\to C^{p-1,q+2}$  is defined to be the \emph{signed} Gysin morphism as follows:
	Let $I$ be an ordered multi-index with $|I|=p$. For any $i \in I$, let $I_i$ denote the ordered multi-index corresponding to the set $I \setminus \{i\}$. Then for  any element $\alpha\in H^q(D_I,E)$, let 
	$$
	g_{I,i}:H^q(D_I,E)\to H^{q+2}(D_{I_i},E)
	$$
	be the   Gysin morphism defined in \eqref{def:Gysin}.  Then
	we define a differential
	\begin{align}\label{eq:defdiff}
		d\alpha:= \sum_{i\in I}(-1)^ {\ep(I,i)}g_{I,i} (\alpha),
	\end{align}
	where $\ep(I,i)$ denotes the number of elements in $I$ that are strictly smaller than $i$. 
	
	For any element $\alpha\in C^k$, we denote by $|\alpha|=k$ its degree. 
	
	\begin{lem}
		$(C^\bullet,d)$ is a  chain complex. 
	\end{lem}
	\begin{proof}
		It suffices to verify that  $d^2=0$. For any ordered multi-index $I$ and any distinct $i$ and $j$ in $I$, we denote by $K$ the ordered multi-index $I\backslash \{i,j\}$. Then     the summmand  of 
		$
		d^2(\alpha)
		$ in $H^{q+4}(D_{K},E)$ is given by
		$$
		(-1)^{\ep(I_i,j)} g_{I_i,j}\left( (-1)^{\ep(I,i)}g_{I,i}(\alpha)
		\right)+  (-1)^{\ep(I_j,i)} g_{I_j,i}\left( (-1)^{\ep(I,j)}g_{I,j}(\alpha)
		\right)$$
		By \cref{lem:Gysin}, we have
		$$
		g_{I_i,j} \circ g_{I,i}(\alpha)=  g_{I_j,i} \circ g_{I,j}(\alpha)=\{(\iota_{K,I})_*\eta\}
		$$
		where $\eta\in A^q(D_I,E)$ such that $\nabla\eta=0$, and  is a representative of $\alpha$, and $\iota_{I,K}:D_I\to D_K$ is the inclusion map. 
		We also note that
		$$
		(-1)^{\ep(I_i,j)}  (-1)^{\ep(I,i)} +  (-1)^{\ep(I_j,i)}  (-1)^{\ep(I,j)} =0.
		$$
		This implies that $d^2(\alpha)=0$.  Hence $(C^\bullet,d)$ is a chain complex. 
	\end{proof}
	
	\begin{deflem}\label{dfn:dgla} 
		We define  a bracket 
		$$
		C^{p,q}\times C^{k,\ell}\to C^{p+k,q+\ell}
		$$ 
		as follows.   
		Pick any $ \alpha \in A^q(D_I,E)$ and $ \beta \in A^\ell(D_J,E)$  with $|I|=p$ and $|J|=k$, that are both $\nabla$-closed.  We abusively denote by $\alpha$ and $\beta$ their cohomolgy classes.   We define  \begin{equation}\label{eq:Lie}
			[\alpha,\beta]:=\begin{cases}
				0 \quad \mbox{ if }  I\cap J\neq \varnothing\\
				(-1)^{\ep(IJ)+q|J|}\{[\iota_1^*\alpha,\iota_2^*\beta]\}   \quad \mbox{ if }  I\cap J= \varnothing 
			\end{cases}
		\end{equation} 
		Here, 
		\begin{itemize}
			\item $IJ$ denotes the multi-index obtained by appending $J$ to $I$
			\item    $K$ denotes the ordered multi-index corresponding to the set $I\cup J$. 
			\item  $\iota_1:D_K\to D_I$ and $\iota_2:D_K\to D_J$ denote the inclusion maps. 
		\end{itemize}  Then  the chain complex $(C^\bullet,d)$ endowed with the above  bracket is a  dgla.   
	\end{deflem} 
	\begin{proof}
		It is straightforward to see that the above bracket is a (graded) Lie algebra. Let us prove the graded Leibniz rule. 
		
		We first assume that $I\cap J=\varnothing$.	Pick any $i\in I$. Let $L$ be the ordered multi-index corresponding to the set $K\backslash \{i\}$. 
		Let $\eta\in A^{q+1}(D_{I_i},D_{I},E)$ be such that $\Res_{D_I}\eta=\alpha$.  Let $\iota_3:D_L\to D_{I_i}$ and $\iota_4:D_L\to D_J$.
		Then 
		$$\iota_3^*\eta\in A^{q+1}(D_L,D_K,E).$$ Hence
		$$
		[\iota_3^*\eta,\iota_4^*\beta]\in A^{q+\ell+1}(D_L,D_K,E)
		$$
		such that 
		$$
		\Res_{D_K}[\iota_3^*\eta,\iota_4^*\beta]=[\iota_1^*\alpha,\iota_2^*\beta]. 
		$$
		It follows that $\nabla[\iota_3^*\eta,\iota_4^*\beta]$ represents $$g_{K,i}(\{[\iota_1^*\alpha,\iota_2^*\beta]\})\in H^{q+\ell+2}(D_L,E),$$
		and thus
		$$
		g_{K,i}([\alpha,\beta])=-(-1)^{\ep(IJ)+qk}\{\nabla[\iota_3^*\eta,\iota_4^*\beta]\}.
		$$
		On the other hand, note that 
		$$
		\nabla[\iota_3^*\eta,\iota_4^*\beta]=[\iota_3^*\nabla\eta,\iota_4^*\beta].
		$$ 
		Since we have 
		$$
		g_{I,i}(\alpha)=-\{\nabla\eta\}\in H^{q+2}(D_{I_i},E),
		$$
		it follows that
		$$
		\{\nabla[\iota_3^*\eta,\iota_4^*\beta]\}=-(-1)^{\ep(I_iJ)+(q+2)k}[g_{I,i}(\alpha),\beta]. 
		$$
		This implies that 
		\begin{align}\nonumber
			(d[\alpha,\beta])_L&=(-1)^{\ep(K,i)} g_{K,i}([\alpha,\beta])=-(-1)^{\ep(K,i)+\ep(IJ)+qk}\{\nabla[\iota_3^*\eta,\iota_4^*\beta]\} \\\nonumber
			& = (-1)^{\ep(K,i)+\ep(IJ)+qk+\ep(I_iJ)+(q+2)k}[g_{I,i}(\alpha),\{\beta\}]\\\nonumber
			&= (-1)^{\ep(K,i)+\ep(IJ)+qk+\ep(I_iJ)+(q+2)k+\ep(I,i)}[ (d\alpha)_{I_i},\{\beta\}]\\\label{eq:dg}
			&=[ (d\alpha)_{I_i},\{\beta\}].
		\end{align} 
		
		On the other hand,  pick any $i\in J$. Let $L$ be the ordered multi-index corresponding to the set $K\backslash \{i\}$.  Then
		\begin{align*}
			(d[\alpha,\beta])_L&= (-1)^{|\alpha||\beta|+1}(d[\beta,\alpha])_L\\
			& \stackrel{\eqref{eq:dg}}{=}(-1)^{|\alpha||\beta|+1}([(d\beta)_{J_i},\alpha])\\
			&=(-1)^{|\alpha||\beta|+1}(-1)^{|\alpha|(|\beta|+1)+1}([\alpha,(d\beta)_{J_i}])\\
			&=(-1)^{|\alpha|}([\alpha,(d\beta)_{J_i}]).
		\end{align*} 
		This proves the graded Leibniz rules when $I\cap J=\varnothing$.

		Now assume that $I\cap J\ne\varnothing$. If $|I\cap J|\geq 2$,
		both sides of the graded Leibniz rule vanish by definition.
		It remains to consider $I\cap J=\{a\}$. Set
		\[
		\xi:=c_1\bigl(\mathcal O_X(D_a)|_{D_K}\bigr)
		\wedge  \bigl\{[\iota_1^*\alpha,\iota_2^*\beta]\bigr\}.
		\]
		Note that only the terms obtained by removing $a$ can contribute.
		Then  we have
		\[
		\begin{aligned}
			[d\alpha,\beta]
			&=(-1)^{\ep(I,a)+\ep(I_aJ)+(q+2)k}\xi,\\
			(-1)^{|\alpha|}[\alpha,d\beta]
			&=(-1)^{p+q+\ep(J,a)+\ep(IJ_a)+q(k-1)}\xi.
		\end{aligned}
		\]
		Since
		\[
		\ep(I_aJ)+\ep(IJ_a)
		\equiv p-1-\ep(I,a)+\ep(J,a)\pmod 2,
		\]
		the two signs are opposite. Hence
		\[
		[d\alpha,\beta]+(-1)^{|\alpha|}[\alpha,d\beta]
		=0=d[\alpha,\beta].
		\]  This shows  the graded Leibniz rule.  The lemma is proved. 
	\end{proof}

	For any \emph{ordered} multi-index $I$, and any $\eta\in A^k(X,D,E)$, 
	we set  $$\eta_{I}:=\Res_{D_I}\eta\in A^{k-j}(D_I,\sum_{p\notin I}D_I\cap D_p, E).$$ 
	If $I$ not not ordered, then we set 
	\begin{align}\label{eq:permu}
		\eta_I:=(-1)^{\ep(I)}\eta_{\sigma(I)}. 
	\end{align}  
	By \cref{lem:equiv} and \eqref{eq:permu}, we obtain the identity
	\begin{align}\label{eq:resicur3}
		\bnabla\eta_{I}=\nabla\eta_{I}+ \sum_{i\notin I} (\iota_{I,i})_*\eta_{Ii},
	\end{align}
	where $\iota_{I,i} \colon D_I \cap D_i \hookrightarrow D_I$ is the natural inclusion, and the subscript $Ii$ denotes the multi-index formed by appending $i$ to $I$. 
		
		Note that
		\begin{align}\label{eq:nabla0}
			\bnabla' \eta_I=\nabla'\eta_I= (-1)^{|I|}\Res_{D_I}(\nabla'\eta)=0.
		\end{align}  
		
		We are ready to construct the morphism of $\psi: L^\bullet\to C^\bullet $.     
		Consider any $\eta\in L^k$.  Define
		\begin{align} \label{eq:psi}
			\psi(\eta)=\left( (-1)^{|I|}\{\cH(\eta_I)\}\right)_{I}\in C^k
		\end{align} 
		where $I$ ranges over all ordered multi-index contained in $\{1,\ldots,m\}$, and  $\cH(I)\in A^{k-|I|}(D_I,E)$ is the harmonic projection of $\eta_I$. 
		\begin{proposition}\label{prop:1qis}
			The map $\psi:L^\bullet\to C^\bullet$ is a morphism of dgla, and is moreover a $1$-quasi-isomorphism.
		\end{proposition}
		\begin{proof}
			\noindent {\textbf{Step 1.}}	We first show that $\psi$ is a morphism of chain complexes. 
			
			Pick any $\eta\in L^k$. Note that 
			\begin{align}\label{eq:exchange}
				(\nabla'' \eta)_{I}=	 (-1)^{|I|}(\nabla'' \eta_{I})\stackrel{\eqref{eq:resicur3}}{=}(-1)^{|I|}(\bnabla  \eta_{I}-  \sum_{i\notin I} (\iota_{I,i})_*\eta_{Ii}).
			\end{align}
			It follows that
			$$
			\cH((\nabla'' \eta)_{I})=(-1)^{|I|+1}\cH(\sum_{i\notin I} (\iota_{I,i})_*\eta_{Ii}).$$
			Since $$\bnabla'\eta_{I}=(\bnabla'\eta)_{I}=0,$$ 
			it follows that
			$$
			0= \bnabla' \eta_{I}=\bnabla'\bnabla'^*  \bnabla' G\eta_{I}=\Delta' \bnabla' G\eta_{I}-\bnabla'^* \bnabla' \bnabla' G\eta_{I}=\Delta' \bnabla' G\eta_{I}
			$$
			Therefore, $ \bnabla' G\eta_{I}=0$, and we have  
			\begin{align}\label{eq:exactI}
				\eta_{I}=\cH(\eta_{I})+\bnabla' \bnabla'^* G\eta_{I}.
			\end{align} 
			Hence for any $i\notin I$, we have 
			$$
			(\iota_{I,i})_*\eta_{Ii}=(\iota_{I,i})_*\cH(\eta_{Ii})+\bnabla'(\iota_{I,i})_* \left(\bnabla'^* G\eta_{Ii} \right),
			$$
			which yields
			$$
			\cH\left((\iota_{I,i})_*\eta_{Ii}\right)=	\cH\left((\iota_{I,i})_*\cH(\eta_{Ii}) \right).
			$$
			Therefore, we have
			$$
			\cH((\nabla'' \eta)_{I})=(-1)^{|I|+1}\sum_{i\notin I}\cH( (\iota_{I,i})_*\cH(\eta_{Ii})).$$
			By \cref{lem:Gysin}, we have
			\begin{align}\label{eq:com2}
				\{\cH\left( (\iota_{I,i})_*\cH(\eta_{Ii})\right)\}= g_{K(i)}(\{\cH(\eta_{Ii})\})
			\end{align}
			where we denote by $K(i)$ the ordered multi-index corresponding to the set $I\cup \{i\}$, and 
			$$
			g_{K(i)}:H^q(D_{K(i)},E)\to H^{q+2}(D_{I},E)
			$$
			be the   Gysin morphism defined in \eqref{def:Gysin}.
			This implies that
			$$
			\{ \cH((\nabla'' \eta)_{I})\}=(-1)^{|I|+1}\sum_{i\notin I}g_{K(i)}(\{\cH(\eta_{Ii})\})$$
			Hence 
			$$
			\psi(\nabla''\eta)_I=-\sum_{i\notin I}g_{K(i)}(\{\cH(\eta_{Ii})\})
			$$
			
			On the other hand,  by \eqref{eq:defdiff}, we note that 
			\begin{align*}
				\left(	d\psi(\eta)\right)_I&=     \sum_{i\notin I} (-1)^ {\ep(K(i),i)}g_{K(i)}\left((-1)^{|I|+1}\{\cH( \eta_{K(i)})\}\right)\\
				&	\stackrel{\eqref{eq:permu}}{=}  \sum_{i\notin I} (-1)^ {\ep(K(i),i)+\ep(Ii)+|I|+1}g_{K(i)}\left(\{\cH( \eta_{Ii})\}\right)\\
				&	=-\sum_{i\notin I}g_{K(i)}(\{\cH(\eta_{Ii})\})
			\end{align*}
			This implies that
			$$
			d\psi(\eta)=\psi(\nabla''\eta).
			$$
			Hence $\psi$ is a morphism of complexes.

			\medskip
			
			\noindent {\textbf{Step 2.}} We show  $\psi$ preserves the Lie brackets. 
			
			We endow $\cO_X(D_i)$ with a smooth Hermitian metric $h_i$
			and denote by $s_i$ its canonical section. Set
			$$
			\theta_i:=\frac{1}{2\pi i}\d\log|s_i|_{h_i}^2.
			$$
			For $I=\{i_1<\cdots<i_p\}$, write
			$\theta_I:=\theta_{i_1}\wedge\cdots\wedge\theta_{i_p}$,
			with $\theta_\varnothing:=1$.
			For $I\subset K$, define by linearity
			$$
			r_{K,I}:A^k(D_I,E_I,E)\longrightarrow A^k(D_K,E_K,E)
			$$
			by
			$$
			r_{K,I}(\theta_J|_{D_I}\wedge h)=
			\begin{cases}
				0,&J\cap(K\setminus I)\ne\varnothing,\\
				\theta_J|_{D_K}\wedge h|_{D_K},
				&J\cap(K\setminus I)=\varnothing,
			\end{cases}
			$$
			where $J\cap I=\varnothing$ and
			$h\in A^{k-|J|}(D_I,E)$ is smooth.
			A local computation shows that $r_{K,I}$ is independent of
			the chosen decomposition, preserves brackets, and commutes
			with $\bnabla'$.
			
			Fix any $\eta\in L^k$, $\sigma\in L^\ell$, and a multi-index $K$.
			For $I\sqcup J=K$, denote the inclusions by
			$\iota_1:D_K\to D_I$ and $\iota_2:D_K\to D_J$.
			We have 
			\begin{align}\label{eq:residueDK}
				\Res_{D_K}[\eta,\sigma]
				=
				\sum_{I\sqcup J=K}
				(-1)^{\ep(IJ)+|J|(k-|I|)}
				[r_{K,I}\eta_I,r_{K,J}\sigma_J].
			\end{align}

			Note  that  by \eqref{eq:exactI}, we can write
			$
			\eta_I=\cH(\eta_I)+\bnabla'u_I,
			$ 
			where $u_I$ is a finite sum of forms in
			$A^{k-|I|-1}(D_I,E_I,E)$ multiplied by polynomials in
			$\{\log|s_j|_{h_j}^2|_{D_I}\}_{j\notin I}$. We extend $r_{K,I}$ to these expressions by sending
			$\log|s_j|_{h_j}^2|_{D_I}$ to zero for $j\in K\setminus I$
			and restricting the remaining logarithmic functions to $D_K$.
			This extension is well-defined and commutes with $\bnabla'$.
			Setting $h(\eta,I,K):=r_{K,I}u_I$, we obtain
			$$
			r_{K,I}\eta_I
			=
			\iota_1^*\cH(\eta_I)+\bnabla'h(\eta,I,K),
			$$
			and similarly for $\sigma$.
			Here $h(\eta,I,K)$ a finite sum of forms in
			$A^{k-|I|-1}(D_K,E_K,E)$ multiplied by polynomials in
			$\{\log|s_j|_{h_j}^2|_{D_I}\}_{j\notin K}$.   It implies that 
			$$ [r_{K,I}\eta_I,r_{K,J}\sigma_J]
			-[\iota_1^*\cH(\eta_I),\iota_2^*\cH(\sigma_J)] =
			\bnabla'\Big(
			[h(\eta,I,K),r_{K,J}\sigma_J]
			+(-1)^{k-|I|}
			[\iota_1^*\cH(\eta_I),h(\sigma,J,K)]
			\Big). 
			$$Taking harmonic projections and by \eqref{eq:residueDK}, we obtain
			$$
			\cH(\Res_{D_K}[\eta,\sigma])
			=
			\sum_{I\sqcup J=K}
			(-1)^{\ep(IJ)+|J|(k-|I|)}
			\cH\left(
			[\iota_1^*\cH(\eta_I),\iota_2^*\cH(\sigma_J)]
			\right).
			$$
			Together with \eqref{eq:Lie}, this yields
			$$
			[\psi(\eta),\psi(\sigma)]_K
			=
			(-1)^{|K|}\{\cH(\Res_{D_K}[\eta,\sigma])\}
			=
			\psi([\eta,\sigma])_K.
			$$
			Hence $\psi$ preserves the Lie bracket. Together with Step 1,
			this proves that $\psi$ is a morphism of dglas.
			
			\medspace
			
			\noindent {\textbf{Step 3.}} We show that  $H^i(\psi)$ is an isomorphim for $i=0,1$.
			
			Note that $d:C^0\to C^1$ is the zero map. Therefore, $H^0(\psi):H^0(L^\bullet)\to H^0(C^\bullet)$ is an identity map. 
			Consider $H^1(\psi):H^1(L^\bullet)\to H^1(C^\bullet)$. By  the proof of \cref{lem:1-equi},   any element  $\alpha\in H^1(L^\bullet)$ has a unique representative 
			$$
			(\eta^{1,0},\eta^{0,1})\in H^0(X,\Omega_X(\log D)\otimes E)\oplus \cH^{0,1}(X, E). 
			$$
			Write $\eta:=\eta^{1,0}+\eta^{0,1}$, and $\eta_{(1)}:=\Res_D\eta\in H^0(D^{(1)},E)$. Then 
			$$
			\bnabla \eta^{1,0}=\iota_*\eta_{(1)}, 
			$$ 
			and $f(\eta_{(1)})\in A^{1,0}(X,D,E)$ where $f$ defined in \eqref{eq:f} satisfies 
			$$
			\bnabla f(\eta_{(1)})=-\iota_*\eta_{(1)},
			$$
			and $f(\eta_{(1)})$ is $\bnabla'$-exact. It follows that
			$$
			\Delta''	(\eta^{1,0}+f(\eta_{(1)})=0 
			$$
			which yields
			$$
			\eta^{1,0}+ f(\eta_{(1)})\in H^0(X,\Omega_X\otimes E). 
			$$ 
			In particular, 
			$$
			\eta^{1,0}+f(\eta_{(1)})= \cH(\eta^{1,0}+f(\eta_{(1)}))= \cH(\eta^{1,0}). 
			$$ 
			Hence we have
			$$
			\psi(\eta)=(\{\eta^{1,0}+ f(\eta_{(1)})+\eta^{0,1}\},-\{\eta_{(1)}\}). 
			$$
			Thus  if $\{\psi(\eta)\}=0$ in $H^1(C^\bullet)$, then 
			$$\eta_{(1)}=0,\quad  \eta^{1,0}- f(\eta_{(1)})=0, \quad \eta^{0,1}=0.$$  This implies that 
			$
			\eta=0.
			$   We obtain the injectivity of $H^1(\psi)$.  
			
			Let $( a^{0,1},a^{1,0})\in  C^{0,1}\oplus C^{1,0}$ such that $d( a^{0,1},a^{1,0} )=0$. It follows that 
			$$
			g(a^{1,0})=0 
			$$
			where $g:H^0(D^{(1)},E)\to H^2(X,E)$ is the Gysin morphism. 
			By the construction of $f$ in \eqref{eq:f}, this implies that  
			$$
			\bnabla f(a^{1,0})=-\iota_*a^{1,0}. 
			$$
			In particular,   $$f(a^{1,0})\in H^0(X,\Omega_X(\log D)\otimes E).$$ 
			Let $\eta\in \cH^1(X,E)$ be a representative of $a^{0,1}\in H^1(X,E)$. Then   $\eta+f(a^{1,0})\in L^1$, such that  
			$$
			\psi(\eta+f(a^{1,0}))=(\{\cH(\eta+f(a^{1,0}))\} ,-\{\cH(\Res_D{(\eta+f(a^{1,0}))})\}=(a^{0,1},a^{1,0}). 
			$$
			Here we use the fact that $f(a^{1,0})$ is $\bnabla'$-exact. 
			This shows that $H^1(\psi)$ is surjective. 
			
			\medskip
			
			\noindent {\textbf{Step 4.}}  	We now prove that $H^2(\psi)$ is injective.  Pick any $\eta\in L^2$ such that $\nabla''\eta=0$.  We decompose $\eta=\eta^{2,0}+\eta^{1,1}+\eta^{0,2}$, where $\eta^{i,2-i}\in A^{i,2-i}(X,D,E)$. Then we have 
			\begin{align*}
				\nabla''\eta^{i,2-i}=0,\quad 			\nabla'\eta^{i,2-i}=0.
			\end{align*} 
			This implies that $\eta^{2,0}\in H^0(X,\Omega^2_X(\log D)\otimes E)$ and $\eta^{0,2}\in \cH^{0,2}(X,E)$.

			If $\psi(\eta)=d(a^{0,1},a^{1,0})$  for some $(a^{0,1},a^{0,1})\in  C^{0,1}\oplus C^{1,0}$, then we have
			$
			\psi(\eta)\in C^{0,2},
			$ 
			which yields that 
			\begin{align}\label{eq:three}
				\{ \cH(\eta)\}=g(a^{1,0}), \quad \{ \cH(\eta_{(1)})\}=0,  \quad \{ \cH(\eta_{I})\}=0   
			\end{align} 
			for each multi-index $I$ with $|I|=2$     by \eqref{eq:defdiff}.
			By the definition of $f$ in \eqref{eq:f}, we have 
			$$
			\bnabla f(a^{1,0})=\cH(g(a^{1,0}))-\iota_*(a^{1,0})\stackrel{\eqref{eq:three}}{=} \cH(\eta)  -\iota_*(a^{1,0}).
			$$
			Note that $\nabla''f(a^{1,0})=\cH(\eta)\in A^{1,1}(X,E)$, that is $\nabla$-closed.  It follows \eqref{eq:three} that 
			$$
			\psi(\eta-\nabla''f(a^{1,0}))=0. 
			$$
			Therefore, we can replace $\eta$ by $\eta-\nabla''f(a^{1,0})$ and   assume that 
			\begin{align}\label{eq:three2}
				\{ \cH(\eta)\}=0, \quad \{ \cH(\eta_{(1)})\}=0,  \quad \{ \cH(\eta_{I})\}=0.  
			\end{align} 
			for each multi-index $I$ with $|I|=2$. Fix any such $I$.   
			Since  $\nabla'\eta=0$, it follows that
			$$
			\nabla'\eta_I=0. 
			$$
			\eqref{eq:three2} then  implies that $\eta_I=\eta^{2,0}_{I}=0$.  Note that 
			$$
			\eta_{(1)}=\eta^{2,0}_{(1)}+\eta^{1,1}_{(1)} 
			$$
			with 
			$$
			\eta^{2,0}_{(1)}\in  H^0(D^{(1)},\Omega_{D^{(1)}}(\log D_{(1)})\otimes E)
			$$
			and 
			$$
			\eta^{1,1}_{(1)}\in  A^{0,1}(D^{(1)},E). 
			$$ 
			By \eqref{eq:exchange}, we have 
			$$
			\nabla'' \eta_{(1)}= -(\nabla'' \eta)_{(1)}=0, \quad  \nabla' \eta_{(1)}= -(\nabla' \eta)_{(1)}=0. 
			$$
			This implies that  $$
			\eta^{1,1}_{(1)}\in  \cH^{0,1}(D^{(1)},E). 
			$$
			Since
			$$ \{ \cH(\eta_{(1)})\}=0,  \quad \{ \cH(\eta_{ij})\}=\{ \cH(\Res_{D_{i\cap D_j}}\eta_{i})\}=0\quad \mbox{ for any }i<j,$$
			the  arguments  in Step 3 in proving the injectivity of $H^1(\psi)$ show that $\eta_{(1)}=0$.  Hence, we have
			$$
			\eta^{2,0}\in H^0(X,\Omega_X^2\otimes E),\quad {\rm Res}_D\eta^{1,1}=0.   
			$$
			Hence
			$$
			\bnabla \eta^{1,1}=\bnabla'' \eta^{1,1}=\nabla''\eta^{1,1}=0.
			$$
			Note that $$\cH(\eta^{1,1})\in \cH^{1,1}(X,E).$$ 
			By \eqref{eq:three2} together with the Hodge composition of harmonic forms   $$\cH^2(X,E)=\cH^{2,0}(X,E)\oplus \cH^{1,1}(X,E)\oplus \cH^{0,2}(X,E),$$ we deduce that 
			$$
			\eta^{2,0}=0,\quad \eta^{0,2}=0,\quad \{\eta^{1,1}\}=0\in H^{2}(X,E). 
			$$
			By \cref{lem:ddbar}, we have
			$$
			\eta^{1,1}=\bnabla''\bnabla'\sigma
			$$
			for some $\sigma\in \sD^0(X,E)$.  By \cref{prop:regularity}, if we define
			$$
			\omega:=\bnabla''^*G\eta^{1,1}, 
			$$
			then $\omega\in A^1(X,D,E)\cap{\rm Im}\bnabla'$,  and we have 
			$$ \bnabla''\omega=\eta^{1,1}. 
			$$ by \cref{lem:equiv}.  This implies that 
			$\{\eta\}=0\in H^2(L^\bullet)$.  The injectivity of $H^2(\psi)$ is proved. 
			
			The proof of the proposition is accomplished. 
		\end{proof}

		In summary, we obtain the following result:
		\begin{thm}\label{thm:zigzag}
			Let $T^\bullet := (A^\bullet(X,D,E), \nabla)$ be the dgla of logarithmic forms. Let $L^\bullet := (\ker \nabla', \nabla'')$ be the dgla defined previously, and let $(C^\bullet, d)$ be the dgla introduced in \Cref{dfn:dgla}. Let $\varepsilon : T^0 \to E_x$, $\varepsilon : L^0 \to E_x$, and $\varepsilon : C^0 \to E_x$ denote the evaluation maps at $x$, which define the corresponding augmentations.
			
			Then, the inclusion map $\phi: L^\bullet \to T^\bullet$ and the map $\psi: (L^\bullet, \nabla'') \to (C^\bullet, d)$ defined in \eqref{eq:psi} are morphisms of dglas, and both are $1$-quasi-isomorphisms. In particular, the diagram
			$$
			T^\bullet \xleftarrow{\phi} L^\bullet \xrightarrow{\psi} C^\bullet
			$$
			establishes a zigzag of $1$-quasi-equivalences connecting $(T^\bullet, d)$ and $(C^\bullet, d)$. \qed
		\end{thm}

		\subsection{Constructing  the universal connection (I)}\label{subsec:uni1}
		In this subsection, we construct the \emph{universal} connection on the representation variety, in the sense specified in \cref{thm:universal2}. This subsection is dedicated to the analytic construction; we defer the proof of its ``universal property'' to \cref{thm:universal2}, using the Goldman–Millson theory  in \cref{sec:GM} and the zigzags of $1$-quasi-equivalences established in \cref{sec:qis}. 
		
		Fix a basis $\alpha_1,\ldots,\alpha_p\in \cH^1(X,E)$, and a basis $\beta_1,\ldots,\beta_q\in H^0(D^{(1)},E)$. 
		Consider the polynomial ring $\bC[z_1,\ldots,z_p,w_1,\ldots,w_q]$. 
		Let $\mathscr{Q}'\subset \bC^{p+q}$ be the quasi-homogeneous cone defined by
		\begin{align}  \label{eq:quasihomo}
			d(\sum_{\gamma=1}^{q}w_\gamma \beta_\gamma+\sum_{i=1}^{p}z_i\alpha_i)+ \frac{1}{2}[\sum_{\gamma=1}^{q}w_\gamma \beta_\gamma+\sum_{i=1}^{p}z_i\alpha_i,\sum_{\gamma=1}^{q}w_\gamma \beta_\gamma+\sum_{i=1}^{p}z_i\alpha_i]  =0\in C^2 
		\end{align} where   $d \colon C^1 \to C^2$ is the differential of the dgla $(C^\bullet,d)$ defined in \eqref{eq:defdiff}, and $[-,-]$ is the Lie bracket defined in \eqref{eq:Lie}.
		Let $f(\beta_\gamma)\in A^1(X,D,E)$ be defined in \eqref{eq:f}.   Then $\Res_Df(\beta_\gamma)=-\beta_\gamma$ and we have $\nabla f(\beta_\gamma)\in \cH^2(X,E)$   by our construction.  Let us denote by $\eta_{i,0}:=\alpha_i$ and $\eta_{0,\gamma}:= f(\beta_\gamma)$. 
		
		For any $(z,w):=(z_1,\ldots,z_p;w_1,\ldots,w_q)\in \sQ'$, we define
		\begin{align}\label{eq:eta12}
			\eta_1(z,w)&:=\sum_{i=1}^{p}z_i\eta_{i,0}\\
			\eta_2(z,w)&:=\sum_{\gamma=1}^{q}w_\gamma\eta_{0,\gamma}-\frac{1}{2}\nabla''^*G [\eta_1(z,w),\eta_1(z,w)]\\ \nonumber 
			&=\sum_{\gamma=1}^{q}w_\gamma\eta_{0,\gamma}-\frac{1}{2}\sum_{1\leq i,j\leq p}z_iz_j\nabla''^*G [\eta_{i,0},\eta_{j,0}]
		\end{align}
		We proceed by induction on $k$. Let $k \geq 3$ and assume that $\eta_j(z,w) \in A^1(X,D,E)\cap \mathrm{Im}\bnabla'$ has been defined for all $1 \leq j \leq k-1$. Then the $2$-form $\gamma_k(z,w)$ given by
		\begin{align*} 
			\gamma_k(z,w) &:= - \sum_{j=1}^{k-1}\eta_j(z,w)\wedge \eta_{k-j} (z,w),
		\end{align*} 
		lies in $A^2(X,D,E)$. Next, we define $\eta_k(z,w)$ using the Green operator:
		\begin{align} \label{eq:etak}
			\eta_k(z,w):=\bnabla''^*G\gamma_k(z,w). 
		\end{align} 
		Note that \eqref{eq:quasihomo} is exactly the three equations in \cref{thm:construction}. 	By \cref{thm:construction}, when $(z,w)\in \sQ'$,  this $\eta_k(z,w)$ satisfies the regularity condition
		\[
		\eta_k(z,w)\in A^1(X,D,E)\cap \mathrm{Im}\bnabla',
		\]
		and solves the equation 
		\[
		\bnabla \eta_k(z,w)=\gamma_k(z,w). 
		\]
		Moreover, there exists some sufficiently small $\ep>0$ such that,  for $(z,w)\in \sQ'_\ep:=\sQ'\cap \bB_\ep^{p+q}$, where $\bB_{\ep}^{p+q}$ is the   ball in $\bC^{p+q}$ of radius $\ep$, 
		the connection
		$$
		\nabla_{z,w}:=\nabla+\sum_{i=1}^{\infty}\eta_i(z,w). 
		$$
		is flat over $X_0$ and has logarithmic poles. 

		We list one fact that will be used later. 
		\begin{lem}\label{lem:nablaexact}
			For any $k \geq 2$ and any $(z,w)\in \sQ'_\ep$,  $\eta_{k}(z,w)$ is $\bnabla'$-exact.   
		\end{lem}
		\begin{proof}
			It follows from \cref{thm:construction}. 
		\end{proof}
		
		We extend \cref{thm:construction} to the universal setting as follows. 
		\begin{thm}\label{main1}
			Let $X$ be a compact K\"ahler manifold and let $D$ be a simple normal crossing divisor on $X$.  Let $(V,\nabla)$ be a unitary bundle on $X$. Denote by $(E,\nabla):=(\End(V),\nabla)$.    
			
			Fix a basis $\eta_{1,0},\ldots,\eta_{p,0}\in \cH^1(X,E)$, and   a basis $\beta_1,\ldots,\beta_q\in H^0(D^{(1)},E)$.  Set $\eta_{0,\gamma}:=f(\beta_\gamma)$ where $f:H^0(D^{(1)},E)\to A^1(X,D,E)$ is defined in \eqref{eq:f}. 
		Let  $\sQ'\subset  \bC^{p+q}$ be the quasi-homogeneous cone defined by the following three equations with values in $H^2(X,E)$, $H^1(D^{(1)},E)$ and $H^0(D^{(2)},E)$ respectively: 
		\begin{thmlist} 
			\item\label{item:condition1'}  $$g(\sum_{i=1}^{q}w_i\beta_i)+\frac{1}{2}\{[\sum_{i=1}^{p}z_i\eta_{i,0},\sum_{i=1}^{p}z_i\eta_{i,0}]\}=0,$$ where $g:H^0(D^{(1)},E)\to H^2(X,E)$ is the Gysin morphism defined in \cref{def:Gysin}. 
			\item \label{item:condition2'}  
			\begin{align*} 
				\{	[\iota^*\sum_{i=1}^{p}z_i\eta_{i,0},\sum_{j=1}^{q}w_j\beta_j]\}=0\in H^1(D^{(1)},E),
			\end{align*} where $\iota:D^{(1)}\to X$ denotes  the natural morphism.
			\item \label{item:condition3'}For any $i\neq j$,  
			\begin{align*} 
				[\iota_{ij}^*(\sum_{\ell=1}^{q}w_k \beta_k|_{D_j}),\iota_{ji}^*(\sum_{k=1}^{q}w_k \beta_k|_{D_i}) ]=0\in H^0(D_i\cap D_j,E),
			\end{align*} where $\iota_{ij}:D_i\cap D_j\to D_i$ is the natural morphism.
		\end{thmlist}
		For any $(z,w)\in \sQ'$,  let	 $\eta_k(z,w)$ be defined as in \eqref{eq:etak}. 
		Then there exists some $\ep>0$ such that, for every $(z,w)\in \sQ_\ep:=\sQ'\cap \bB_\ep^{p+q}$,    $\sum_{i=1}^{\infty}\eta_k(z,w) $ converges and lies in $A^1(X,D,E)$.  Moreover, for the connection 
		\begin{align}\label{eq:universal}
			\nabla_{z,w}:=\nabla+\sum_{i=1}^{\infty}\eta_k(z,w),   
		\end{align} 
		it is an analytic family of flat connections for $V|_{X_0}$ with logarithmic poles  around $D$ parametrized by $(z,w)\in \sQ'_\ep$. 
	\end{thm} 
	\begin{proof}
		Suppose that $(z,w) \in \sQ'$.    Then $\eta_1(z,w)$ and 
		$$\Res_D\eta_2(z,w)=-\sum_{j=1}^{q}w_j\beta_j$$ satisfy the three conditions in \cref{main1}. Moreover,   $\eta_k(z,w)$ defined in \eqref{eq:etak} satisfies the inductive equations in the proof of \cref{thm:construction} for each $k\geq 1$. Therefore, we have
		$$
		\bnabla\eta_k(z,w)+\sum_{i=1}^{k}\eta_{i}(z,w)\wedge\eta_{k-i}(z,w)=0
		$$
		for $k\geq 3$.  By \eqref{eq:cata},   we have 
		\begin{equation}
			\|\eta_k(z,w)\|_{W^m} \leq \frac{1}{2C} \Lambda(z,w)^k,
		\end{equation}
		where $$\Lambda(z,w): = 2C \|\eta_1(z,w)\|_{W^m} + \sqrt{4C^2 \|\eta_1(z,w)\|_{W^m}^2 + 4C \|\eta_2(z,w)\|_{W^m}}.$$ 
		By the definitions of $\eta_1$ and $\eta_2$, there exists some $\ep\in (0,\frac{1}{2})$ such that 
		$
		\Lambda(z,w) \leq \frac{1}{2}
		$ for any $(z,w)\in \sQ_\ep'$. Consequently,  $\|\eta(z,w)\|_{W^m} < \infty$ for all $(z,w) \in \sQ'_\ep$. The bootstrapping arguments used in the proof of \cref{thm:construction} imply that $\eta(z,w) \in A^1(X,D,E)$ and satisfies 
		$$\nabla\eta(z,w) + \eta(z,w)\wedge\eta(z,w) = 0.$$
		Thus,   we conclude that $\nabla_{z,w} := \nabla + \eta(z,w)$ defines an analytic family of flat logarithmic connections on $V$, parametrized by $(z,w) \in \sQ'_\ep$. This completes the proof.
	\end{proof}

	\subsection{Constructing the universal connection (II)}
	In this subsection, we prove that the connection constructed in \cref{main1} is universal. The argument closely follows that of \cref{thm:ES}, and we include all details for the sake of completeness.
	
	Consider the dgla  $T^\bullet:=(A^\bullet(X,D,E),\nabla)$. Define $L^\bullet:=(\ker \nabla',\nabla'')$.     Let $\varepsilon : T^0 \to E_x$, $\varepsilon : L^0 \to E_x$, and $\varepsilon : C^0 \to E_x$ be the evaluation maps at a fixed base point $x\in X_0$, which define the corresponding augmentations. 
	Then by \cref{thm:zigzag}, we have the 1-quasi-equivalences  
	$$
	(T^\bullet,d) \stackrel{\phi}{\leftarrow} (L^\bullet,d)\stackrel{\psi}{\rightarrow} (C^\bullet,d). 
	$$
	By \cref{thm:GM}, for any Artin local $\bC$-algebra,  the natural maps 
	\begin{align}
		[\sC(T^\bullet,\ep,A)/\exp(T^0\otimes\km)]\leftarrow	  [\sC(L^\bullet,\ep,A)/\exp(L^0\otimes\km)]\to   [\sC(C^\bullet,\ep,A)/\exp(C^0\otimes\km)]
	\end{align}
	are equivalences of groupoids.  
	
	We note that  augmentation 
	$$
	\ep:C^0\to E_x
	$$
	defined by the evaluation is injective. In this case, we have the following result. Let $\sQ$ be the quasi-homogeneous cone in $C^1=C^{1,0}\oplus C^{0,1}$ defined by 
	\begin{align}\label{eq:defquasi}
		\sQ_0:=	 \left\{u\in  C^1  \mid   d u+\frac{1}{2}[u,u]=0\right\}.
	\end{align} 
	Denote by 
	\begin{align}\label{eq:sQ2}
		\sQ:=\mathscr{Q}_0\times \kg / \varepsilon(C^0). 
	\end{align} 
	\begin{lem} \label{lem:GM2}
		The analytic germ $(\sQ,0)$ (non-canonically) prorepresents the functor  
		\begin{align*}
			{\rm Art}&\to {\rm Grp}\\ 
			A&\mapsto \Def (C^{\bullet}, \ep, A )
		\end{align*}
		where $\Def (C^{\bullet}, \ep, A )$ is the small groupoid  defined in \eqref{eq:defaug}. 
	\end{lem}\begin{proof}
		Let $\ell := \dim_{\mathbb{C}}(\mathfrak{g}/\ep(C^0))$. We choose a subset  $\{v_{1},\ldots,v_{\ell}\}$ of $\mathfrak{g}$ such that 
		the images of $\{v_1,\ldots,v_{\ell}\}$ form a basis for the quotient space $\mathfrak{g}/\ep(C^0)$.  
		Let $W$ be the subspace of $\kg$ spanned by   $\{v_1,\ldots,v_\ell\}$.  Then  the natural map $p:W\to  \kg/\ep(C^0)$ is an isomorphism.  Since $
		\ep:C^0\to E_x
		$  and $\ep:D^0\to E_x$
		defined by the evaluation is injective, it follows that  $\exp \left(C^0 \otimes \mathrm{m}\right)$ acts freely on $G_A^0:=\exp(\kg\otimes\km)$  by left-multiplication.  Hence the following (non-canonical) maps 
		\begin{align*}
			\kg/\ep(C^0)\otimes \km 
			\;\leftarrow\; & W\otimes\km 
			\;\rightarrow\; \frac{\exp(\kg\otimes\km)}{\exp(\ep(C^0)\otimes \km)} \\
			p(r) 
			\;\mapsfrom\; & r 
			\;\mapsto\; \exp(\ep(C^0)\otimes \km)\cdot \exp(r).
		\end{align*}
		are both bijective.  Here   $\exp(\ep(C^0)\otimes \km)\cdot \exp(r)$ is the right coset of $\exp(r)$.  
		Hence  the following maps
		\begin{align}\label{eq:first2}
			\sC(C^\bullet,A)\times W\otimes\km &\to	 \frac{\sC(C^\bullet,\ep,A)}{\exp(C^0\otimes \km)}  \\ \nonumber
			(\alpha,r) &\mapsto  [(\alpha,\exp(r))]. 
		\end{align} 
		and
		\begin{align}\label{eq:first22}
			\sC(L^\bullet,A)\times W\otimes\km &\to	 \frac{\sC(L^\bullet,\ep,A)}{\exp(L^0\otimes \km)}  \\ \nonumber
			(\alpha,r) &\mapsto  [(\alpha,\exp(r))]. 
		\end{align} 
		are both equivalences of groupoids. 
		
		\medspace
		
		Since $C^1$ is a finite-dimensional complex vector space, we view $C^1 \times W$ as an affine space over $\mathbb{C}$.  Let $\cI$ be the ideal of $\bC[C^1\times W]$ generated by the  quasi-homogeneous equations 
		\begin{align*}
			dq(u)+\frac{1}{2}[q(u),q(u)]=0
		\end{align*} 
		for 	  $u\in  C^1\times W  $, 
		where $q:C^1\times W\to C^1$ denotes the projection map,     $d \colon C^1 \to C^2$ is the differential of the dgla $(C^\bullet,d)$ defined in \eqref{eq:defdiff}, and $[-,-]$ is the Lie bracket defined in \eqref{eq:Lie}.    It follows that 
		$$
		\spec\bC[C^1\times W]/\cI=\sQ.
		$$
		Let $\{\alpha_{1},\ldots,\alpha_{p}\}$ be a basis of $\cH^1(X,E)$, and we denote abusively $\{\alpha_{1},\ldots,\alpha_{p}\}$   their cohomology classes in $H^1(X,E)=C^{0,1}$. Let  $\{\beta_1,\ldots,\beta_q\}\in  C^{1,0}$ be a basis.   Then, any elements $\alpha\in C^1\otimes\km$  and $r\in W\otimes \km$,   can be written as 
		\begin{align}\label{eq:alpha2}
			\alpha=(\sum_{i=1}^{p} \alpha_i \otimes f_i, \sum_{\gamma=1}^{q} \beta_\gamma\otimes g_\gamma), \quad r=\sum_{j=1}^{\ell} v_j \otimes h_j,
		\end{align}
		where $f_i,g_\gamma,h_j\in \km$. Hence, $(\alpha,r)$ gives rise to a morphism
		\begin{align*}
			\Hom(\bC[C^1\times W], A)\\
			((\alpha_i)^*,(\beta_\gamma)^*,(v_j)^*)\mapsto  f_i+g_\gamma+h_j
		\end{align*} 
		Here we consider $(\alpha_i)^*\in (C^{0,1})^*$, $(\beta_\gamma)^*\in (C^{1,0})^*$ and $(v_j)^*\in W^*$. 
		Note that the above morphism factors through
		\[
		\Hom(\bC[C^1 \times W]/\cI, A)
		\]
		if and only if 
		$$d(\alpha)+\frac{1}{2}[\alpha,\alpha]=0.$$
		Equivalently, this holds if and only if
		\[
		\alpha \in \sC^1(C^\bullet,A).
		\]
		Let \(\cI_1\) be the ideal of \(\bC[\sQ]\) generated by the image of 
		\((C^1\times W)^* \subset \bC[C^1\times W]\) under the natural projection 
		\(\bC[C^1\times W] \to \bC[C^1\times W]/\cI\).
		Let \(\widehat{\bC[\sQ]}_{\cI_1}\) denote the \(\cI_1\)-adic completion of \(\bC[\sQ]\).
		Since $A$ is an Artin local $\bC$-algebra, it follows that 
		\begin{align*}
			\Hom(\bC[C^1\times W]/\cI, A) \simeq 	\Hom(\widehat{\bC[\sQ]}_{\cI_1}, A). 
		\end{align*} 
		By standard properties of completion of quotient rings, there is a natural isomorphism
		\[
		\widehat{\bC[\sQ]}_{\cI_1} \;\simeq\; 
		\frac{\bC[[C^1\times W]]}{\cI},
		\]
		where, by abuse of notation, we continue to denote by \(\cI\) the ideal  
		\(\cI\cdot  \bC[[C^1\times W]]\).  
		
		On the other hand, \(\cI_1\)-adic completion commutes with localization at the
		maximal ideal \(\cI_1\) of \(\bC[\sQ]\).
		Therefore  we have a natural isomorphism
		$$
		\widehat{\bC[\sQ]}_{\cI_1} \stackrel{\simeq}{\to}
		\widehat{\cO_{\sQ,0}}.
		$$  
		Therefore, the following map 
		\begin{align}\label{eq:precise2}
			\sC(C^\bullet,A)\times W\otimes\km&\to 	\Hom( \widehat{\cO_{\sQ,0}}, A) \\ \nonumber
			(\alpha,r) &\mapsto \left(   ((\alpha_i)^*,(\beta_\gamma)^*,(v_j)^*)\mapsto  f_i+g_\gamma+h_j \right)
		\end{align}
		is an equivalence, where $(\alpha,r)$ is written as in~\eqref{eq:alpha2}, and  $ \alpha_i ^*$, $\beta_\gamma ^*$, and $v_j^*$ are understood as the images of  $ \alpha_i ^*$, $\beta_\gamma ^*$, and $v_j^*$  under the composite morphisms 
		$$\bC[C^1\times W]\to \bC[\sQ]\to \cO_{\sQ,0}\to\hcO_{\sQ,0}.$$  
		We remark that $\widehat{\cO_{\sQ,0}}$ is generated by 
		$ \alpha_i ^*$, $\beta_\gamma ^*$, and $v_j^*$.  
		Therefore, for any morphism in 
		$\Hom(\widehat{\cO_{\sQ,0}},A)$, the images of these generators
		completely determine the morphism.  
		Combining \eqref{eq:precise2} with \eqref{eq:first2} yields the lemma. 
	\end{proof}  
	
	Consider the  polynomial ring $\bC[z_1,\ldots,z_p,w_1,\ldots,w_q]$. 
	Let $\mathscr{Q}'\subset \bC^{p+q}$ denote  the quasi-homogeneous cone defined  in \eqref{eq:quasihomo}.   We set
	\begin{align*} 
		\sQ'':=\sQ'\times \bC^\ell. 
	\end{align*}
	Write $\eta_{i,0}:=\alpha_i$ and  let $\eta_{0,\gamma}:=f(\beta_\gamma)\in A^1(X,D,E)$ be defined in \eqref{eq:f}.   Then $\Res_D\eta_{0,\gamma}=-\beta_\gamma$ and we have 
	$$\nabla f(\beta_\gamma)\in \cH^2(X,E)$$   by our construction. 
	Let 
	\begin{align}\label{eq:universal2}
		\nabla_{z,w}:=\nabla+\sum_{i=1}^{\infty}\eta_k(z,w)
	\end{align}  
	be the connection parametrized by $(z,w)\in \sQ'_\ep$, 
	defined in \eqref{eq:universal}. 
	
	By \cref{lem:nablaexact}, we know for each $(z,w)\in \sQ'$,  each 
	$
	\eta_k(z,w) \in A^1(X,D,E)
	$ is $\bnabla'$-exact for $k\geq 2$.   Hence we have
	\begin{align}\label{eq:nabla00}
		\nabla' \eta_k(z,w) =\bnabla' \eta_k(z,w) =0\quad \mbox{ for }k\geq 1.
	\end{align} 
	Fix a  $k\in\bZ_{> 0}$.  Set $A_k:= \cO_{\mathscr{Q}'',0}/\km^{k+1}$, that is an  Artin local $\bC$-algebra. Let $\km$ be the maximal ideal of $\cO_{\mathscr{Q}'',0}$.  Write $\km_k:=\km A_k$, that is the maximal ideal for $A_k$.   We now regard $z_I$,  $w_\Gamma$, and $x_j$ as the image of   $z_I$, $w_\Gamma$ and $x_j$ under the morphism 
	$$\cO_{\bC^{p+q+\ell},0}\to \cO_{\mathscr{Q}'',0}\to \cO_{\mathscr{Q}'',0}/\km^{k+1}=A_k.$$ Then we have $z_I\in \km_k$ and $w_\Gamma\in \km_k$ for $|I|\geq 1$ and $|\Gamma|\geq 1$, and $x_i\in \km_k$. Thus   by \eqref{eq:nabla00}, 
	we can consider $$ \eta_i(z,w)  \in L^{1,0}\otimes\km_k $$  
	for any $i\in \bN$. It follows that 
	$$ \eta_j(z,w)=0  \in L^{1,0}\otimes\km_k $$  
	for any $j\geq 2k+1$. 
	By \cref{main1}, we have
	$$
	\nabla   \big(\sum_{j=1}^{2k}\eta_j(z,w)\big) +\frac{1}{2}[ \sum_{j=1}^{2k}\eta_j(z,w), \sum_{j=1}^{2k}\eta_j(z,w)]=0\in L^1\otimes \km_k.   
	$$
	This implies that, 
	\begin{align}\label{eq:2k}
		\sum_{j=1}^{2k}\eta_j(z,w) \in  \obj\Def(L^\bullet,  A_k).
	\end{align}  
	Consider the morphism $\psi:L^\bullet\to C^\bullet$  in \eqref{eq:psi}. By \cref{lem:nablaexact}, each 
	$
	\eta_j(z,w)
	$ is $\bnabla'$-exact for any $j\geq 2$.  By \eqref{eq:residuenull} and  \eqref{eq:psi}, for any $k\geq 2$,    we have
	$$
	\psi(\sum_{j=1}^{2k}\eta_j(z,w))=\sum_{i=1}^{p}z_i \eta_{i,0} +\sum_{\gamma=1}^{q}w_\gamma \beta_{\gamma}   \in \sC(C^\bullet,A_k). 
	$$ 
	Here we denote abusively by $\{\eta_{1,0},\ldots,\eta_{p,0}\}$  and   $\{\beta_{1},\ldots,\beta_{q}\}$   their cohomology class in $H^1(X,E)=C^{0,1}$ and  $H^0(D^{(1)},E)=C^{1,0}$ respectively. 
	Therefore, using   the isomorphisms of groupoids defined in \eqref{eq:first2} and \eqref{eq:precise2}, we can see   that  we have  the isomorphism  
	\begin{align}\label{eq:Phi2}
		\Phi:  	\sC(L^\bullet,A)\times W\otimes\km_k  \to	\sC(C^\bullet,A)\times W\otimes\km_k {\to} \Hom(\hcO_{\sQ,0},A_k)
	\end{align}  
	such that 
	$$  
	\Phi	( \sum_{j=1}^{2k}\eta_j(z,w),  \sum_{j=1}^{\ell}x_jv_j)	=\left(	(\eta_{i,0}^*,\beta_{\gamma}^*,v_j^*)\mapsto  z_i+w_\gamma+x_j \right).$$
	As remarked above,  since $\widehat{\cO_{\sQ,0}}$ is generated by 
	$ \eta_{i,0}^*$, $\beta_{\gamma}^*$, and $v_j^*$,  any morphism in 
	$\Hom(\widehat{\cO_{\sQ,0}},A_k)$ is determined by  images of these generators.

	We state and prove the universal property of the connection constructed in \cref{main1}.
	\begin{thm}\label{thm:universal2}
		Let $X$ be a compact K\"ahler manifold and let $D$ be a simple normal crossing divisor on $X$. Assume that $(V,\nabla)$ be a unitary flat bundle of rank $N$ on $X$. We denote by $(E,\nabla):=(\End(V),\nabla)$ and by $\varrho:\pi_1(X)\to \GL_N(\bC)$ the monodromy representation of $(V,\nabla)$.   Let \(\nabla_{z,w}\) denote the analytic family of flat connections on \(V|_{X_0}\), 
		parametrized by \((z,w) \in \sQ'_{r}\) for some \(r>0\), as defined in \eqref{eq:universal2}.
		Let  ${\rm Mon}(\nabla_{z,w}) $ be the monodromy representation for any $(z,w)\in \mathscr{Q}'_r$.   Let $R(X_0,N)$ be the representation variety  for $\pi_1(X_0,x)$ into $\GL_N$. Consider the evaluation map $\operatorname{ev}_x: H^0(X,E) \to E_x$. Let $\{v_1, \ldots, v_\ell\} \subset E_x$ be a subset such that its image forms a basis of the quotient space $E_x / \operatorname{Im}(\operatorname{ev}_x)$.  Then  the natural map between analytic germs
		\begin{align}\label{eq:localiso}
			f:	\left(\mathscr{Q}'_r\times \bC^\ell,0\right) &\to (R(X_0,N),\varrho)\\\nonumber
			(z,w;x)&\mapsto \exp(-\sum_{i=1}^{\ell}x_iv_i){\rm Mon}(\nabla_{z,w})\exp(\sum_{i=1}^{\ell}x_iv_i) 
		\end{align}
		is an isomorphism. 
	\end{thm} 
	\begin{proof}
		We recall that 	 we have morphisms of groupoids 
		\begin{equation}\label{eq:great2}\hspace{-0.4cm}
			\begin{tikzcd}[column sep=1em]
				& \Hom\!\left(\widehat{\cO}_{\sQ,0},A_k\right)		& \\
				{ [R_\varrho(A_k)/{\rm Id}] }	 &	\sC(L^\bullet,A_k) \times W\otimes\km_k \arrow[r,"\psi\times \mathrm{Id}"] \arrow[d]\arrow[u,"\Phi"]
				& 	\sC(C^\bullet,A_k)  \times W\otimes\km_k  \arrow[ul] \arrow[d]  \\
				{[\sC(T^\bullet,\ep,A_k)/\exp(T^0\otimes\km_k)]} \arrow[u,"h"']
				& {[\sC(L^\bullet,\ep,A_k)/\exp(L^0\otimes\km_k)]} \arrow[r,"\psi"]\arrow[l,"\phi"'] 
				& {[\sC(C^\bullet,\ep,A_k)/\exp(C^0\otimes\km_k)]}    
			\end{tikzcd}
		\end{equation}
		in which each of them is an isomorphism by \eqref{eq:first2}, \eqref{eq:first22},  and \cref{thm:zigzag,thm:GM,lem:holonomy2}.     Here \begin{align*}
			h:  [\sC(T^\bullet,\ep_x,A_k)/\exp(T^0\otimes\km)]&\to  [R_\varrho(A_k)/{\rm Id}]\\
			(\alpha,e^r) &\mapsto  \exp(-r) {\rm hol}_x(\nabla+\alpha) \exp(r)
		\end{align*}  
		is defined in   \Cref{lem:holonomy2}. 
		We denote by
		$$
		\Upsilon: \sC(L^\bullet,A_k) \times W\otimes\km \to  [R_\varrho(A_k)/{\rm Id}]
		$$
		the induced equivalence of groupoids in \eqref{eq:great2}. Then, for $$ \sum_{j=1}^{2k}\eta_j(z,w)\in \sC(L^\bullet,A_k)$$ defined in \eqref{eq:2k}, 
		$$
		\Upsilon(\sum_{j=1}^{2k}\eta_j(z,w),\sum_{j=1}^{\ell}x_jv_j)=\exp(-\sum_{i=1}^{\ell}x_iv_i){\rm Mon}(\nabla_{z,w})\exp(\sum_{i=1}^{\ell}x_iv_i), 
		$$
		corresponds to the morphism 
		$$
		f_k: \hcO_{R(X_0,N),\varrho}\to  A_k=\cO_{\mathscr{Q}'',0}/\km^{k+1}
		$$
		induced by $f$. 
		
		Let us define a morphism  
		$
		g_k\in \Hom(\hcO_{\sQ,0},A_k) 
		$
		by setting
		\begin{align}\label{eq:g_k2}
			g_k(\alpha_i^*)=  z_i,\quad 		g_k( \beta_\gamma ^*)=  w_\gamma, \quad g_k(v_j^*)=x_j. 
		\end{align}  By \eqref{eq:Phi2}, we have
		$$
		\Phi (\sum_{j=1}^{2k}\eta_j(z,w), \sum_{i=1}^{\ell}x_iv_i)=g_k. 
		$$

		One can see the compatibility of $ {g}_k$ by
		\[
		\begin{tikzcd}
			& \vdots \arrow[d] \\
			& A_{k+1} \arrow[d] \\
			\hcO_{\mathscr{Q},0}  \arrow[ur," {g}_{k+1}"] \arrow[r," {g}_k"']  \arrow[ddr,"g_1"']
			& A_k \arrow[d] \\ 
			& \vdots\arrow[d]\\
			&	A_1
		\end{tikzcd}
		\]
		In conclusion, there exists   natural maps \begin{equation*}
			\begin{tikzcd} 
				\Hom\!\left(\widehat{\cO}_{\sQ,0},A_k\right)
				\arrow[r] 
				&
				\Hom\!\left(\widehat{\cO}_{\sQ,0},A_{k-1}\right) 
				\\
				\sC(C^\bullet,A_k) \times W\otimes\km_k 
				\arrow[r]
				\arrow[u,"\Phi"]
				\arrow[d,"\Upsilon"']
				&
				\sC(C^\bullet,A_{k-1}) \times W\otimes\km_{k-1}
				\arrow[u,"\Phi"']
				\arrow[d,"\Upsilon"]
				\\
				\Hom(\widehat{\cO}_{{R(X_0,N)},\varrho}, A_k)
				\arrow[r]
				&
				\Hom(\widehat{\cO}_{{R(X_0,N)},\varrho}, A_{k-1})
			\end{tikzcd}
		\end{equation*}
		such that all vertical maps are isomorphisms.   We have the inverse limit
		$$
		\varprojlim_k g_k\in  \Hom\!\left(\widehat{\cO}_{\sQ,0},\varprojlim_k A_k\right)=\Hom\!\left(\widehat{\cO}_{\sQ,0},\widehat{\cO}_{\sQ'',0}\right)
		$$
		and by \eqref{eq:g_k2}, it is an isomorphism.
		By \eqref{eq:great2}, it follows that
		$$
		f^*=\varprojlim_k f_k: \widehat{\cO}_{R(X_0,N),\varrho}\to   \hcO_{\mathscr{Q}'',0},
		$$
		is  also an isomorphism. 
		By \cref{thm:GM2}, this implies that 
		$f:  (\mathscr{Q}'_r\times \bC^\ell,0 )  \to (R(X_0,N),\varrho)$ is an \emph{isomorphism} between analytic germs. 
		The theorem is therefore proved. 
	\end{proof}
	
	\begin{rem}
		The above theorem shows that, for extendable unitary representation $\varrho:\pi_1(X_0)\to \GL_N(\bC)$,  the analytic germ of the representation variety $R(X_0,N)$ at $\varrho$ is isomorphic to a quasi-homogeneous singularity defined by generators of weights $1$ and $2$, subject to relations of weights $2, 3,$ and $4$. Similar result for representations with finite image were obtained by Kapovich-Millson \cite{KM98} based on Morgan's theorems \cite{Mor78}.  See also \cite{Lef19}.
	\end{rem}
	\section{Two-step phenomenon}
	\subsection{2-step phenomenon}\label{sec:2-step}
	We will prove the main technical theorem of this paper, the so-called \emph{2-step phenomenon}.  Its proof is based on \cref{thm:universal2,thm:construction}.

	Throughout this subsection,    let $X$  be a compact K\"ahler manifold and let $D=\sum_{i=1}^{m}D_i$  be a simple normal crossing divisor on $X$.  Set $X_0:=X\backslash D$.     Let $(V,\nabla)$ be the trivial flat bundle of rank $N$  and we denote by $\varrho_0:\pi_1(X_0)\to \GL_N(\bC)$ the trivial representation. We denote by $(E,\nabla) := (\operatorname{End}(V), \nabla)$ the induced bundle of endomorphisms. Let $\sQ'\subset \bC^{p+q}$ be the quasi-homogeneous cone defined in \eqref{eq:quasihomo}.  Fix some $r>0$ sufficiently small as in \cref{thm:universal2}, let $\cT:=\sQ'\cap \bB_r^{p+q}$. Then the universal connection $$	\nabla_{z,w}:=\nabla+\sum_{i=1}^{\infty}\eta_k(z,w)$$
	is the connection parametrized by $(z,w)\in \cT$ defined in \eqref{eq:universal2}. It   induces a tautological flat connection   on  $V\otimes \cO(\cT)$ that corresponds to  a representation
	$$
	\varrho_{\cT}:	\pi_1(X_0) \to   \GL_N(\cO_{\cT,0}). 
	$$
	Let $\km$ be the maximal ideal for the local ring $\cO_{\cT,0}$.   For each $i$, we have 
	\begin{align}\label{eq:Ttruncatei2}
		\varrho_{\cT,i}: \pi_1(X_0)  \stackrel{\varrho_{\cT}}{\to} \GL_N(\cO_{\cT,0})\to  
		\GL_N(\cO_{\cT,0}/\km^i). 
	\end{align} 
	Note that as a complex vector bundle, we have a decomposition
	\begin{align} \label{eq:decompose}
		\mathbf{V}_k\stackrel{C^\infty}{\simeq} \oplus_{i=0}^{k-1}V\otimes {\km^i/\km^{i+1}}=:V_k.
	\end{align} 
	With respect to this decomposition, we will write the connection explicitly. 
	
	Fix an arbitrary $(z_0,w_0) \in \cT$.  Note that $\sQ'$ is a quasi-homogeneous cone with weight 1 in $z$ and weight 2 in $w$.  
	We then can define a holomorphic curve $\varphi: \bD_{1+\ep} \to \cT$ by $\varphi(t) = (t  z_0, t^2 w_0)$ for some small $\ep>0$. Then we have the following commutative diagram:
	\begin{equation}\label{eq:bigdiagram}
		\begin{tikzcd}
			\pi_1(X_0) \arrow[rrrd,controls={+(3,2) and +(2.5,1.4)},"\sigma_k"]\arrow[r,"\varrho_{\cT}"]\arrow[dr,"\varrho_{\cT,k}"'] & \GL_N(\cO_{\cT,0}) \arrow[d] \arrow[r,"\varphi^*"] & \GL_N(\cO_{\bD_{1+\ep},0}) \arrow[d] \\
			& \GL_N(\cO_{\cT,0}/\mathfrak{m}^k) \arrow[r,"\varphi^*"] & \GL_N(\cO_{\bD_{1+\ep},0}/\mathfrak{m}^k) \arrow[r,equal] & \GL_N(\bC[t]/(t^k))
		\end{tikzcd}
	\end{equation}
	Consider the family of connections over the disk given by $\nabla_t := \nabla_{t  z_0, t^2w_0}$. The monodromy of this family corresponds to  the  representation $\varphi^* \circ \varrho_\cT: \pi_1(X_0) \to \GL_N(\cO_{\bD_{1+\ep},0})$. 
	Let $(V_k,\nabla_k)$ be the corresponding flat bundle for the representation 
	\begin{align} \label{eq:sigma_kdef}
		\sigma_k:\pi_1(X_0)\to \GL_N(\bC[t]/(t^k)). 
	\end{align} Then the connection matrix  $\nabla_k$ is given by
	\begin{equation}\label{eq:connection3new1}
		\nabla_k =
		\begin{pmatrix}
			\nabla\otimes\mathrm{Id}_{t^{k-1}/t^k}
			& t\eta_1(z_0,w_0)
			& t^2\eta_2(z_0,w_0)
			& \cdots
			& t^{k-1}\eta_{k-1}(z_0,w_0) \\
			0
			& \nabla\otimes\mathrm{Id}_{t^{k-2}/t^{k-1}}
			& t\eta_1(z_0,w_0)
			& \cdots
			& t^{k-2}\eta_{k-2}(z_0,w_0) \\
			\vdots & \ddots & \ddots & \ddots & \vdots \\
			0 & \cdots & 0
			& \nabla\otimes\mathrm{Id}_{t/t^2}
			& t\eta_1(z_0,w_0) \\
			0 & \cdots & \cdots & 0
			& \nabla\otimes\mathrm{Id}_{\bC}
		\end{pmatrix}.
	\end{equation}
	This implies that, the monodromy representation for $(V_k,\nabla_k)$ is $(k-1)$-step unipotent.  In particular,  the monodromy for $(V_3,\nabla_3)$ with $\nabla_3$ being
	\begin{equation} \label{eq:connection3new2}
		\nabla_3=  \begin{pmatrix}
			\nabla\otimes \mathrm{Id}_{t^2/t^3} & t\eta_1(z_0,w_0) & t^2\eta_2(z_0,w_0) \\
			0 & \nabla\otimes \mathrm{Id}_{t/t^2} & t\eta_1(z_0,w_0) \\
			0 & 0 & \nabla\otimes \mathrm{Id}_{\bC}
		\end{pmatrix}
	\end{equation}
	is 2-step unipotent. 
	
	Let us prove the following lemma.
	\begin{lem}\label{claim:crucial}
		Let $Y$ be a compact K\"ahler manifold and let $D_Y$ be a simple normal crossing divisor on $Y$.  Let $g:Y\to X$ be a holomorphic map  such that $ g^{-1}(D)\subset D_Y.$  Set $Y_0:=Y\backslash D_Y$.   Let $a:X_0\to A$ be the quasi-Albanese map of $X_0$. If $a\circ f(Y_0)$ is a point, then for any $(z,w)\in \cT$, we have 
		\begin{align}\label{eq:universal3}
			g^*	\nabla_{z,w}:=g^*\nabla+ g^*\eta_2(z,w).
		\end{align}  
		In other words, $\nabla g^*\eta_2(z,w)=0$ and  $g^*\eta_i(z,w)=0$ for each $i\neq 2$, and $(z,w)\in \cT$.  
	\end{lem} 
	\begin{proof}
		It suffices to prove the claim for any $(z_0,w_0) \in \cT$.  
		Since $a\circ g_0(Y_0)$ is a point, 	it follows that 
		$$
		g_0^*H^1(X_0,E)=g_0^*H^1(X_0,\bC)\otimes \bC^{N^2}=\{0\}. 
		$$ 
		Hence we have 
		$
		g^*\eta_1(z_0,w_0)=0. 
		$ 
		By \cref{thm:construction}, we have
		$$
		\nabla\eta_2(z_0,w_0)+\eta_1(z_0,w_0)\wedge \eta_1(z_0,w_0)=0.
		$$
		This implies that, $\nabla g^*\eta_2(z_0,w_0)=0$. Hence we have
		\begin{align}\label{eq:type}
			g^*  \eta_2(z_0,w_0)\in A^1(Y,D_Y, E) \cap \ker\nabla\simeq \left(A^1(Y,D_Y,\bC) \cap \ker d\right)   \otimes \bC^{N^2}.
		\end{align}
		Since $\eta_i(z_0,w_0)\in {\rm Im}\bnabla'$ for any $i\geq 2$,  it follows that
		\begin{align}\label{eq:nabla'exact}
			g^*\eta_i(z_0,w_0)\in {\rm Im}\bnabla' \quad \mbox{for } i\geq 3.
		\end{align} \eqref{eq:type} implies that 
		\begin{align}\label{eq:holo4} 
			g^* \eta_2(z_0,w_0)  \in H^0(Y,\Omega_Y(\log D_Y)  \otimes E)\simeq H^0(Y,\Omega_Y(\log D_Y))   \otimes \bC^{N^2}.
		\end{align} 
		By \cref{item:etak}, for every $k\geq 3$, we have
		\begin{align} \label{eq:eta_k} 
			\bnabla   g^*  \eta_k(z_0,w_0) =- \sum_{i=1}^{k-1}g^*\eta_i(z_0,w_0)\wedge g^*\eta_{k-i}(z_0,w_0). 
		\end{align} 
		$g^*\eta_3(z_0,w_0)\in  \ker\bnabla$.    Together with \eqref{eq:nabla'exact} and \cref{lem:ddbar}, we have $g^*\eta_3(z_0,w_0)=0$. 
		
		By \eqref{eq:eta_k} and \eqref{eq:nabla'exact},  we have
		$$ 
		\bnabla'' g^*  \eta_4(z_0,w_0)=	\bnabla g^*  \eta_4(z_0,w_0)=-g^*\eta_2(z_0,w_0)\wedge g^*\eta_2(z_0,w_0).
		$$
		However, \eqref{eq:nabla'exact} implies that
		$$
		g^*\eta_i(z_0,w_0)\in A^{1,0}(Y,D_Y,E)\quad \mbox{for } i\geq 3.
		$$
		Since $$g^*\eta_2(z_0,w_0)\wedge g^*\eta_2(z_0,w_0)\in A^{2,0}(Y,D_Y,E),$$
		this implies that $\bnabla''g^*  \eta_4(z_0,w_0)=0$. Together with \eqref{eq:nabla'exact} and \cref{lem:ddbar}, we have $g^*  \eta_4(z_0,w_0)=0$. 
		
		We shall prove the claim by induction, and assume that   $g^*\eta_i(z_0,w_0)=0$ for $i=4,\ldots,k-1$ for some $k\geq 5$.  
		It then follows from \eqref{eq:eta_k}  that $g^*\eta_k\in \ker\bnabla$.   Together with \eqref{eq:nabla'exact} and \cref{lem:ddbar}, we have $g^*\eta_k(z_0,w_0)=0$.  This completes the induction and thus the claim is proved.  
	\end{proof}	
	
	We call the following result the \emph{2-step phenomenon}, in analogy with the 1-step phenomenon in \cref{thm:1-step}. Although we do not use it in the proofs of \cref{main,main:2step}, we believe it is of independent interest.
	\begin{proposition} [2-step phenomenon]\label{thm:final}
		Let $Y$ be a compact K\"ahler manifold and let $D_Y$ be a simple normal crossing divisor on $Y$.  Let $g:Y\to X$ be a holomorphic map  such that $ g^{-1}(D)\subset D_Y.$ 
		Write $g_0:=g|_{Y_0}$.   Let $\varrho_{\cT,3}: \pi_1(X_0)  \to  
		\GL_N(\cO_{\cT,0}/\km^3) $ be the representation defined in \eqref{eq:Ttruncatei2}. Assume that  
		\begin{thmlist}
			\item   either the pullback $g_0^*\varrho_{\cT,3}
			: \pi_1(Y_0)  \to  
			\GL_N(\cO_{\cT,0}/\km^3) 
			$  is a trivial representation;
			\item or there there exists some quasi-K\"ahler compact manifolds such that $g_0 $ factors through 
			$$
			g_0: Y_0\stackrel{r_0}{\to} Z_0\stackrel{h_0}{\to}X_0
			$$
			satisfying 
			\begin{align}\label{eq:contracalba}
				\operatorname{alb}_{X_0}(h_0(Z_0)) = \{\mathrm{pt}\} \\ \label{eq:contracalba2}
				\operatorname{alb}_{Z_0}(r_0(Y_0)) = \{\mathrm{pt}\},
			\end{align} 
			where $\operatorname{alb}_{X_0} \colon X_0 \to \operatorname{Alb}(X_0)$ and $\operatorname{alb}_{Z_0} \colon Z_0 \to \operatorname{Alb}(Z_0)$ denote their quasi-Albanese maps.  
		\end{thmlist}      
		Then for   any representation $\tau \colon \pi_1(X_0, x) \to \operatorname{GL}_N(\mathbb{C})$ belonging to the irreducible component of the representation variety $ {R}(X_0, N)$ that contains the trivial representation, the pullback representation $ g_0 ^*\tau$ is trivial.    
	\end{proposition}
	\begin{proof}  
		\noindent{\em Proof of Item (i).} We first  prove that for any $(z_0,w_0)\in \cT$, we have
		\begin{align}\label{eq:bothvanish}
			g^*\eta_i(z_0,w_0)=0 \quad \mbox{for }i=1,2. 
		\end{align} 
		
		We assume by contradiction that \eqref{eq:bothvanish} does not hold. It follows that either 
		$  
		g^*\eta_1(z_0,w_0)\neq 0 ,
		$  
		or $\Res_{D_Y}g^*\eta_2(z_0,w_0)\neq 0$. 
		
		Let $\sQ_Y$ be the quasi-homogeneous cone of $H^1(Y,E)\times H^0(D_Y^{(1)},E)$ with weight $(1,2)$ defined by \eqref{eq:defquasi}. By  the proof of \cref{claim:crucial}, the construction in \cref{thm:construction} is functorial under pullback, and  we have
		$$
		(g^*\eta_1(z_0,w_0),-\Res_{D_Y}g^*\eta_2(z_0,w_0))\in \sQ_Y. 
		$$
		
		We apply   \cref{thm:universal2}   to $Y_0$ and the trivial representation $\tau_0:\pi_1(Y_0)\to \GL_N(\bC)$. Then for the open ball $\bB_\ep\subset H^1(Y,E)\times H^0(D_Y^{(1)},E)$ of radius $\ep>0$,   there exists an analytic isomorphism
		\begin{align} \label{eq:isoQY}
			(\sQ_Y\cap \bB_\ep,0)\stackrel{\simeq}{\to}  (R(Y_0,N),\tau_0) 
		\end{align}given by the monodromy representation of the universal connection constructed in \eqref{eq:universal2}, parametrized by $\sQ_Y\cap \bB_\ep$,  in a similar way as \eqref{eq:localiso}. Then there exists some $\delta>0$ such that
		$$
		(g^*\eta_1(tz_0,t^2w_0),-\Res_{D_Y}g^*\eta_2(tz_0,t^2w_0))=(t g^*\eta_1(z_0,w_0),-t^2\Res_{D_Y}g^*\eta_2( z_0, w_0))\in \sQ_Y\cap \bB_\ep
		$$
		for $t\in \bD_\delta$. Let $\km$ be the maximal ideal of $\cO_{\sQ_Y,0}$. Then the natural morphism
		$$
		\cO_{\sQ_Y,0}/\km^3\to \cO_{\bD_\delta,0}/(t^3)
		$$ 
		induced by 
		$$
		t\mapsto   (t g^*\eta_1(z_0,w_0),-t^2\Res_{D_Y}g^*\eta_2( z_0, w_0))
		$$
		is non-zero. 
			By the construction in \eqref{eq:sigma_kdef}, for the representation $\sigma_3$ defined in \eqref{eq:sigma_kdef},  the isomorphism \eqref{eq:isoQY}  implies that 
			$$
			g_0^*\sigma_3:\pi_1(Y_0)\to \GL_N(\bC[t]/(t^3)). 
			$$
			is a non-trivial representation.  Since $g_0^*\sigma_3$ factors through $g_0^*\varrho_{\cT,3}$ by \eqref{eq:bigdiagram}, this yields a contradiction. Hence, \eqref{eq:bothvanish} holds. 
			By \eqref{eq:universal3}, it implies that
			\begin{align} \label{eq:pullnablazw}
				g^*\nabla_{z,w}=g^*\nabla. 
			\end{align} 
			
			Applying \cref{thm:universal2} to $X_0$ and the trivial representation $\varrho_0:\pi_1(X_0)\to \GL_N(\bC)$, the natural map of analytic germs
			\begin{align*}
				(\mathcal{T}, 0) &\to \left( R(X_0, N),\varrho_0\right) \\
				(z,w) &\mapsto \operatorname{Mon}(\nabla_{z,w})
			\end{align*}
			is an isomorphism. Together with \eqref{eq:pullnablazw}, it yields   that,  for the  natural map
			$$
			g_0^*:    R(X_0, N)\to R(Y_0,N),
			$$
			it maps an analytic neighborhood of $\varrho_0$ in  $R(X_0, N)$ to the trivial representation $\tau_0:\pi_1(Z_0)\to \GL_N(\bC)$ in $R(Y_0,N)$. Hence for any geometric irreducible component $R$ of the representation variety $ {R}(X_0, N)$ that contains the trivial representation, $ g_0^*(R)=\{\tau_0\}$.    We proved Item (i).

			\medskip

			\noindent{\em Proof of Item (ii).}  
			Let $Z$ be a compact K\"ahler manifold that compactifies  $Z_0$ such that $h_0$ extends to a holomorphic map $h:Z\to X$, and $D_Z:=Z\backslash Z_0$ is a simple normal crossing divisor. 
			
			Fix any $(z_0,w_0)\in \cT$.  By \cref{claim:crucial} and \eqref{eq:contracalba}, we have $h^*\eta_1(z_0,w_0)=0$ and
			$$
			h^*\eta_2(z_0,w_0)\in \ker\nabla\cap A^1(Z,D_Z,E). 
			$$
			Since $(E,\nabla)$ is a trivial flat bundle, \eqref{eq:contracalba2} implies that
			$$
			g_0^*\eta_2(z_0,w_0)=r_0^*h_0^*\eta_2(z_0,w_0)=0. 
			$$
			Since $(z_0,w_0)$ is an arbitrary point in $\cT$, it implies that
			$
			g_0^*\varrho_{\cT,3}
			$ 
			is a trivial representation.    Item (ii) then follows from Iteim (i). 		
			The proposition is proved. 
		\end{proof}

		For our applications to \cref{main,main:2step}, a single quasi-Albanese map suffices under mild hypotheses, allowing us to avoid the two-step construction in \cref{thm:final}. The following result plays a crucial role in the proof of \cref{main:2step}.
		\begin{thm}[2-step phenomenon (II)]\label{thm:strong}
			Assume that   the quasi-Albanese map $a:X_0\to A$ of $X_0$ is dominant.  Let $Y_0$ be a connected component of a general smooth fiber of $a$.  Then for any linear representation $\tau:\pi_1(X_0)\to\GL_{N}(\bC)$ that lies in the same geometric irreducible component of the trivial representation   in $R(X_0,N)$, the group $h_0^*\tau(\pi_1(Y_0))$ is abelian, where $h_0:Y_0\to X_0$ denotes the inclusion map.  
		\end{thm}
		\begin{proof}
			Replace $X$ by a smooth bimeromorphic model without modifying  $X_0$, we may assume that $a$ extends to a  surjective morphism $\bar a: {X}\to \overline{A}$, where $\overline A$ is a compact K\"ahler compactification for the semitori  $A$. Let $A^o$ be the Zariski dense open subset of $A$ such that
			\begin{itemize}
				\item $a$ is smooth over $A^o$;
				\item  for each irreducible component $D_i$ of $D$, either it is smooth and surjective over $A^o$, or $\bar a(D_i)\cap A^o=\varnothing$. After reordering $D_1,\ldots,D_m$,  we assume that $D_1,\ldots,D_k$ are all these horizontal divisors. 
				\item Each stratum of $\sum_{i=1}^{k}D_i$ is smooth over $A^o$.
			\end{itemize}
			Fix any point $y\in A^o$ and let $Y$ be a connected component of $\bar a^{-1}(y)$. Set  $D_Y:=D\cap Y$. Then $Y$ is smooth and   $D_Y$ is a simple normal crossing divisor on $Y$. 

			\medskip

			We need the following result.
			\begin{claim}\label{claim:twostate}
				For any $i\in \{1,\ldots,k\}$,   we have injectivity  
				$$
				\iota_i^*: H^1(X,E)\to H^1(D_i,E)
				$$
				is injective, where $\iota_i:D_i\to X$ denotes the inclusion map.     Moreover, for any $\beta\in H^0(D^{(1)},E) $, if $g(\beta)=0$ where $g$ is the Gysin map defined in \cref{def:Gysin}, then $ \beta|_{D_i}=0$ for any horizontal $D_i$ with $i\in \{1,\ldots,k\}$. 
			\end{claim}
			\begin{proof}
				Since $(E,\nabla)$ is a trivial flat bundle, it thus  suffices to prove the claim for the constant sheaf $\bC$.  Recall that   the quasi-Albanese map $a:X_0\to X$ is assumed to be dominant.  Since   there is $(\bC^*)^\ell$-fibration $q:A\to {\rm Alb}(X)$ where ${\rm Alb}(X)$ is the Albanese variety for $X$, by the following commutative diagram
				\begin{equation*}
					\begin{tikzcd}
						X_0\arrow[r] \arrow[d,"a"] & X\arrow[d,"{\rm alb}_X"] \\
						A \arrow[r] & {\rm Alb}(X)
					\end{tikzcd}
				\end{equation*}it follows that the Albanese map  ${\rm alb}_X:X\to {\rm Alb}(X)$  of $X$ is surjective.   Since $D_i$ is horizontal over $A$ for each $i\in\{1,\ldots,k\}$, it follows that 
				\begin{align}\label{eq:surjectivenow}
					{\rm alb}_X(D_i)={\rm Alb}(X).
				\end{align}

				Pick any non-trivial element $\omega\in H^1(X,\bC)$. Then there exists some non-trivial $\omega_0\in H^1({\rm Alb}(X),\bC)$ such that $ {\rm alb}_X^* \omega_0=\omega$. We thus have      $$
				\omega|_{D_i}= ({\rm alb}_X|_{D_i})^* \omega_0\stackrel{\eqref{eq:surjectivenow}}{\neq} 0  .
				$$ 
				The first assertion in the claim is proved.     
				
				Note that we have the short exact sequence
				$$
				H^1(X,D,\bC)\stackrel{\Res_D}{\to} H^0(D^{(1)},\bC)\stackrel{g}{\to}H^2(X,\bC)  
				$$
				where $g$ is the Gysin map defined in \cref{def:Gysin}. 
				Hence there exists some element $\eta\in H^1(X,D,\bC)$ such that $\Res_D\eta=\beta$. 
				
				We denote by $D_A:=\overline{A}\backslash A$, which is assumed to be a simple normal crossing divisor. 
				Since we have a commutative diagram \begin{equation*}
					\begin{tikzcd}
						H^1(A,\bC)\arrow[r] \arrow[d,"a^*"] & H^1(\overline{A},D_{A},\bC)\arrow[d,"\bar a^*"] \\
						H^1(X_0,\bC) \arrow[r] & H^1(X,D,\bC)
					\end{tikzcd}
				\end{equation*} 
				in which  each map   is an isomorphism, 
				there exists some 1-form  $\omega \in A^1(\overline{A},D_{A},\bC)$ such that  
				$ 
				d\omega=0 
				$ and $\{\bar a^*\omega\}=\eta$. However, since $D_i $ is dominant over $A$ for each horizontal $D_i$, it follows that 
				$$
				\beta|_{D_i}=\Res_{D_i}\eta=\Res_{D_i}\bar a^*\omega=0
				$$ for such $D_i$.   The second assertion is proved. We complete the proof of the claim.    
			\end{proof}

			By the construction in \cref{thm:construction,main1}, there exists $\beta\in H^0(D^{(1)},E)$ such that
			$$\eta_2(z_0,w_0)= f(\beta)-\frac{1}{2}\nabla''^*G [\eta_1(z_0,w_0),\eta_1(z_0,w_0)].$$  
			Since $(E,\nabla)$ is trivial,   the pullback 
			\begin{align}\label{eq:pullbackiota}
				H^0(X,E)\xrightarrow[\simeq]{\iota_i^*}   H^0(D_i,E) 
			\end{align} 
			is thus an isomorphism.  Let $\beta_i:=\beta|_{D_i}$.    We denote by  $$\tilde{\beta}_i:=(\iota_i^*)^{-1}({\beta}_i)\in H^0(X,E). $$
			
			\begin{claim}\label{claim:commutesmatrix}
				For any $i,j\in \{1,\ldots,k\}$, we have $[\tilde{\beta}_i,\tilde{\beta}_j]=0$.
			\end{claim}
			Note that the condition in \Cref{item:condition3} of \cref{main1}, only ensures that the claim holds when $D_i\cap D_j\neq \varnothing$. 
			\begin{proof}[Proof of \cref{claim:commutesmatrix}]
				Write $\eta_1=\eta_1(z_0,w_0)$ for short.   By \Cref{item:condition1,item:condition2} in \cref{main1}, for each $i\in \{1,\ldots,k\}$, we have
				\begin{align}\label{eq:harmonicnew}
					0= \{[\iota_i^*\eta_1,\beta_i]\}\in H^1(D_i,E),\quad  0=g(\beta)+\frac{1}{2}\{[\eta_1,\eta_1]\} \in H^2(X,E),   .
				\end{align}  where $g$ is the Gysin map defind in \eqref{def:Gysin}.
				Since $\eta_1$ is harmonic and $(E,\nabla)$ is trivial flat bundle,  
				$$
				[\iota_i^*\eta_1,\beta_i ]=  \iota_i^*[\eta_1,\tilde{\beta}_i] 
				$$ 
				is also harmonic, and by \eqref{eq:harmonicnew}, we have
				$$
				0= \iota_i^*[\eta_1,\tilde{\beta}_i]\in \cH^1(D_i,E) 
				$$ 
				for each $i\in \{1,\ldots,k\}$. It then follows from \cref{claim:twostate}  that
				\begin{align}\label{eq:equaltozeronow}
					[\eta_1,\tilde{\beta}_i] =0. 
				\end{align}

				Let $f$ be the map \begin{align*}  
					f:H^0(D^{(1)},E)\to A^1(X,D,E)
				\end{align*} 
				defined in \eqref{eq:f}. Then by \eqref{eq:ff}   we have 
				$$
				\bnabla f(\beta)= \cH(g(\beta))-\iota_*(\beta)= \cH(g(\beta))-\sum_{i=1}^{m}(\iota_i)_*(\beta_i). 
				$$ 
				Then cohomologically we have
				\begin{align}\label{eq:sumpush}
					\{\sum_{j=1}^{m}(\iota_j)_*[ \tilde{\beta}_i|_{D_j},  \beta_j ]\} &= \{[\tilde{\beta}_i, \sum_{j=1}^{m}(\iota_j)_*(\beta_j)]\}\\\nonumber
					&= \{[\tilde{\beta}_i, \cH(g(\beta))]\}\\\nonumber
					& \stackrel{\eqref{eq:harmonicnew}}{=}-\frac{1}{2}\{[\tilde{\beta}_i, [\eta_1,\eta_1]]\}\\\nonumber
					&=-\{[\tilde{\beta}_i,  \eta_1],\eta_1]]\} \stackrel{\eqref{eq:equaltozeronow}}{=}0.\nonumber
				\end{align}
				Set
				$$
				\beta':=( [ \tilde{\beta}_i|_{D_j},  \beta_j])_{j\in \{1,\ldots,m\}}\in H^0(D^{(1)},E). 
				$$Hence, by \cref{lem:Gysin}, we have
				\begin{align*}
					g(\beta')=  \{\sum_{j=1}^{m}(\iota_j)_*[ \tilde{\beta}_i|_{D_j},  \beta_j ]\}  \stackrel{\eqref{eq:sumpush}}{=}0.
				\end{align*}
				We now apply the second assertion in \cref{claim:twostate} to conclude that
				$$
				[ \tilde{\beta}_i|_{D_j},  \beta_j]=0
				$$
				for $j\in \{1,\ldots,k\}$. Note that
				$$
				[ \tilde{\beta}_i|_{D_j},  \beta_j]=\iota_j^* [\tilde{\beta}_i,\tilde{\beta}_j].
				$$
				It then follows from    \eqref{eq:pullbackiota} that $[\tilde{\beta}_i,\tilde{\beta}_j]=0.$ The claim is proved.  
			\end{proof}
			
			Let $h:Y\to X$ be the natural inclusion map.  Note that $D_Y:=Y\cap D=Y\cap \sum_{i=1}^{k}D_i$.  We write
			$$
			D_{Y,i}:=Y\cap D_i 
			$$
			and set
			\begin{align}\label{eq:betaYi}
				\beta_{Y,i}:= \beta_i|_{D_{Y,i}}\in H^0(D_{Y,i},E)
			\end{align}  for each $i=1,\ldots,k$. 
			Note that 
			\begin{align} \label{eq:residuebeta}
				\Res_{D_Y}h^*\eta_2(z_0,w_0)=(\Res_D\eta_2(z_0,w_0)|_{D_{Y,i}})_{i=1,\ldots,k}=( - \beta_{Y,i})_{i=1,\ldots,k}.
			\end{align}

			Consider the map
			$$
			f_\bC:H^0(D_Y^{(1)},\bC)\to A^1(Y,D_Y,\bC)
			$$
			defined in \eqref{eq:f} with $E$ be replaced by the constant sheaf $\bC$. 
			Define $e_i\in H^0(D_Y^{(1)},\bC)$ with $e_i|_{D_{Y,j}}=\delta_{ij}$. 
			
			\begin{claim}\label{claim:commutative}
				We have
				$$
				h^*\eta_2(z_0,w_0)=  \sum_{i=1}^{k} f_{\bC}(e_i)\otimes \tilde{\beta}_{Y,i},
				$$
				where 
				$\tilde{\beta}_{Y,i} $ is the lift of $\beta_{Y,i}$ of the injection
				\begin{align}
					H^0(Y,E)\hookrightarrow   H^0(D_{Y,i},E) 
				\end{align} 
				induced by the pullback. 
			\end{claim}
			\begin{proof}
				Let
				$$
				f_Y:H^0(D_Y^{(1)},E)\to A^1(Y,D_Y,E)
				$$
				be the map
				defined in \eqref{eq:f}.  By \cref{claim:crucial}, $h^*\eta_2(z_0,w_0)$ is $\bnabla'$-exact,  
				$\nabla h^*\eta_2(z_0,w_0)=0$. Together with \eqref{eq:ff}, \eqref{eq:ffff} and \eqref{eq:residuebeta}, this implies that
				$$
				f_Y(h^*\beta)= h^*\eta_2(z_0,w_0). 
				$$
				Together with the following obvious fact
				$$
				f_Y(h^*\beta)= \sum_{i=1}^{k} f_{\bC}(e_i)\otimes \tilde{\beta}_{Y,i},
				$$
				it  yields  the claim.  
			\end{proof}

			In summary, by \cref{claim:commutesmatrix} and \cref{claim:commutative}, we have
			$$
			h^*\eta_2(z_0,w_0)=  \sum_{i=1}^{k} f_{\bC}(e_i)\otimes \tilde{\beta}_{Y,i},
			$$
			with $f_{\bC}(e_i)\in A^1(Y,D_Y,\bC)$ and $\tilde{\beta}_{Y,i}\in H^0(Y,E)$ satisfying 
			\begin{align}\label{eq:commutativebetaY}
				[\tilde{\beta}_{Y,i},\tilde{\beta}_{Y,j}]=0
				\quad \mbox{for any }  i,j\in \{1,\ldots,k\} . 
			\end{align}
			We then apply \cref{claim:crucial} to conclude that,    the monodromy representation of $h_0^*(V,\nabla_{z_0,w_0})$, denoted by $\tau_{z_0,w_0}:\pi_1(Y_0)\to\GL_N(\bC)$, is given by
			$$
			\gamma\mapsto \exp(-\int_{\gamma}h^*\eta_2(z_0,w_0))=\exp(-\sum_{i=1}^{k}\int_{\gamma}  f_{\bC}(e_i)\otimes \tilde{\beta}_{Y,i}).
			$$ 
			Furthermore, \eqref{eq:commutativebetaY} implies that for any $\gamma_1,\gamma_2\in \pi_1(Y_0)$, we have 
			$$
			\tau_{z_0,w_0}(\gamma_1\cdot \gamma_2)=\tau_{z_0,w_0}(\gamma_2\cdot \gamma_1).  
			$$
			As $(z_0,w_0)$ is an arbitrary point in $\cT$, it implies that
			\begin{align}\label{eq:derivedinkernel}
				\pi_1(Y_0)'\subset  \ker h_0^*\varrho_{\cT}. 
			\end{align}


			For any $\gamma\in \pi_1(Y_0)$, it induces an  algebraic morphism 
			\begin{align*}
				h_\gamma:	R(X_0,N) & \to M_{N\times N}\\
				\tau&\mapsto h_0^*\tau(\gamma).
			\end{align*}   
			We now apply \cref{thm:universal2} and shall use the notation introduced there without restating it. Denote by
			$\varrho_0:\pi_1(X_0)\to \GL_N(\bC)$ the trivial representation. Recall that in \cref{thm:universal2} we have a local analytic isomorphism of germs \begin{align*} 
				f: \left(\cT ,0\right) &\to (R(X_0,N),\varrho_0)\\\nonumber
				(z,w )&\mapsto  {\rm Mon}(\nabla_{z,w}) 
			\end{align*}
			defined in \eqref{eq:localiso}. Fix any $\gamma\in \pi_1(Y_0)'$. By the construction, we have
			$$
			h_\gamma\circ f(z,w)={\rm id}. 
			$$
			Therefore, $h_\gamma$ is equal to the identity matrix ${\rm id}$ in $M_{N\times N}$ in an analytic neighborhood of the trivial representation $\varrho_0$ in $R(X_0,N)$ . Since it is algebraic, it follows that $h_\gamma$ is equal to ${\rm id}$ in any geometric irreducible component of $R(X_0,N)$  containing $\varrho_0$.  Hence, for any  linear representation $\tau:\pi_1(X_0)\to\GL_{N}(\bC)$ that lies in a same geometric irreducible component of $R(X_0,N)$  containing $\varrho_0$,  we have  
			$$
			\pi_1(Y_0)'\subset \ker h_0^*\tau. 
			$$  Therefore, $h_0^*\tau(\pi_1(Y_0))$ is abelian.  The theorem is proved. 
		\end{proof}  
		On the other hand, in the following situation, we still have 1-step phenomenon like the compact setting in \cref{thm:1-step}.
		\begin{proposition}\label{item:1step}
			Let $X$ and $Y$ be two compact K\"ahler manifolds, and let $D$ be a simple normal crossing divisor on $X$.  Set $X_0:=X\backslash D$.     
			Let $(V,\nabla)$ be the trivial flat bundle of rank $N$ on $X$ with monodromy 
			representation $\varrho:\pi_1(X)\to \GL_N(\bC)$.
			Let $g:Y\to X_0$ be a holomorphic map such that     $a\circ g(Y)$ is a point, where $a:X_0\to A$ denotes the quasi-Albanese map of $X_0$.  
			Then for any representation $\tau:\pi_1(X_0)\to \GL_{N}(\bC)$ that is in the same irreducible component  of $\varrho$ in $R(X_0,N)$,  
			the pullback representation	$
			g^*\tau
			$ is a trivial representation.    In particular,   for any unipotent representation $\tau:\pi_1(X_0)\to \GL_{N}(\bC)$,  the pullback $g^*\tau$ is trivial.   
		\end{proposition}
		\begin{proof} 
			We will follow the same proof as \cref{thm:strong} and will use the notation introduced there without restating it.  Since $Y$ has no boundary, the map $g_0$ in \cref{thm:strong} coincides with $g$ in our setting. Fix any $(z_0,w_0)\in \cT$.  Since $a\circ g(Y)$ is a point,  it follows   that
			$g^*\eta_1(z_0,w_0)=0. 
			$  
			By \cref{thm:construction}, we have
			$$
			\nabla\eta_2(z_0,w_0)+\eta_1(z_0,w_0)\wedge \eta_1(z_0,w_0)=0.
			$$
			By \eqref{eq:holo4}, we have   \begin{align}
				g^* \eta_2(z_0,w_0)  \in H^0(Y,\Omega_Y   \otimes E)\simeq H^0(Y,\Omega_Y )   \otimes \bC^{N^2}.
			\end{align} 
			By \cref{thm:construction}, we know that $\eta_2(z_0,w_0)$ is $\bnabla'$-exact; hence, its pullback $g^*\eta_2(z_0,w_0)$ is also $\bnabla'$-exact. By \cref{lem:ddbar}, we have
			$
			g^*\eta_2(z_0,w_0)=0. 
			$ 
			By the same arguments as in \cref{thm:final}, we can show  that $g^*\tau$ is a trivial representation. The theorem is proved.  
		\end{proof}
		\cref{item:1step} will be used to study the Shafarevich conjecture for quasi-compact K\"ahler manifolds with proper Albanese map in next subsection. 
		\subsection{A geometric application}
		As a direct consequence of \cref{item:1step}, we provide a new proof of a result   established by Green--Griffiths--Katzarkov \cite{GGK24} and Aguilar--Campana \cite{AC25}.
		
		\begin{thm}\label{thm:GGK}
			Let $X$ be a quasi-compact K\"ahler manifold such that its quasi-Albanese map $a:X\to A$ is proper. If $\pi_1(X)$ is nilpotent and torsion free, then its Shafarevich morphism is the Stein factorization $g:X\to S$ of $a$. Moreover, the universal covering of $X$ is holomorphically convex. 
		\end{thm}
		
		\begin{proof}
			By the assumption,  there exists a faithful representation 
			$$
			\varrho:\pi_1(X)\to \GL_N(\bC)
			$$
			such that its image is unipotent. 
			
			Let $f:Y\to X$ be any compactifiable holomorphic map from a quasi-compact K\"ahler manifold $Y$ to $X$ such that the image $a(f(Y))$ is a point. 
			Since $a$ is proper, the image $f(Y)$ is contained in a compact fiber. Consequently, by replacing $Y$ with a suitable bimeromorphic modification (or compactification), we may assume that $Y$ is a compact K\"ahler manifold. Under this assumption, \cref{item:1step} implies that $f^*\varrho$ is trivial.   Since $\varrho$ is faithful, it implies that
			$f_*(\pi_1(Y))$ is trivial.

			Next, by the characterization of the quasi-Albanese map, for any compactifiable holomorphic map $f:Y\to X$ from a quasi-compact K\"ahler manifold, if $a(f(Y))$ is not a point, then the image of $f^*H^1(X,\bC)$ is non-trivial. This implies that $f_*\pi_1(Y)$ is infinite. 
			Combining this with the contracting property established above, we conclude that $g$ is the Shafarevich morphism of $X$. Hence we have 
			\begin{align}\label{eq:imagetrivial}
				{\rm Im}[\pi_1(F)\to \pi_1(X)]=\{1\}\quad \mbox{ for each fiber } F \mbox{ of }  g.
			\end{align}

			We now show the last assertion.  
			Let $\widetilde{X}^{\rm ab}$ be a connected component of the fiber product $X \times_A \widetilde{A}$, where $\widetilde A\to A$ denotes the universal covering of $A$, and let $\varphi: \widetilde{X}^{\rm ab} \to W$ be the Stein factorization of the induced map $\widetilde{X}^{\rm ab} \to \widetilde{A}$. 
			
			Let $\pi: \widetilde{X} \to \widetilde{X}^{\rm ab}$ be the universal cover of $\widetilde{X}^{\rm ab}$ (which is also the universal cover of $X$), and consider the composition $\Phi := \varphi \circ \pi: \widetilde{X} \to W$. By \eqref{eq:imagetrivial},  for any $y \in W$, its preimage $\Phi^{-1}(y) = \pi^{-1}(\varphi^{-1}(y))$ in $\widetilde{X}$ is a disjoint union of copies of the compact connected space $\varphi^{-1}(y)$. In particular, the connected components of the fibers of $\Phi$ are compact.
			
			By a theorem of Cartan \cite{Car60}, the set $Z$ of connected components of the fibers of $\Phi$ can be endowed with the structure of a normal complex space such that $\Phi$ factors as $h \circ  {\psi}$, where $ {\psi}: \widetilde{X} \to Z$ is a proper holomorphic fibration and $h: Z \to W$ is a map with discrete fibers. This yields the following commutative diagram:
			\begin{equation*}
				\begin{tikzcd}
					\widetilde{X} \arrow[d, " {\psi}"'] \arrow[r, "\pi"]\arrow[dr,"\Phi"]  & \widetilde{X}^{\rm ab}\arrow[d, "\varphi"] \arrow[r] & X \arrow[dd, "g"] \\
					Z \arrow[r, "h"'] & W \arrow[d] & \\
					& \widetilde{A} \arrow[r] & A
				\end{tikzcd}
			\end{equation*}
			Let $G := \mathrm{Aut}(\widetilde{X}/\widetilde{X}^{\rm ab})$ be the Galois group of the covering $\pi$.  By \eqref{eq:imagetrivial} again, the preimage $\pi^{-1}(\varphi^{-1}(y))$ is exactly the $G$-orbit of a single connected component in $\widetilde{X}$, and thus the group $G$ acts transitively and freely on the set of connected components of the fibers of $\Phi$. Consequently, $G$ induces a free action on the complex space $Z$. Furthermore, this action is properly discontinuous (see, e.g., \cite[Lemma 3.34]{DY23}). This implies that the natural action of $G$ on $Z$ is free and properly discontinuous. 
			
			By construction, the quotient $Z/G$ is canonically isomorphic to $W$, meaning $h: Z \to W$ is a  topological covering map with Galois group $G$. Note that the map $W \to \widetilde{A}$ is finite and proper, and since $\widetilde{A} \cong \mathbb{C}^{\dim A}$ is Stein, $W$ is Stein. Since $h: Z \to W$ is a   topological covering of a Stein space, $Z$ is also Stein. 
			
			Finally, since $ {\psi}: \widetilde{X} \to Z$ is a proper holomorphic map to the Stein space $Z$, it follows from the Cartan-Remmert theorem that $\widetilde{X}$ is holomorphically convex.  The theorem is proved.
		\end{proof}

		\section{Nipotent local systems with two-step nilpotent monodromy} 
		In this section, we will prove   \cref{main:2step}.  The proof combines \cref{thm:strong} with ideas developed in     \cite[Proof of Theorem 4.2]{CDY25b} and ideas by Campana \cite[\S 4]{Cam01}.
		\begin{dfn}\label{def:pi1exact}
			Let $X$ be a smooth quasi-projective variety. We say quasi-Albanese map $a:X\to A$ is \emph{$\pi_1$-exact}, if it is dominant morphism with connected general fibers, and we have the following short exact sequence
			\begin{align}
				\pi_1(F)\to \pi_1(X)\to\pi_1(A)\to 0. 
			\end{align}
		\end{dfn}

		\begin{thm}[=\cref{main:2step}] \label{thm:main}
			Let $X$ be a smooth quasi-projective variety whose quasi-Albanese map is $\pi_1$-exact. Let $H$ be a connected nilpotent linear algebraic group over $\bC$.  Assume that there exists a Zariski dense representation $\varrho:\pi_1(X)\to H(\bC)$. Then   the following short exact sequence
			$$
			1\to H'\to H\to H/H'\to 1
			$$
			is a \emph{central extension} of two commutative algebraic group, where $H':=[H,H]$.  In particular, the image $\varrho(\pi_1(X))$ is    2-step nilpotent. 
		\end{thm}
		\begin{proof}
			\noindent {\em Step 1: we reduce the problem to the unipotent representation.} 
			
			Since $H$ is nilpotent and connected, we have an isomorphism
			$$
			H\simeq G\times T
			$$
			where $T$ is an algebraic tori and $G:=R_u(H)$ is the unipotent radical of $H$. Hence $G$ is   a unipotent algebraic group.  
			Via this isomorphism, $\varrho$ induces a natural decomposition of the representation 
			$$
			\varrho = (\tau , \varphi) : \pi_1(X) \to T(\bC) \times G(\bC),
			$$
			such that $\tau(\pi_1(X))$ is abelian, and $\varphi(\pi_1(X))$ is Zariski dense. It then  suffices to show that $\varphi(\pi_1(X))$ is nilpotent of class $2$.  
			
			\medskip
			
			\noindent {\em Step 2: we prove that $\varphi(\pi_1(X))$ is metabelian.}

			By  the assumption,  the quasi-Albanese map $a:X\to A$ of $X$ is dominant with connected general fibers $F$ and   we have the exact sequence
			$$
			\pi_1(F) \to \pi_1(X) \to \pi_1(A) \to 1. 
			$$  It induces 
			a natural surjective homomorphism
			$$
			\pi_1(A)\twoheadrightarrow \varphi(\pi_1(X))/\varphi({\rm Im}[ \pi_1(F)\to \pi_1(X)]),
			$$
			and thus  $\varphi(\pi_1(X))/\varphi({\rm Im}[ \pi_1(F)\to \pi_1(X)])$ is abelian.  
			Since $\varphi(\pi_1(X))$ is unipotent,  it follows from \cref{claim:samecon}   that $\varphi$ lies in the irreducible component of the representation variety $R(X, G)(\bC)$ that contains the trivial representation.   Note that conditions in  \cref{thm:strong} are fulfilled for $F\to X$. Hence  the image 
			$$
			\varphi({\rm Im}[ \pi_1(F)\to \pi_1(X)])
			$$
			is abelian.   We then have the short exact sequence
			\begin{align}\label{eq:specialness}
				1\to \varphi({\rm Im}[ \pi_1(F)\to \pi_1(X)])\to  \varphi(\pi_1(X)) \to \varphi(\pi_1(X))/\varphi({\rm Im}[ \pi_1(F)\to \pi_1(X)]) \to 1,
			\end{align} 
			and thus $\varphi(\pi_1(X))$ is    an extension of two abelian groups.   
			
			\medskip
			
			\noindent {\em Step 3: we reduce the problem to the triviality of a suitable local system $\bV_\sigma$.}
			
			We shall use similar ideas in \cite[Proof of Theorem 4.2]{CDY25b} to prove that \eqref{eq:specialness} is moreover a central extension for $\varphi(\pi_1(X))$. We retain the notation introduced there for the sake of reference and to facilitate the comparison.     Let $G':=[G,G]$. 
			We denote by
			$$
			\varphi_1:\pi_1(X)\to G/G'
			$$
			the composition of $\varphi$ with the quotient $G\to G/G'$.  Here for notational simplicity, whenever $G$ is regarded as the target of the representation, we write $G$ for its group of complex points $G(\bC)$. 
			
			\begin{claim}
				We have $\varphi(\pi_1(X)')\subset G'$, which is Zariski dense in $G'$. 
			\end{claim}
			\begin{proof}
				Set $\Gamma=\varphi(\pi_1(X))$.  Let $\Gamma':=[\Gamma,\Gamma]$.  Let $U\subset G'$ be the Zariski closure of $\Gamma'$.  Note that $\Gamma' \subset G'$ .
				We first show that $\Gamma'$ is Zariski dense in $G'$.
				Let $U\subset G'$ be the Zariski closure of $\Gamma'$.
				Then we have $\Gamma/\Gamma'\to G/U$, whose image is Zariski dense.
				Since $\Gamma/\Gamma'$ is commutative, $G/U$ is commutative.
				Hence $G'\subset U$.
				This shows $U=G'$.
				Thus $\Gamma'\subset G'$ is Zariski dense.  
			\end{proof}

			Set $\Phi:={\rm Im}[\pi_1(F)\to \pi_1(X)]$. Since $a$ is $\pi_1$-exact, we have  $\pi_1(X)'\subset \Phi$.   Let $U$ be the Zariski closure for $\varphi(\Phi)$, that is commutative since $\varphi(\Phi)$ is abelian by \eqref{eq:specialness}.  
			Hence we have  $G'\subset U$.  This also implies that $G'$ is also commutative.

			Since $G/G'$ is commutative, it follows that $\varphi_1$ factors through $\pi_1(X)^{ab}$.  Note that $\pi_1(A)=\pi_1(X)^{ab}/{\rm torsion}$. Since we have an inclusion
			$
			\pi_1(X)'\subset  \Phi,
			$ it follows that   $
			\varphi_1(\Phi)
			$ 
			is a finite group, and thus trivial,  as   $G/G'$ is also unipotent.   Thus we have 
			$ \varphi(\Phi)\subset  G' $. 
			
			By $\pi_1(X)'\subset \Phi$,  we have $$\varphi(\pi_1(X)')\subset \varphi(\Phi)\subset G'.$$
			This implies that $\varphi(\Pi)$ is Zariski dense in $G'$.

			In summary, we have short exact sequences
			\begin{equation}\label{eq:twoshort}
				\begin{tikzcd}
					1\arrow[r] &\Phi \arrow[r] \arrow[d,"\varphi"] & \pi_1(X)  \arrow[r] \arrow[d,"\varphi"] &\pi_1(A)\arrow[d]\arrow[r] &1\\
					1\arrow[r] &G' \arrow[r]  & G \arrow[r]   &G/G'\arrow[r] &1
				\end{tikzcd} 
			\end{equation} 
			in which each vertical morphism has Zariski dense image. 
			Recall that $G' $ is a commutative unipotent group. It implies that $G' \simeq (\mathbb G_a)^n$, where $\mathbb G_a$ is the additive group. 
			The exponential map $\mathrm{Lie}(G' )\to G' $ is an isomorphism. 
			Note that $(\mathbb G_a)^n=\mathbb C^n$ as additive group.
			We have 
			$$\mathrm{Aut}((\mathbb G_a)^n)=\mathrm{GL}(\mathrm{Lie}(G'))=\mathrm{GL}(\mathbb C^n).
			$$
			Hence by conjugation, we have 
			\begin{align}\label{eq:conjugateaction}
				\mu:G/G'&\to \mathrm{Aut}(G')=\mathrm{GL}({\rm Lie}(G'))\\\label{eq:adjoint}
				[g] &\mapsto (k\mapsto {\rm Ad}_g(k))
			\end{align}
			where $g\in G$ is a representative in $[g]\in G/G'$ and $k\in {\rm Lie}(G')$.  This definition does not depend on the choice for $g\in [g]$. 
			
			Let $\sigma:\pi_1(A)\to \mathrm{GL}({\rm Lie}(G'))$ be the composition of $\mu$ and $\pi_1(A)\to G/G'$ and let $\bV_\sigma$ be the corresponding local system. 
			In the following, we are going to prove that $\mu$ is trivial, which is equivalent  to that  $G$ is 2-step nilpotent. Since $\pi_1(A)\to G/G'$ is Zariski dense, this is equivalent to that $\bV_\sigma$ is a trivial local system.  To prove this result we shall relate $V_\sigma$ with the geometric variation of mixed Hodge structure  induced by $a$.

			\medskip
			
			\noindent {\em Step 4: we prove that $\bV_\sigma$ is a trivial local system.}
			
			Let $A^o\subset A$ be a non-empty Zariski open set such that the restriction $a_o:X^o\to A^o$ of $X\to A$ over $A^o$ is a locally trivial $C^\infty$-fibration.  
			We fix a base point $y\in A^o$ and let $F:=a^{-1}(y)$, that is a smooth quasi-projective variety.   
			
			
			Note that the morphism $\pi_1(F)\to \pi_1(X)\stackrel{\varphi}{\to}G $ factors through $H_1(F,\bZ)/{\rm tor}\to G'$ whose image is Zariski dense. If we take the composition with the logarithm of $G'$, then we get   a surjective natural linear map
			\begin{align} \label{eq:quotient}
				g:H_1(F,\bC)\twoheadrightarrow  {\rm Lie}(G')=:\kg'.
			\end{align} 
			Since $F$ is smooth, by Deligne's mixed Hodge theory (see \cite{Del71} or \cite[\S 8.4]{Voi02}),  $H_1(F,\bC)$ carries  the weight filtration
			$$
			0=W_{-3}H_1(F,\bC)\subset W_{-2}H_1(F,\bC)\subset W_{-1}H_1(F,\bC)=H_1(F,\bC)
			$$
			such that ${\rm Gr}_{-1}^WH_1(F,\bC)$ is a    pure polarized Hodge structure   of  weight \(-1\), with Hodge types \((-1,0)\) and \((0,-1)\),
			and ${\rm Gr}_{-2}^WH_1(F,\bC)=W_{-2}H_1(F,\bC)$ is a pure polarized Hodge structure of  weight \(-2\), with Hodge types $(-1,-1)$. We denote by $$\kl:=g(W_{-2}H_1(F,\bC)).$$
			Then $g$ induces a surjective linear map
			\begin{align} \label{eq:quotient2}
				g_{-1}:{\rm Gr}_{-1}^WH_1(F,\bC)\twoheadrightarrow \kg'/\kl.
			\end{align}

			Since $a_o:X^o\to A^o$ is a locally trivial $C^\infty$-fibration, by the work of Deligne,  the direct image   $\bV:=R_1(a_o)_*\bQ_{X^o}$   is a local system on $A^o$ that underlies a geometric variation of mixed Hodge structure.  It has the weight filtration
			$$
			0=W_{-3}\bV\subset W_{-2}\bV\subset W_{-1}\bV=\bV
			$$
			such that $\{{\rm Gr}_{-i}^W\bV\}_{i=-2,-1}$ are both polarized $\bV$-variation of Hodge structure ($\bZ$-VHS for short). By the theorem of Schmid, it follows that their underlying local systems are both semisimple.   For the linear map $g$ defined in   \eqref{eq:quotient},   one can see  that it is equivariant with respect to the action of $\pi_1(A^o)$, and thus we have a surjective morphism between local systems
			$$
			q:\bV\twoheadrightarrow \bV_\sigma|_{A^o} 
			$$
			such that at the fiber over $y$, it is the surjective linear map \eqref{eq:quotient}.  
			Note that $\sigma$ factors through $\mu$, and since $G/G'$ is unipotent, it follows that the algebraic monodromy group of $\bV_\sigma|_{A^o}$, and thus $\bV_\sigma$, is a   unipotent group $U_c$.   
			
			Consider the image $q(W_{-2}\bV)$. It is a sublocal system for $\bV_\sigma|_{A^o}$.  Its algebraic monodromy group is a quotient of $U_c$, and is thus still   unipotent. However, since $W_{-2}\bV$ is a semisimple (moreover unitary) local system, it implies that its monodromy is finite and thus trivial as it is unipotent. 
			In particular, it implies that $\kl$ is an invariant subspace of $\sigma(\pi_1(X))$,  which moreover acts trivially on $\kl$. As $\pi_1(A)\to G/G'$ is Zariski dense, it follows that the adjoint action of $G/G'$ on $\kl$ is trivial.

			Note that the weight filtration $W_\bullet \bV$ induces a weight filtration on $W_\bullet \bV_\sigma|_{A^o}$ given by
			$$
			W_i \bV_\sigma|_{A^o}:=q(W_i \bV),
			$$
			which induces surjective morphisms
			$$
			q_i:{\rm Gr}_i^W\bV\twoheadrightarrow {\rm Gr}_i^W\bV_\sigma|_{A^o}. 
			$$
			such that at the fiber over $y$, it is the surjective linear map \eqref{eq:quotient2} when $i=-1$. 
			Consider $q_{-1}({\rm Gr}_{-1}^W\bV)$, that is a quotient of the semisimple local system, and is thus also semisimple. On the other hand, its algebraic monodromy group is  isomorphic to a quotient of $U_c$, and is thus unipotent. By the same reason, it is also trivial.  Therefore, the adjoint action of $G/G'$ on $\kg'/\kl$
			is trivial.   In conclusion, $\{{\rm Gr}_{-i}^W\bV_\sigma|_{A^o}\}_{i=1,2} $ are both trivial local systems.

			Hence, for any $\lambda\in (\kg'/\kl)^*$,  there exists a unique global flat section $e_\lambda\in H^0(A^o, ({\rm Gr}_{-1}^W\bV_\sigma|_{A^o})^*)$ such that $e_\lambda(y)=\lambda$. Note that we have an injective morphism 
			$$
			({\rm Gr}_{-1}^W\bV_\sigma|_{A^o})^*\stackrel{q_{-1}^*}{\hookrightarrow} ({\rm Gr}_{-1}^W)\bV^*=  W_1\bV ^* =W_1 R^1(a_o)_*\bC.
			$$
			Let $e'_{\lambda}\in  H^0(A^o,W_1R^1(a_o)_*\bC)$ be its image.  
			\begin{claim}\label{claim:surjectivemap}
				The following composite map
				\begin{align*}
					\Psi:   W_1H^1(X,\bC)\to W_1H^1(X^o,\bC)\to H^0(A^o,W_1 R^1(a_o)_*\bC) 
				\end{align*}
				is surjective.   
			\end{claim}
			\begin{proof}
				This follows similar ideas by \emph{invariant cycle theorem} by Deligne (see  e.g. \cite[Th\'eor\`eme 16.18]{Voi02}). We remark that similar ideas have already used by Campana \cite[Lemme 4.6]{Cam01}.   However, since the fibration $a_o:X^o\to A^o$ is not proper and the invariant cycle theorem in general fails in this situation, we cannot directly apply invariant cycle theorem to conclude the claim. Nevertheless, we will use similar ideas to prove the claim. 
				
				By the Leray spectral sequence (see \cite[\S 16.1.3]{Voi02}), we have
				a short exact sequence
				$$
				H^1(X^o,\bC)\stackrel{h}{\to} H^0(A^o,R^1(a_o)_*\bC)\stackrel{d_2}{\to} H^2(A^o,\bC)
				$$
				such that  both $h$ and $d_2$ are morphisms of mixed Hodge structures (see also \cite[\S 16]{Voi02}).  Since we have
				$$
				H^0(A^o,W_1R^1(a_o)_*\bC)=W_1H^0(A^o, R^1(a_o)_*\bC),
				$$
				and   $h$ is strict for the weight filtration, it follows that  for any $e\in H^0(A^o,W_1R^1(a_o)_*\bC)$,  $d_2(e)\in W_1H^2(A^o,\bC)$. However, since $A^o$ is smooth, it follows that $W_1H^2(A^o,\bC)=0$, and thus $d_2(e))=0$. Therefore, we have a surjection
				$$
				W_1H^1(X^o,\bC)\twoheadrightarrow H^0(A^o,W_1R^1(a_o)_*\bC). 
				$$
				On the other hand, by Deligne's mixed Hodge structure again, we note that the inclusion $i:X^o\hookrightarrow X$ gives rise to a natural isomorphism
				$$
				i^*:W_1H^1(X,\bC)\stackrel{\simeq}{\to} W_1H^1(X^o,\bC) 
				$$
				which is moreover isomorphic to the $H^1$ of some (thus any) smooth projective compactification of $X$.  Together with the above surjective map, this yields the claim. 
			\end{proof}
			Therefore,  by \cref{claim:surjectivemap}, there exists some (non-unique) $ {\alpha}_\lambda\in W_1H^1(X,\bC)$ such that we have 
			$$\Psi(\alpha_\lambda)= e'_{\lambda}.$$
			We shall consider  $\alpha_\lambda$ as  a morphism
			$
			\pi_1(X) \to \bC$ defined by  $$\alpha_\lambda(\gamma):=\int_\gamma\alpha_\lambda.$$
			
			\begin{claim}\label{claim:lastone}
				For any $\gamma\in \Phi:={\rm Im}[\pi_1(F)\to \pi_1(X)]$ and  any $\lambda\in (\kg'/\kl)^*$, we have 
				$$
				\lambda([\log \varphi(\gamma)])= \alpha_\lambda(\gamma). 
				$$ 
				where $\iota:F\hookrightarrow X$ denotes the natural inclusion map, and 
				$[\log \varphi(\gamma)]$ denotes the image of  $[\log \varphi(\gamma)]$ under the quotient map $\kg'\to \kg'/\kl$. 
			\end{claim}
			\begin{proof}
				By the construction above, we note that
				$$
				\iota^*(\alpha_\lambda))=(q_{-1}^*e_\lambda)_y\in H^1(F,\bC). 
				$$
				Note that  for any $c\in \pi_1(F)$, we have 
				$$
				(q_{-1}^*e_\lambda)_y([c])= \lambda([g([c])])=\lambda([\log\varphi(\iota_*(c))])
				$$
				according to the definition of $g$ in \eqref{eq:quotient}. Here $[c]$ is the image of $c$ under the map $\pi_1(F)\to H_1(F,\bC)$. 
				The above two equalities gives rise to the equality in the claim. 
			\end{proof}
			This claim enables us to conclude the following result. 
			\begin{claim}
				We have
				$
				\kl=\kg'.
				$ 
			\end{claim}
			\begin{proof}
				Pick any $\lambda\in (\kg'/\kl)^*$.  By \cref{claim:lastone},  for any $\gamma_1,\gamma_2\in \pi_1(X)$,  we have  
				$$\lambda([\log (\varphi([\gamma_1,\gamma_2])])=  \alpha_\lambda([\gamma_1,\gamma_2]) =0$$
				where the second equality is due to $\alpha_\lambda|_{\pi_1(X)'}\equiv 0$. 
				Recall that $\pi_1(X)'\subset \Phi$, and $\varphi(\pi_1(X)')$   is Zariski dense in $G'$. Hence $\log \varphi(\pi_1(X)') $  is Zariski dense in $\kg'$.  This implies that 
				$\lambda=0$. The claim is proved.  
			\end{proof}
			\cref{claim:lastone} implies that  $\bV_\sigma|_{A^\circ}=q (W_{-2}\bV).$  Then by the above arguments,  $\bV_\sigma|_{A^\circ}$, and thus $\bV_\sigma$, is a trivial local system. 
			
			\medskip
			
			\noindent{\em Step 5: completion of the proof}
			
			The triviality of $\bV_\sigma$ yields that $\mu$ defined in \eqref{eq:conjugateaction} is trivial. Hence the exact  sequence 
			$$
			1\to G'\to G\to G/G'\to 1
			$$
			is a  central extension  of commutative unipotent groups. Via the above isomorphism $H\simeq T\times G$,   we have the short exact sequence
			$$
			1\to G'\to  H\to T\times G/G'\to 1.
			$$
			Since $T$ is an algebraic torus, we have $G'\subset Z(G)=Z(H)$. 
			Therefore, $H$ is 2-step nilpotent algebraic group, and   the image  $\varrho(\pi_1(X))$ is 2-step nilpotent. The theorem is proved. 
		\end{proof}

	\medspace
	
	\noindent \textbf{Acknowledgements.}
	YD would like to thank Michel Brion, Patrick Brosnan, Frédéric Campana, and Jon  Pridham for very fruitful discussions and for answering his questions; Beno\^it Claudon, Philippe Eyssidieux and Jon  Pridham for their interest  and remarks on the paper; and Carlos Simpson and Katsutoshi Yamanoi for very stimulating discussions during the initial stage of this work.  We thank Joaquín Moraga for his remarks.  JC and YD   are partially supported by the ANR grant Karmapolis (ANR-21-CE40-0010).  CH was partially supported by  NSF research grant   DMS-2301374 and by a grant from the Simons Foundation SFI-MPS-MOV-00006719-07. MP gratefully acknowledges support from \emph{Deutsche
		Forschungsgemeinschaft}.
	
		\medspace
	
 \noindent \textbf{Statement on LLM/AI use}. All mathematical results and proofs  are done by the authors without LLM/AI.
	
	\providecommand{\bysame}{\leavevmode\hbox to3em{\hrulefill}\thinspace}
	\providecommand{\MR}{\relax\ifhmode\unskip\space\fi MR }
	\providecommand{\MRhref}[2]{%
		\href{http://www.ams.org/mathscinet-getitem?mr=#1}{#2}
	}
	\providecommand{\href}[2]{#2}

\end{document}